\documentclass[12pt]{article}
\usepackage[utf8]{inputenc}
\usepackage{tikz}
\usepackage{amsthm}
\usepackage{amsmath}
\usepackage{amssymb}
\usepackage{amsthm}
\usepackage{amscd}
\usepackage[all]{xy}
\usepackage[margin=1in]{geometry}
\usepackage{comment}
\usepackage{tikz-cd}
\usepackage{pdfpages}
\usepackage{url}

\usepackage{amscd}
\usepackage{graphicx}
\usepackage{amsmath,amsfonts,amssymb,enumerate}
\usepackage{amsthm} %Serve per mettere i teoremi
\usepackage[english]{babel}
\newcommand{\ec}{\color{black}} %% COLORE NERO

\newtheorem{thm}{Theorem}[section]
\newtheorem{lem}[thm]{Lemma}
\newtheorem{prop}[thm]{Proposition}
\newtheorem{cor}[thm]{Corollary}

\newtheorem{conjecture}[thm]{Conjecture}

\theoremstyle{definition}

\newtheorem{remark}[thm]{Remark}

\numberwithin{equation}{section}

\newcommand{\FF}{\mathbb{F}}
\newcommand{\A}{\mathbb{A}}
\newcommand{\BB}{\mathbb{B}}
\newcommand{\DD}{\mathbb{D}}

\newcommand{\m}{\mathfrak{m}}

\def\ker{\mbox{\rm ker}}

\def\rank{\mbox{\rm rank}}

\def\Spec{\mbox{\rm Spec}}
\def\dim{\mbox{\rm dim}}

\newcommand{\mf}{\mathfrak}
\newcommand{\mb}{\mathbb}

\begin{document}

\title{Higher structure maps for free resolutions of length 3 and linkage}

%\author{ Lorenzo Guerrieri, Xianglong Ni, Jerzy Weyman
%\thanks{... \textbf{Email address:} lorenzo.guerrieri@uj.edu.pl }} 

\author{Lorenzo Guerrieri \thanks{ Jagiellonian University, Instytut Matematyki, Krak\'{o}w \textbf{Email address:} lorenzo.guerrieri@uj.edu.pl 
} 
\and Xianglong Ni
\thanks{UC Berkeley, Department of Mathematics \textbf{Email address:} xlni@berkeley.edu } 
\and Jerzy Weyman
\thanks{Jagiellonian University, Instytut Matematyki, Krak\'{o}w \textbf{Email address:} jerzy.weyman@uj.edu.pl }} 
\maketitle
\begin{abstract}
\noindent Let $I$ be a perfect ideal of height 3 in a Gorenstein local ring $R$. Let $\FF$ be the minimal free resolution of $I$. A sequence of linear maps, which generalize the multiplicative structure of $\FF$, can be defined using the generic ring associated to the format of $\FF$. Let $J$ be an ideal linked to $I$.
We provide formulas to compute some of these maps for the free resolution of $J$ in terms of those of the free resolution of $I$. %This generalizes the work of Avramov, Kustin, and Miller describing the multiplicative structure of the free resolution of a linked ideal. 
We apply our results to describe classes of licci ideals, showing that a perfect ideal with Betti numbers $(1,5,6,2)$ is licci if and only if at least one of these maps is nonzero modulo the maximal ideal of $R$. \\ 
%\medskip
\end{abstract}

%\footnotemark[3]
%\thanks{Universit\`a di Catania, Dipartimento di Matematica e Informatica, Viale A. Doria, 6, 95125 Catania, Italy.   \textbf{Email address:} guerrieri@dmi.unict.it  } { }

%\maketitle

%\begin{abstract}
%\noindent 

%\medskip

%\noindent MSC: 13F20;  05C25. \\
%\noindent Keywords: Monomial ideals, Cover ideals of graphs, Freiman ideals, fiber cone.
%\end{abstract} X_{11} &

\section{Introduction}
%\linenumbers

In their landmark paper \cite{Peskine-Szpiro-linkage}, Peskine and Szpiro laid the modern algebraic foundations of the theory of linkage---a concept which had existed in some form since the work of Macaulay \cite{Macaulay}. They also show that, in codimension two, an ideal $I$ in a Gorenstein local ring $R$ is in the linkage class of a complete intersection (\emph{licci}) if and only if it is perfect---i.e. $R/I$ is a Cohen-Macaulay ring. (Ap\'ery \cite{Apery1},\cite{Apery2} and Gaeta \cite{Gaeta} had previously shown this for curves in $\mb{P}^3$.) % and later generalized by Hartshorne and Rao [?] to other equivalence classes. In brief, the failure of $S/I$ to be Cohen-Macaulay is measured by $\operatorname{Ext}^3(S/I,S)$ and this is a complete invariant for linkage in codimension two.

Without the codimension two assumption, only the forward implication holds. Our focus will entirely be on perfect ideals of codimension three, and there are simple examples of such ideals which are \emph{not} licci, e.g. $(x,y,z)^2 \subset \mb{C}[x,y,z]_{(x,y,z)}$.

But there are some positive results: Watanabe showed that Gorenstein ideals of codimension three are licci \cite{Watanabe}. Since almost complete intersections are linked to Gorenstein ideals, they are licci as well. The following conjecture from \cite{CVWdynkin} extends this to a few other families:

\begin{conjecture}\label{licciconjecture}
	Let $I$ be a perfect ideal of codimension three in a Gorenstein local ring $S$ with residue field $k$. Let $r_i$ denote the Betti number $\rank \operatorname{Tor}_i(R/I, k)$. Suppose that $(1,r_1,r_2,r_3)$ is either $(1,n,n,1)$ for some $n$, $(1,4,n,n-3)$ for some $n$, $(1,5,6,2)$, $(1,6,7,2)$, $(1,5,7,3)$, $(1,7,8,2)$, or $(1,5,8,4)$.	Then $I$ is licci.
\end{conjecture}
The first two families in this list are the Gorenstein ideals and almost complete intersections respectively. The remaining five are more mysterious, and are explained by a deep connection to the ADE classification. It is also shown in \cite{CVWdynkin} that this conjecture is ``tight'' in the sense that there exists a perfect but not licci ideal having each sequence of Betti numbers not listed above.

Our study of linkage will be from the vantage point of free resolutions, revolving around the following useful lemma from \cite{Peskine-Szpiro-linkage}: if $\mb{A}$ resolves $R/I$ and $\mb{B}$ is a Koszul complex resolving $R/K$ where $K\subset I$ is generated by a regular sequence of maximal length, then a (non-minimal) resolution of $R/(K:I)$ can be obtained as the dual of the mapping cone of $\mb{B} \to \mb{A}$ extending $R/K \to R/I$.

The resolution $\mb{A}$ has the (non-unique) structure of a graded-commutative DGA. After choosing such a structure, there is a unique map of DGAs $\mb{B} \to \mb{A}$ covering $R/K \to R/I$. Hence the differentials in the resolution of $R/(K:I)$ can be understood in terms of the differentials and multiplicative structure of $\mb{A}$. In particular, one can show that the multiplication $\bigwedge^2 A_1 \to A_2$ must be nonzero mod $\mf{m}$ (i.e. contain units) in order for the total Betti number of $R/(K:I)$ to be lower than that of $R/I$. If $I$ is licci, then such reductions in total Betti number must happen eventually, and this observation tells us when it happens after the first link.

On the other hand, consider a Gorenstein ideal $I \subset R$, with Betti numbers $(1,n,n,1)$. We know $I$ to be licci from Watanabe's work, but the first minimal link does not yield a reduction in total Betti number: the linked ideal $J$ is an almost complete intersection with Betti numbers $(1,4,n,n-3)$. It is in the next link that a drop may occur, as $J$ in turn can be linked to a Gorenstein ideal on $(n-2)$ generators. Through the lens of the preceding discussion, this means that the resolution $\mb{D}$ of $R/J$ had units in the multiplication $\bigwedge^2 D_1 \to D_2$, although the resolution $\mb{A}$ of $R/I$ did not. The natural question to pose is whether units in $\bigwedge^2 D_1 \to D_2$ are heralded by some other structure on the original resolution $\mb{A}$ of $R/I$.

As Avramov, Kustin, and Miller analyzed the multiplicative structure on $\operatorname{Tor}_*(R/I,k)$ in \cite{AKM88}, they showed how the multiplication on $\mb{D}$ can be described in terms of structure maps computed from the original resolution $\mb{A}$, thereby answering the preceding. For this it was necessary to introduce two new maps which they call $X$ and $Y$.

The bulk of this paper is dedicated to going one step further: to show how the maps $X,Y$ on $\mb{D}$ can be related to additional structure maps computed from the original resolution $\mb{A}$. To achieve this, we first show how all the maps discussed above are merely the first few higher structure maps $w^{(i)}_j$ coming from Weyman's generic ring \cite{W18}.

Towards explaining these notions, define the \emph{format} of a free resolution \begin{equation}
\label{complexF}
\FF: 0 \longrightarrow R^{r_m} \buildrel{d_m}\over\longrightarrow  R^{r_{m-1}} \longrightarrow \cdots  \longrightarrow R^{r_1} \buildrel{d_1}\over\longrightarrow R^{r_0}
\end{equation}
to be the sequence $(r_0,r_1,\ldots,r_m)$. Resolutions with format $(1,n,n-1)$ are characterized by the Hilbert-Burch theorem: the differential $d_1$ is comprised of the $(n-1)\times (n-1)$ minors of $d_2$, multiplied by a nonzerodivisor. An alternative way of stating the theorem is as follows. Let $R_{univ}$ be the polynomial ring on variables $\{x_{ij}\}_{1 \leq i \leq n, 1 \leq j \leq n-1}$ and an additional variable $u$. Let $\mb{F}_{univ}$ be the free resolution
\[
\FF\colon 0 \to R_{univ}^{n-1} \overset{d_2}{\longrightarrow} R_{univ}^n \overset{d_1}{\longrightarrow} R_{univ}
\]
where $d_2$ is the generic matrix with entries $x_{ij}$, and the $i$th entry of $d_1$ is $(-1)^i u M_i$ where $M_i$ is the minor of $d_2$ excluding the $i$th row. Then the pair $(R_{univ},\mb{F}^{univ})$ is \emph{universal} for free resolutions of format $(1,n,n-1)$ in the sense that if $\mb{F}$ is such a resolution over some ring $R$, there exists a unique ring homomorphism $R_{univ} \to R$ so that $\mb{F} = \mb{F}^{univ} \otimes R$.

The idea of using universal resolutions to understand the structure theory of free resolutions was adopted by Hochster in \cite{H75}, who also found $(R_{univ},\mb{F}^{univ})$ for formats of length two where $r_0 > 1$. However, for formats of length three and beyond, Bruns \cite{Br84} showed that $(R_{univ}, \mb{F}^{univ})$ does not exist. The issue lies with the requirement that the map $R_{univ} \to R$ be unique for each resolution $\mb{F}$. If we drop this uniqueness requirement, then we get the weaker notion of a generic pair $(R_{gen},\mb{F}^{gen})$, and Bruns showed that such objects always exist.

Although this settled the question of existence, one would like to understand the generic ring and resolution more explicitly, as that is what ultimately translates to concrete structure theorems about free resolutions. Over the complex numbers, this was done by Weyman for formats of length three in \cite{w89} and \cite{W18}. In \S2, we review how one can obtain structure maps $w^{(i)}_j$ for the resolution $\mb{F}$ via Weyman's construction.

In \S3 we show that the maps $X,Y$ from \cite{AKM88} can be reinterpreted in this framework as the structure maps $w^{(2)}_2, w^{(3)}_2$. Guided by this connection to the generic ring, the rest of \S3 is devoted to continuing the pattern one step further, and showing how the maps $w^{(i)}_2$ on $\mb{D}$ can be expressed in terms of various structure maps $w^{(i)}_j$ on the resolution $\mb{A}$. The proofs reduce to the verification of identities relating the higher structure maps $w^{(i)}_j$, which are deferred to \S5. 

In \S4, we apply the preceding results to linkage, in the setting of a local Gorenstein ring with infinite residue field. We also discuss the case $S/I$ has Betti numbers $(1,5,6,2)$ at length, which is one of the cases listed in Conjecture~\ref{licciconjecture}. In fact, those cases are exactly the ones for which only finitely many higher structure maps are nonzero. For $(1,5,6,2)$, there are few enough structure maps that we can describe them all. Although we do not prove Conjecture~\ref{licciconjecture} for $(1,5,6,2)$, we reduce it to the concrete question of whether these maps are nonzero mod $\mf{m}$ in Theorem 4.6. %fix hyperlink

The program outlined here and the theorems in \S3 strongly suggest a pattern which continues beyond the structure maps explicitly considered in this paper. In a sequel to this paper, we hope to extend the results of \S3 and \S4 in a way which circumvents the computational difficulties of working explicitly with higher structure maps.

\section{Preliminaries}

%We will work over a commutative Noetherian ring $R$ throughout, in which $2,3\in R$ are units. When discussing linkage in later sections, we will moreover assume that $R$ is local and Gorenstein, with infinite residue field.

\subsection{The generic ring}

For this subsection only, we will assume that $R$ is a $\mb{C}$-algebra. In this paper we will only consider formats $(1,r_1,r_2,r_3)$, i.e. ones arising for resolutions of cyclic modules. Fixing such a format, let $({\hat R}_{gen}, \mb{F}^{gen})$ denote Weyman's generic pair. Let $F_i = \mb{C}^{r_i}$, so that
\[
	\mb{F}^{gen} \colon 0 \to F_3 \otimes {\hat R}_{gen} \to F_2 \otimes {\hat R}_{gen} \to F_1 \otimes {\hat R}_{gen} \to \mb{C} \otimes {\hat R}_{gen}.
\]
The Lie algebra $\prod \mf{gl}(F_i)$ acts on ${\hat R}_{gen}$. Inside the generic ring are the representations $F_2^* \otimes F_3$, $F_2 \otimes F_1^*$, and $F_1$: the $\mb{C}$-linear spans of the entries of $d_3, d_2, d_1$ respectively. 

We do not go into the details here, but there is a graph $T_{p,q,r}$ (depending on the format) and an associated Kac-Moody Lie algebra $\mf{g}(T_{p,q,r})$ with $\mf{g} = \mf{gl}(F_2) \times \mf{g}(T_{p,q,r})$ acting on ${\hat R}_{gen}$. Each differential $d_i$ generates a representation $W(d_i)$ of $\mf{g}$ inside of ${\hat R}_{gen}$. We call these the three \emph{critical representations}. Decomposing these representations with respect to the grading induced by a certain root of $\mf{g}(T_{p,q,r})$, one finds
\begin{gather*}
	W(d_3)=F^*_2 \otimes [F_3 \oplus \bigwedge^2 F_1 \oplus \bigwedge^4 F_1 \otimes F_3^* \oplus \cdots]\\
	W(d_2)=F_2 \otimes [F_1^* \oplus F_3^*\otimes F_1 \oplus \bigwedge^3 F_1 \otimes \bigwedge^2 F_3^* \oplus \cdots]\\
	W(d_1)={\mathbb{C}} \otimes [F_1 \oplus F_3^* \otimes \bigwedge^3 F_1 \oplus \cdots] 
\end{gather*}
In particular, the differentials reside in the bottom graded components.

Given a resolution $\mb{F}$ over a $\mb{C}$-algebra $R$, with format $(1,r_1,r_2,r_3)$ as fixed before, the genericity of $({\hat R}_{gen}, \mb{F}^{gen})$ yields a (non-unique) homomorphism $w\colon {\hat R}_{gen} \to R$ for which we have $\mb{F} = \mb{F}^{gen} \otimes R$. Let $w^{(i)}_j$ denote the restriction of $w$ to the $j$th graded piece of $W(d_i) \subset {\hat R}_{gen}$, where the bottom piece is $j=0$. For example, $w^{(3)}_0$ is a $\mb{C}$-linear map $F_2^* \otimes F_3 \to R$, i.e. an $R$-linear map $F_3 \otimes R \to F_2 \otimes R$, which is exactly $d_3$ of the resolution $\mb{F}$. Likewise $w^{(2)}_0$ and $w^{(1)}_0$ give $d_2$ and $d_1$.

For brevity, we will abuse notation and just write $F_i$ for $F_i \otimes R$ when the meaning is clear from context. When $j=1$, we obtain maps $w^{(3)}_1\colon \bigwedge^2 F_1 \to F_2$, $w^{(2)}_1\colon F_1 \otimes F_2 \to F_3$, and $w^{(1)}_1\colon \bigwedge^3 F_1 \to F_3$. By analyzing the relations in ${\hat R}_{gen}$, one can show that these maps endow $\mb{F}$ with the structure of a commutative differential graded algebra. Explicitly, writing $\lbrace e_1, \ldots, e_{r_1}\rbrace$, $\lbrace f_1, \ldots, f_{r_2} \rbrace$, $\lbrace g_1, \ldots, g_{r_3}\rbrace$ for the bases of $F_1, F_2, F_3$ respectively,
\begin{gather}\label{multiplicationrelations}
	\begin{gathered}
	d_2(e_i^.e_j)= d_1(e_i)e_j- d_1(e_j)e_i \qquad  d_3(e_i^.f_h)= d_1(e_i)f_h- e_i^.d_2(f_h) \\  
	 d_3(e_i^.e_j^.e_k)= d_1(e_i)e_j^.e_k - d_1(e_j)e_i^.e_k + d_1(e_k)e_i^.e_j 
	\end{gathered}
\end{gather}
where $e_i^. e_j = w^{(3)}_1(e_i\wedge e_j)$, $e_i^. f_h = w^{(2)}_1(e_i \otimes f_h)$, and $e_i^. e_j^. e_k = w^{(1)}_1(e_i \wedge e_j \wedge e_k)$. This multiplicative structure has been well-known since the famous Buchsbaum-Eisenbud papers \cite{BE74}, \cite{BE77}.

For most formats $(1,r_1,r_2,r_3)$, the Lie algebra $\mf{g}$ is infinite-dimensional as are the critical representations. Consequently, resolutions $\mb{F}$ of such formats have infinitely many higher structure maps $w^{(i)}_j$. The exceptions are the formats listed in Conjecture~\ref{licciconjecture}; in these cases the graph $T_{p,q,r}$ is a Dynkin diagram. Accordingly, we call these the \emph{Dynkin formats}.

The formats $(1,n,n,1)$ and $(1,4, n, n-3)$ are associated to $D_n$. Their critical representations are described in \cite{GW20} and all the structure maps are explicitly computed in the case where $\FF$ is a split exact complex, or the direct sum of a generic Hilbert-Burch complex of length 2 with a split exact complex. For the $E_6$ format $(1,5,6,2)$, which we investigate in \S4, the critical representations are:
$$W(d_3)=F^*_2 \otimes [F_3 \oplus \bigwedge^2 F_1 \oplus \bigwedge^4 F_1 \otimes F_3^* \oplus \bigwedge^5 F_1 \otimes F_1 \otimes \bigwedge^2 F_3^* ],$$
$$W(d_2)=F_2 \otimes [F_1^* \oplus F_3^*\otimes F_1 \oplus \bigwedge^3 F_1 \otimes \bigwedge^2 F_3^* \oplus \bigwedge^5 F_1 \otimes S_{2,1} F_3^*],$$
$$W(d_1)={\mathbb{C}} \otimes [F_1 \oplus F_3^* \otimes \bigwedge^3 F_1 \oplus (\bigwedge^2F_3^* \otimes \bigwedge^4 F_1 \otimes F_1 +  S_2F_3^* \otimes \bigwedge^5F_1 + \bigwedge^2F_3^* \otimes \bigwedge^5 F_1) \oplus  $$
$$ \oplus S_{2,1}F_3^* \otimes S_{2,2,1,1,1}F_1 \oplus S_{2,2}F_3^* \otimes S_{2,2,2,2,1}F_1].  $$
For the other Dynkin formats, including ones where $r_0 > 1$, tables describing the critical representations can be found in \cite{LW19}.

% In \cite{GW20}, the higher structure maps are computed by choosing generic liftings introducing sets of new variables over $R$, called defect variables. In some cases indeed the lift of a cycle is not unique and defect variables were used to add a generic term to a particular choice of a solution (in this paper we do this operation in the last section, where we work with a split exact complex). In the notation of \cite{GW20} the maps, computed generically using the defect variables, are denoted by $v^{(i)}_j$, where $i=1,2,3$ denotes the critical representation and $j$ denotes the graded component.

% Through this paper we call $w^{(i)}_{j,k}$ some chosen image of corresponding maps, computed over the ring $R$ without adding new variables.
%The index $k$ here denotes the fact that for formats different from $D_n$ some graded components involve more than one map. When this is not the case and there is only one map, we simply use the notation $w^{(i)}_{j}$. 
 
\subsection{Higher structure maps}
To avoid having to assume that $R$ is a $\mb{C}$-algebra, we will not define $w^{(i)}_j$ in terms of homomorphisms ${\hat R}_{gen} \to R$ as in \S2.1. Instead we explicitly define each structure map via a lift. Of course, the following definitions are motivated by the relations which hold in ${\hat R}_{gen}$, but we do not rely on any technical results pertaining to ${\hat R}_{gen}$ in what follows. Although there are infinitely many higher structure maps in general, here we will only consider those necessary to understand the behavior of $w^{(i)}_2$ under linkage. Their definitions are valid over any ring $R$ containing $1/2$ and $1/3$.

Let $\mb{F}$ be a resolution of format $(1,r_1,r_2,r_3)$ over a ring $R$. We illustrate the preceding for $w^{(i)}_1$. Rather than referencing ${\hat R}_{gen}$, we simply define $w^{(3)}_1$ to be a lift in the diagram
\[\begin{tikzcd}
0 \ar[r] & F_3 \ar[r] & F_2 \ar[r] & F_1 \ar[r] & R\\
& & \bigwedge^2 F_1 \ar[u, dashed, "w^{(3)}_1"] \ar[ur, "q^{(3)}_1",swap]
\end{tikzcd}\]
where $q^{(3)}_1(e_i \wedge e_j) = d_1(e_i)e_j - d_1(e_j)e_i$. This lift is not unique; if $M \colon \bigwedge^2 F_1 \to F_3$ is any map, then $w^{(3)}_1 + d_3 M$ is another valid lift.

Similarly we define $w^{(2)}_1$ and $w^{(1)}_1$ as lifts
\[\begin{tikzcd}
0 \ar[r] & F_3 \ar[r] & F_2 \ar[r] & F_1 \\
& F_1 \otimes F_2 \ar[u, dashed, "w^{(2)}_1"] \ar[ur, "q^{(2)}_1",swap]
\end{tikzcd}\]
\[\begin{tikzcd}
0 \ar[r] & F_3 \ar[r] & F_2 \ar[r] & F_1 \\
& \bigwedge^3 F_1\ar[u, dashed, "w^{(1)}_1"] \ar[ur, "q^{(1)}_1",swap]
\end{tikzcd}\]
where $q^{(2)}_1(e_i \otimes f_j) = d_1(e_i) f_j - w^{(3)}_1(e_i \wedge d_2(f_j))$ and $q^{(1)}_1 (e_i \wedge e_j \wedge e_k) = d_1(e_i)w^{(3)}_1(e_j \wedge e_k) - d_1(e_j)w^{(3)}_1(e_i \wedge e_k) + d_1(e_k) w^{(3)}_1(e_i \wedge e_j)$. With these definitions, the relations \eqref{multiplicationrelations} are satisfied.

%Analyzing the second graded components of the critical representations, we observe that $w^{(3)}_{2}$, $w^{(2)}_{3}$ are well-defined by lifting cycles over an appropriate complex as it happens for the multiplicative structure. We do not know whether a similar pattern holds for every structure map and every format (see Remark \ref{remarkcycles}).

We define $w^{(3)}_2: \bigwedge^4 F_1 \to F_3 \otimes F_2$ by lifting in the complex: %$$\begin{matrix} 0&\rightarrow&S_{2,2}F_3&\rightarrow&S_{2,1}F_3\otimes F_2&\rightarrow& \bigwedge^2 F_3\otimes S_2 F_2\oplus S_2 F_3\otimes\bigwedge^2  F_2\\ &&&\nwarrow\ v^{(1)}_4&\uparrow\ q&&\\&&&&\bigwedge^5 F_1\otimes \bigwedge^4F_1&&&&  \end{matrix}$$
\[\begin{tikzcd}
0 \ar[r] & \bigwedge^2 F_3 \ar[r] & F_3 \otimes F_2 \ar[r] & S_2 F_2 \ar[r] & S_2 F_1\\
&& \bigwedge^4 F_1 \ar[u,dashed,"w^{(3)}_2"] \ar[ur,"q^{(3)}_2",swap]
\end{tikzcd}\]
where 
$$ q^{(3)}_2(e_i\wedge e_j\wedge e_k\wedge e_l)= e_i^.e_j \otimes e_k^.e_l - e_i^.e_k \otimes e_j^.e_l + e_i^.e_l \otimes e_j^.e_k. $$
It is not hard to check that the composition of $q^{(3)}_2$ with the map $S_2 F_2 \to S_2F_1$ is zero. As was the case for $w^{(3)}_1$, the lift for $w^{(3)}_2$ is not unique.

Similarly, $w^{(2)}_2: \bigwedge^3 F_1 \otimes F_2 \to \bigwedge^2 F_3$ is the lift of $q^{(2)}_2$ in the diagram
\[\begin{tikzcd}
0 \ar[r] & \bigwedge^2 F_3 \ar[r] & F_3 \otimes F_2 \ar[r] & S_2 F_2\\
& \bigwedge^3 F_1 \otimes F_2 \ar[u,"w^{(2)}_2",dashed] \ar[ur, "q^{(2)}_2", swap]
\end{tikzcd}\]
where %the map $q_2^{(2)}$ is defined as the combination $ 
  %v^{(2)}_1 v^{(3)}_1   + v^{(3)}_2 d_2 - v^{(1)}_1 \otimes 1_{F_2}$.
%More precisely,
$$  q_2^{(2)}(e_i\wedge e_j\wedge e_k \otimes f_h) = e_i^.e_j \otimes e_k^.f_h - e_i^.e_k \otimes e_j^.f_h+ e_j^.e_k \otimes e_i^.f_h + $$
$$  -w^{(3)}_{2}(e_i\wedge e_j\wedge e_k \wedge d_2(f_h))  + f_h \otimes e_i^.e_j^.e_k.  $$
The behavior of these two maps under linkage for grade 3 perfect ideals is the main subject of the next section.

\medskip

Let us also define a few more maps coming from the critical representations. We will use the notation $w^{(i)}_{j,1}$ instead of $w^{(i)}_j$ to emphasize that, in general, the $j$th graded component of $W(d_i)$ may have multiple irreducible components---the following maps only correspond to a portion thereof. For compactness of notation we denote by $\varepsilon_{i_1, \ldots, i_t}$ the wedge product $e_{i_1} \wedge \ldots \wedge e_{i_t}$. Starting with $W(d_3)$, the map $ w^{(3)}_{3,1}: \bigwedge^5 F_1 \otimes F_1 \to \bigwedge^2 F_3 \otimes F_2 $ is defined as lifting along the map $\bigwedge^2 F_3 \otimes F_2 \to F_3 \otimes S_2F_2 \to F_3 \otimes F_2 \otimes F_2 $ (induced by $d_3$ and by symmetrization of $F_2 \otimes F_2$) of the term
%$  q^{(3)}_{3,1}(e_1,e_2,e_3,e_4,e_5 \otimes e_6) $. This term $q^{(3)}_{3,1}$ is defined as the image of 
\begin{equation}
\label{q3,3eq}
q^{(3)}_{3,1}(\varepsilon_{1,\ldots,5} \otimes e_6):= \sum_{i=1}^5 (-1)^{i+1} w^{(3)}_2(e_1 \wedge \ldots \widehat{e_i} \ldots \wedge e_5) \otimes e_i^.e_6  + $$
$$
+ \frac{1}{2} \sum_{1 \leq i<j \leq 5} (-1)^{i+j} w^{(3)}_2(e_1 \wedge \ldots \widehat{e_{i,j}}  \ldots \wedge e_6) \otimes e_i^.e_j.
\end{equation}
after applying the symmetrization map $F_3 \otimes F_2 \otimes F_2 \to F_3 \otimes S_2F_2$.
Similarly, the map $ w^{(3)}_{4,1}: \bigwedge^5 F_1 \otimes \bigwedge^3 F_1 \to \bigwedge^3 F_3 \otimes F_2 $ is defined as lifting along the map $\bigwedge^3 F_3 \otimes F_2 \to \bigwedge^2 F_3 \otimes S_2F_2 \to (F_3 \otimes F_2 \otimes F_3 \otimes F_2 +  \bigwedge^2 F_3 \otimes F_2 \otimes F_2)$ (induced by $d_3$ and by usual wedge product and symmetrization) of the term
%$ q^{(3)}_{4,1}= \tau_1 q^{(3)}'_{4,1} + \tau_2 q^{(3)}''_{4,1}  $
%$$  q^{(3)}_{4,1}(e_1,e_2,e_3,e_4,e_5 \otimes e_6,e_7,e_8):= $$
%\sum_{i=1}^4 e_i^.e_5 \otimes e_j^.e_k^. e_l
\begin{equation}
\label{q3,4eq}
q^{(3)}_{4,1}(\varepsilon_{1,\ldots,5} \otimes \varepsilon_{6,7,8}):= 2[\sum_{i=1}^5 (-1)^{i+1} w^{(3)}_2(e_1 \wedge \ldots \widehat{e_i} \ldots \wedge e_5) \otimes  w^{(3)}_2(e_i \wedge e_6 \wedge e_7 \wedge e_8) + $$
$$ + \sum_{i,j,k=6,7,8} (-1)^{i+1} w^{(3)}_{3,1}(\varepsilon_{1,\ldots,5} \otimes e_i) \otimes e_j^.e_k +  \frac{1}{3} \sum_{1 \leq i<j \leq 5} \sum_{k=6}^8 (-1)^{i+j+k} w^{(3)}_{3,1}(\varepsilon_{1,\ldots \hat{i},\hat{j}, \hat{k}, \ldots, 8} \otimes e_k) \otimes e_i^.e_j].
\end{equation}

%$$ + \sum_{i=6}^8 (-1)^{i+1} w^{(3)}_{3,1}(e_1 \wedge \ldots \wedge e_5 \otimes e_i) \otimes \prod_{j \in \lbrace 6,7,8 \rbrace \setminus \lbrace i \rbrace} e_j + $$  e_1 \wedge \ldots \widehat{e_{i,j,k}}  \ldots \wedge e_8

For $j \geq 2$, there are multiple irreducible components in the $j$th graded piece of $W(d_1)$ in general. We define the map $w^{(1)}_{2,1}: \bigwedge^4 F_1 \otimes F_1 \to \bigwedge^2 F_3$ as follows: $w^{(1)}_{2,1}(e_1 \wedge e_2 \wedge e_3 \wedge e_4 \otimes e_5)$ is the lift of
\begin{equation}
\label{q1,2eq}
q^{(1)}_{2,1}(\varepsilon_{1,\ldots,4} \otimes e_5):=  e_1^.e_5 \otimes e_2^.e_3^.e_4 - e_2^.e_5 \otimes e_1^.e_3^.e_4 + e_3^.e_5 \otimes e_1^.e_2^.e_4 - e_4^.e_5 \otimes e_1^.e_2^.e_3 + $$
%$$ + e_1^.e_2 \otimes e_3^.e_4^. e_5 - e_1^.e_3 \otimes e_2^.e_4^. e_5 + e_1^.e_4 \otimes e_2^.e_3^. e_5 + e_2^.e_3 \otimes e_1^.e_4^. e_5 - e_2^.e_4 \otimes e_1^.e_3^. e_5 + $$
%$$ + e_3^.e_4 \otimes e_1^.e_2^. e_5 - 2d_1(e_5) v^{(3)}_2(e_1,e_2,e_3,e_4).
$$ +d_1(e_1) w^{(3)}_2(\varepsilon_{2,3,4,5})-d_1(e_2) w^{(3)}_2(\varepsilon_{1,3,4,5})+d_1(e_3) w^{(3)}_2(\varepsilon_{1,2,4,5})-d_1(e_4) w^{(3)}_2(\varepsilon_{1,2,3,5})
\end{equation}
along the map $\bigwedge^2 F_3 \to F_3 \otimes F_2$ induced by $d_3$.
%$$  q^{(1)}_{2,1}(e_1,e_2,e_3,e_4 \otimes e_5):=  e_1^.e_5 \otimes e_2^.e_3^. e_4 - e_2^.e_5 \otimes e_1^.e_3^. e_4 + e_3^.e_5 \otimes e_1^.e_2^. e_4 - e_4^.e_5 \otimes e_1^.e_2^. e_3 + $$
%$$ + e_1^.e_2 \otimes e_3^.e_4^. e_5 - e_1^.e_3 \otimes e_2^.e_4^. e_5 + e_1^.e_4 \otimes e_2^.e_3^. e_5 + e_2^.e_3 \otimes e_1^.e_4^. e_5 - e_2^.e_4 \otimes e_1^.e_3^. e_5 + $$
%$$ + e_3^.e_4 \otimes e_1^.e_2^. e_5 - 2d_1(e_5) v^{(3)}_2(e_1,e_2,e_3,e_4). $$
It it not hard to check that the composition of $q^{(1)}_{2,1}$ with the map $F_3 \otimes F_2 \to S_2F_2$ induced by $d_3$ is zero. We can say briefly that $w^{(1)}_{2,1}$ is the lift of the relation $w^{(3)}_1 \otimes w^{(1)}_1 - w^{(3)}_2 \otimes d_1$.

%\rc next maps are more complicated in general, check signs \ec
The next map $w^{(1)}_{3,1}: \bigwedge^4 F_1 \otimes \bigwedge^3 F_1 \to \bigwedge^3 F_3$ is defined similarly as lifting of the relation $w^{(3)}_1 \otimes w^{(1)}_{2,1} - w^{(3)}_2 \wedge w^{(1)}_{1} +  w^{(3)}_{3,1} \otimes d_1$. Explicitly, the term to lift is 
%In analogous way, for $ s \geq 1$ we define maps $$ w^{(1)}_{2s,1}: (\bigwedge^4 F_1)^{\otimes s} \otimes F_1 \to \bigwedge^{2s} F_3; \quad w^{(1)}_{2s+1,1}: (\bigwedge^4 F_1)^{\otimes s} \otimes \bigwedge^3 F_1 \to \bigwedge^{2s+1} F_3$$
%as lift of $w^{(3)}_1 \otimes w^{(1)}_{k-1,1} - w^{(3)}_2 \wedge w^{(1)}_{k-2,1}$ along the map $\bigwedge^k F_3 \to \bigwedge^{k-1}F_3 \otimes F_2$ for the opportune $k= 2s, 2s+1$.
\begin{equation}
\label{q1,3eq} 
q^{(1)}_{3,1}(\varepsilon_{1,\ldots,4} \otimes \varepsilon_{5,6,7}):= \sum_{i= 5}^7 (-1)^{i+1} w^{(1)}_{2,1}(\varepsilon_{1,\ldots,4} \otimes e_i) \otimes e_j^.e_k+ $$
 $$  + \sum_{i=1}^4 (-1)^{i+1} w^{(3)}_{2}(\varepsilon_{5,6,7,i}) \otimes e_j^.e_k^.e_r - \sum_{i=1}^4 \sum_{j=5}^7 (-1)^{i+j} d_1(e_i) w_{3,1}^{(3)}(e_1 \wedge \ldots \widehat{e_{i,j}}  \ldots \wedge e_7 \otimes e_j)
\end{equation}
%\pm \sum_{i=1}^4 (-1)^{i} w^{(3)}_{2}(\varepsilon_{i,5,6,7} \otimes e_i) \otimes \prod_{j=1,2,3,4 \atop \scriptstyle j \neq i} e_j q^{(1)}_{3,1}(\varepsilon_{1,\ldots,4} \otimes \varepsilon_{5,6,7}):= \sum_{i,j,k = 5,6,7} (-1)^{i+1} w^{(1)}_{2,1}(\varepsilon_{1,\ldots,4} \otimes e_i) \otimes e_j^.e_k+ $$
% $$ \rc \pm \sum_{i,j,k,r=1,2,3,4} (-1)^{i} w^{(3)}_{2}(\varepsilon_{i,5,6,7}) \otimes e_j^.e_k^.e_r \pm \sum_{i=1}^4 \sum_{j=5,6,7} (-1)^{i+j} w_{3,1}^{(3)}(e_1 \wedge \ldots \widehat{e_{i,j}}  \ldots \wedge e_7 \otimes e_j)

Finally, we can define analogous maps in $W(d_2)$ generalizing $w^{(2)}_2$. The map $w^{(2)}_{3,1}: \bigwedge^4 F_1 \otimes  F_1 \otimes F_2 \to \bigwedge^3 F_3$ is defined by lifting the relation $ w^{(1)}_{2,1} \otimes 1_{F_2} + w^{(3)}_1 \otimes w^{(2)}_{2} - w^{(3)}_2 \wedge w^{(2)}_{1} +  w^{(3)}_{3,1}(d_2)$.

The map $w^{(2)}_{4,1}: S_{2221} F_1 \otimes F_2 \to \bigwedge^4 F_3$ is defined by lifting the relation $ w^{(1)}_{3,1} \otimes 1_{F_2} + w^{(3)}_1 \otimes w^{(2)}_{3,1} - w^{(3)}_2 \wedge w^{(2)}_{2} +  w^{(3)}_{3,1}\wedge w^{(2)}_{1} - w^{(3)}_{4,1}(d_2)$. Explicitly:
%$$ w^{(2)}_{2s,1}: (\bigwedge^4 F_1)^{\otimes s-1} \otimes \bigwedge^3 F_1 \otimes F_2 \to \bigwedge^{2s} F_3; \quad w^{(2)}_{2s+1,1}: (\bigwedge^4 F_1)^{\otimes s} \otimes F_1 \otimes F_2 \to \bigwedge^{2s+1} F_3$$
%defined as lift of $w^{(3)}_1 \otimes w^{(2)}_{k-1,1} + w^{(3)}_2 \wedge w^{(2)}_{k-2,1} - 1_{F_2} \wedge w^{(1)}_{k-1,1}$ along the same map $\bigwedge^k F_3 \to \bigwedge^{k-1}F_3 \otimes F_2$ for $k= 2s, 2s+1$.
\begin{equation}
\label{q2,3eq}
q^{(2)}_{3,1}(\varepsilon_{1,\ldots,4} \otimes e_5 \otimes f_h):= 
w^{(1)}_{2,1}(\varepsilon_{1,\ldots,4} \otimes e_5) \otimes f_h -  w^{(3)}_{3,1}(\varepsilon_{1,\ldots,5} \otimes d_2(f_h)  + \varepsilon_{1,\ldots,4} \wedge d_2(f_h) \otimes e_5)+ $$
$$ + \sum_{j=1}^4 (-1)^j e_j^.e_5 \otimes w^{(2)}_2(e_1 \wedge \ldots \widehat{e_{j}}  \ldots \wedge e_4 \otimes f_h) - \sum_{j=1}^4 (-1)^j e_j^.f_h \otimes w^{(3)}_2(e_1 \wedge \ldots \widehat{e_{j}}  \ldots \wedge e_5). 
\end{equation}
\begin{equation}
\label{q2,4eq} 
q^{(2)}_{4,1}(\varepsilon_{1,\ldots,4} \otimes \varepsilon_{5,6,7} \otimes f_h):= 
\frac{1}{2} w^{(3)}_{4,1}(\varepsilon_{1,\ldots,4} \wedge d_2(f_h) \otimes \varepsilon_{5,6,7}) 
-  w^{(1)}_{3,1}(\varepsilon_{1,\ldots,4} \otimes \varepsilon_{5,6,7}) \otimes f_h+ $$
$$ + \sum_{i=5}^7 (-1)^{i+1} [e_j^.e_k \otimes w^{(2)}_{3,1}(\varepsilon_{1,\ldots,4} \otimes e_i \otimes f_h)+  w^{(2)}_2(e_5, e_6, e_7 \otimes f_h) \otimes w^{(3)}_2(\varepsilon_{1,\ldots,4}). 
\end{equation}

\subsection{Generic structure maps and relations}
We will frequently need to verify relations among the higher structure maps $w^{(i)}_j$ for an arbitrary resolution $\mb{F}$ of a given format. For example, as we inductively defined the maps $w^{(i)}_j$ in \S2.2, we needed to know that certain composites were zero in order to lift $q^{(i)}_j$. Sometimes the relations are easy to verify directly, but it is often more convenient to leverage the following result (c.f. \cite[Lemma 2.4]{w89}, \cite[Proposition 10.4]{W18}):
\begin{thm} 
	\label{equivariantrelations}
	Fix a format $(1,r_1,r_2,r_3)$. If a relation among $w^{(i)}_j$ holds for every choice of structure maps over every split exact complex of the given format, then it holds in general.
	
	In particular, to verify a $\prod_{i=1}^3 GL(F_i)$-equivariant set of relations on the maps $w^{(i)}_j$, it is sufficient to check them for every choice of higher structure maps for one split exact complex.
\end{thm}
\begin{proof}
	Since we defined structure maps in terms of lifts, it is evident that if $R \to S$ is a ring homomorphism and $\{w^{(i)}_j\}$ is a collection of structure maps for a resolution $\mb{F}$ over $R$, then $\{w^{(i)}_j \otimes S\}$ is a collection of structure maps for the complex $\mb{F} \otimes S$. In particular, structure maps remain valid under localization.
	
	Writing $d_1$ for the first differential of $\mb{F}$, let $u\in I(d_1)$ be a nonzerodivisor. Such an element exists because $\operatorname{grade} I(d_1) \geq 1$. Then $\mb{F}$ is a split exact complex on the open set $\Spec R_u$, which is moreover dense in $\Spec R$. Thus a relation holds for $w^{(i)}_j$ if and only if it holds for the localized structure maps over a split exact complex.
	
	The second statement of the theorem follows immediately, as all split exact complexes are equivalent up to a change of basis.
\end{proof}
In order to verify relations for arbitrary choices of structure maps, we introduce the notion of \emph{generic} structure maps $v^{(i)}_j$ for a resolution $\mb{F}$. We define these inductively using the same lifts as for $w^{(i)}_j$, replacing all instances of $w^{(i)}_j$ with $v^{(i)}_j$ in the definitions of the maps $q^{(i)}_j$ that we lift. The difference is that, when the lift is not unique, we parametrize all possible lifts with additional variables. To define $v^{(3)}_1$ for example, we adjoin variables $b_{ij}^k$ ($1\leq i < j \leq r_1$ and $1 \leq k \leq r_3$), which we call \emph{defect variables}, and set
\[
	v^{(3)}_1 = w^{(3)}_1 + d_3 M
\]
where $w^{(3)}_1$ is a particular lift of $q^{(3)}_1$ and $M(e_i \wedge e_j) = b_{ij}^k g_k$ for $i<j$. That is, $M$ is a generic map $\bigwedge^2 F_1 \to F_3$. Evidently the maps $v^{(i)}_j$ specialize to any particular choice of structure maps $w^{(i)}_j$, so Theorem~\ref{equivariantrelations} implies it is sufficient to verify equivariant relations on maps $v^{(i)}_j$ computed over a particular split exact complex. When these calculations arise, we defer them to \S5, with the especially cumbersome ones left to a computer. More background on the maps $v^{(i)}_j$ can be found in that section as well.

\section{Linkage of higher structure maps}
The aim of this section is to describe how some of the structure maps can be computed for the free resolution of a linked ideal, in terms of the structure maps of a given free resolution of the original ideal. In particular we are interested in $w^{(3)}_2$ and $w^{(2)}_2$, as the multiplicative structure has already been studied in \cite{AKM88}.

%In the following we will often %order to study the behavior of higher structure maps under linkage,
%need to check that certain relations hold among the higher
%structure maps of an arbitrary free resolution of length three.
%The following result will be widely used throughout this paper, as
%it reduces the verification of these relations to the case of a split exact
%complex (where higher structure maps are computed with generic liftings).
%Let ${\hat R}_{gen}$ be the generic ring corresponding to the format of a complex $\FF$ of length 3.

%Recall that any relation in the generic complex over the generic ring specializes to any acyclic complex of the same given format. All the computations of relations for higher structure maps of a split exact complex are postponed to Section 5.

%In many of the subsequent proofs, applying Theorem \ref{equivariantrelations}, we reduce to check that some $ \prod^{3}_{i=1} GL(F_i) $-equivariant sets of relations hold over a split exact complex with generic liftings. We postpone all these computations to Section 5. 
From now on, we prefer to slightly change the notation from the previous sections to match the notation of \cite{AKM88}. 
Our setting is the following:
let $R$ be a Gorenstein local (or graded) ring with maximal ideal $\m$, with $2,3\notin \m$. Let $I \subseteq R$ be a perfect ideal of height 3. The minimal free resolution of $\frac{R}{I}$ is
\begin{equation}
\label{basecomplex}
\A: 0 \longrightarrow A_3 \buildrel{a_3}\over\longrightarrow  A_2 \buildrel{a_2}\over\longrightarrow A_1 \buildrel{a_1}\over\longrightarrow R.
\end{equation}
Set $r_i= \rank A_i$. Denote the entries of the matrices of $a_1, a_2, a_3$ respectively by $\lbrace x_i \rbrace$, $\lbrace y_{ij} \rbrace$, $\lbrace z_{ij} \rbrace$.
Given a regular sequence $\mathfrak{a} \subseteq I$ of maximal length, we denote by $J$ the linked ideal $(\mathfrak{a}):I$.
Let $\BB$ be the Koszul complex resolving $\frac{R}{(\mathfrak{a})}$ and let $\alpha_i: B_i \to A_i$ be the map obtained by lifting the quotient map $\pi: \frac{R}{(\mathfrak{a})} \to \frac{R}{I}, $ after fixing the choice of a multiplicative structure on $\A$.

Take basis for $B_1$ equal to $\lbrace s_1, s_2, s_3  \rbrace$, basis for $B_2$ equal to $\lbrace t_1, t_2, t_3 \rbrace$ and basis for $B_3$ equal to $ \lbrace w \rbrace $. The multiplicative structure on $\BB$ provides relations $s_i^.s_j= (-1)^{i+j+1} t_k$ and $s_1^.s_2^.s_3 = w$.
For $i=1,2,3$ let $\tau_i$ be the isomorphism $B_i^* \to B_{3-i}$ induced by such structure. Define maps $\beta_i: A_i^* \to B_{3-i}$ setting $\beta_i:= \tau_i \alpha_i^*$. The mapping cone of the complex map $\A^* \to \BB$ defined by the maps $\beta_i$ gives a free resolution $\DD$ of $\frac{R}{J}$ (not necessarily minimal). 
We have 
\begin{equation}
\label{mappingcone}
\DD: 0 \longrightarrow A_1^* \buildrel{d_3}\over\longrightarrow  A_2^* \oplus B_2 \buildrel{d_2}\over\longrightarrow A_3^* \oplus B_1 \buildrel{d_1}\over\longrightarrow R.
\end{equation}
The free modules in the complex $\DD$ will be also denoted by $D_3, D_2, D_1$.
The differentials are given by the following formulas:
$$ d_1 = \bmatrix \beta_3 & b_1 \endbmatrix; \quad  d_2 = \bmatrix a_3^* & 0 \\ -\beta_2 & -b_2 \endbmatrix; \quad d_3 = \bmatrix a_2^*  \\ \beta_1  \endbmatrix. $$
The entries of $\beta_1$ are simply the coefficients which express the elements of $\mathfrak{a}$ in function of the fixed set of minimal generators of $I$ determined by the entries of $d_1$. 
The matrices of the maps $\beta_2$, $\beta_3$ can be obtained from the multiplicative structure on $\A$. % by the formulas
%$$ \beta_2^t = w^{(3)}_1 \cdot (\bigwedge^2 \beta_1); \quad  \beta_3^t = w^{(1)}_1 \cdot (\bigwedge^3 \beta_1^t). $$ 

In this section we denote the basis of $A_1, A_2, A_3$ respectively by $\lbrace e_1, \ldots, e_{r_1}  \rbrace$, $\lbrace f_1, \ldots, f_{r_2} \rbrace$, $\lbrace g_1, \ldots, g_{r_3}  \rbrace$ and the dual basis by $\lbrace \epsilon_1, \ldots, \epsilon_{r_1}  \rbrace$, $\lbrace \phi_1, \ldots, \phi_{r_2} \rbrace$, $\lbrace \gamma_1, \ldots, \gamma_{r_3}  \rbrace$.
We also denote by %by $z_{ij}, y_{ij}, x_i$ the entries of $a_3, a_2, a_1$ and by 
$u_{ij}$ the entries of $\alpha_1$ and by $\langle \cdot, \cdot \rangle$ the usual evaluation of an element of a module with respect to an element of the dual. 

\begin{remark}
\label{notation}
The elements of the regular sequence $\mathfrak{a}$ are $\sum_{i=1}^{r_1} u_{ij} x_i$ for $j=1,2,3$. Hence for $j=1,2,3$ we have  
$\alpha_1(s_j) = \sum_{i=1}^{r_1} u_{ij} e_i.$ 
When considering a minimal linkage (i.e. the elements of $\mathfrak{a}$ are among minimal generators of $I$), we can assume $(\mathfrak{a})=(x_1, x_2, x_3)$ and $\alpha_1(s_j) =  e_j$ for $j=1,2,3$.

For the maps $\beta_1$, $\beta_2$, $\beta_3$ we have formulas
$$ \beta_1(\epsilon_k)= \sum_{j=1}^3 u_{kj} t_j; \quad \beta_2(\phi_h)= \sum_{j=1}^3 (-1)^{j+1} \langle \alpha_1(s_{k_1})^.\alpha_1(s_{k_2}), \phi_h \rangle s_j $$
where $k_1, k_2$ are the two indices in $ \lbrace  1,2,3 \rbrace $ different from $j$, and 
$$ \beta_3(\gamma_t)= \langle \alpha_1(s_1)^.\alpha_1(s_2)^.\alpha_1(s_3), \gamma_t  \rangle.  $$
%$\beta_2(\phi_h) = \chi_1 s_1 + \chi_2 s_2 + \chi_3 s_3$ where
%$$ \chi_k = \sum_{1 \leq i < j \leq r_1} (\alpha_1)_{ij}^{\hat{k}} \langle e_i^.e_j, \phi_h \rangle = \pm \langle \alpha_1(s_{k_1})^.\alpha_1(s_{k_2}), \phi_h \rangle $$ and $k_1, k_2$ are defined as in relation (\ref{rel2}).
\end{remark}

%In the following denote by $(\alpha_1)_{ijk}$ the maximal minor of $\alpha_1$ obtained taking the rows indexed by $i,j,k$ and by $(\alpha_1)_{ij}^{\hat{k}}$ the $2 \times 2 $ minor of $\alpha_1$ obtained taking the rows indexed by $i,j$ and excluding the column indexed by $k$.
%By definition of $\alpha_1$, %and by linearity of the maps $ w^{(2)}_2 $, 
%we have the relations 
%\begin{equation}
%\label{rel1}
 % \alpha_1(s_1) \wedge \alpha_1(s_2) \wedge \alpha_1(s_3)= \sum_{ijk} (\alpha_1)_{ijk} (e_i \wedge e_j \wedge e_k).
%\end{equation}
%\begin{equation}
%\label{rel2}
%\alpha_1(s_{k_1}) \wedge \alpha_1(s_{k_2})= \pm \sum_{ij} (\alpha_1)_{ij}^{\hat{k}} (e_i \wedge e_j). 
%\end{equation}
%where $k_1, k_2$ are the two indices in $ \lbrace  1,2,3 \rbrace $ different from $k$.

%\medskip
To describe the multiplicative structure on $\DD$, Avramov, Kustin and Miller introduced the next two linear maps \cite[Lemma 1.9 and 1.10]{AKM88}:
%Next result reinterprets the maps $X$ and $Y$, introduced in , in function of the higher structure maps.
 %We recall their definitions: the map 
 
The map $X: \bigwedge^3 A_1 \otimes \bigwedge^2 A_3 \to A_2^*$ is defined as the lift of $$ \langle e_i e_j e_k, \gamma_s   \rangle \gamma_t -
\langle e_i e_j e_k, \gamma_t   \rangle \gamma_s $$ along the map $a_3^*$. 

The map $Y: \bigwedge^3 A_1 \otimes A_3^* \otimes A_2^* \to A_1^*$ is defined by the relation
$$ \langle f_h, a_2^*(Y(e_i \wedge e_j \wedge e_k \otimes \gamma_s \otimes \phi_l)) \rangle = \langle e_i^.e_j^.e_k, \gamma_s  \rangle \langle f_h, \phi_l \rangle - \langle f_h, X(e_i \wedge e_j \wedge e_k \otimes \gamma_s \wedge a_3^*(\phi_l))  \rangle + $$
$$ - \langle e_i^.e_j, \phi_l \rangle \langle e_k^.f_h, \gamma_s \rangle + \langle e_i^.e_k, \phi_l \rangle \langle e_j^.f_h, \gamma_s \rangle - \langle e_j^.e_k, \phi_l \rangle \langle e_i^.f_h, \gamma_s \rangle. $$
These maps are needed to prove the following theorem. 

\begin{thm}
\label{AKMthm}
\rm  (\cite[Theorem 1.13]{AKM88})  \it \\
The multiplication maps $\bigwedge^2 D_1 \to D_2$ and $ D_1 \otimes D_2 \to D_3$ are described as follows. 
$$
s_i^.s_j = (-1)^{i+j+1}t_k.
\quad
s_i^. \gamma_t = \sum_{h=1}^{r_2} \langle \alpha_1(s_i)^.f_h,  \gamma_t \rangle \phi_h.
$$ $$
\gamma_u^. \gamma_t = X(\alpha_1(s_1) \wedge \alpha_1(s_2) \wedge \alpha_1(s_3) \otimes  \gamma_u \wedge \gamma_t) + \lambda(\gamma_u \wedge \gamma_t).
$$
% The multiplication map $ D_1 \otimes D_2 \to D_3$ is described as follows. 
$$
s_j^.t_p = \langle s_j^.t_p, w^*  \rangle (\sum_{k=1}^{r_1}x_k \epsilon_k ).
\quad
s_j^. \phi_h = \sum_{k=1}^{r_1} \langle \alpha_1(s_j)^.e_k,  \phi_h \rangle \epsilon_k.
$$ $$
\gamma_u^. t_j = \sum_{k=1}^{r_1} \langle \alpha_2(t_j)^.e_k,  \gamma_u \rangle) \epsilon_k.
$$ $$
\gamma_u^. \phi_h =  Y(\alpha_1(s_1)\wedge \alpha_1(s_2) \wedge \alpha_1(s_3) \otimes \phi_h \otimes \gamma_u) +  a_1^*(\mu(\gamma_u \wedge \phi_h)). 
$$
%\begin{equation}
%\label{eqlambda}
%\gamma_u^. \gamma_t = \sum_{h=1}^{r_2} \sum_{ijk} (\alpha_1)_{ijk} \langle w^{(2)}_2(e_i, e_j, e_k  \otimes f_h),  \gamma_u \wedge \gamma_t \rangle) \phi_h + \lambda(\gamma_u \wedge \gamma_t).
%\end{equation}
\end{thm}

In the above theorem, the term $\lambda(\gamma_u \wedge \gamma_t) \in B_2$ is defined as the lift of $$\beta_2(\sum_{h=1}^{r_2} \langle w^{(2)}_2(\alpha_1(s_1),\alpha_1(s_2),\alpha_1(s_3)  \otimes f_h),  \gamma_u \wedge \gamma_t \rangle \phi_h)$$
% $$\beta_2(\sum_{h=1}^{r_2} \sum_{ijk} (\alpha_1)_{ijk} \langle w^{(2)}_2(e_i, e_j, e_k  \otimes f_h),  \gamma_u \wedge \gamma_t \rangle) \phi_h)$$
 along the differential $-b_2$ in the Koszul complex $\BB$. The term $\mu$ is also defined along the proof. % of the same theorem.

\subsection{The multiplicative structure on $\DD$}

In the first part of this section we reinterpret the maps $X$ and $Y$ in terms of the structure maps $w^{(3)}_2$,$ w^{(2)}_2 $. This allows us to find a simplified version of Theorem \ref{AKMthm}, showing that there exists an opportune lifting for which $\lambda$ and $\mu$ are zero. We have:

% Next result reinterprets the maps $X$ and $Y$, introduced in \cite[Lemma 1.9 and 1.10]{AKM88}, in function of the higher structure maps.
% We recall their definitions: the map $X: \bigwedge^3 A_1 \otimes \bigwedge^2 A_3 \to A_2^*$ is defined as the lift of $$ \langle e_i e_j e_k, \gamma_s   \rangle \gamma_t -
%\langle e_i e_j e_k, \gamma_t   \rangle \gamma_s $$ along the map $a_3^*$. 

%The map $Y: \bigwedge^3 A_1 \otimes A_3^* \otimes A_2^* \to A_1^*$ is defined by the relation
%$$ \langle f_h, a_2^*(Y(e_i \wedge e_j \wedge e_k \otimes \gamma_s \otimes \phi_l)) \rangle = \langle e_i^.e_j^.e_k, \gamma_s  \rangle \langle f_h, \phi_l \rangle - \langle f_h, X(e_i \wedge e_j \wedge e_k \otimes \gamma_s \wedge a_3^*(\phi_l))  \rangle + $$
%$$ - \langle e_i^.e_j, \phi_l \rangle \langle e_k^.f_h, \gamma_s \rangle + \langle e_i^.e_k, \phi_l \rangle \langle e_j^.f_h, \gamma_s \rangle - \langle e_j^.e_k, \phi_l \rangle \langle e_i^.f_h, \gamma_s \rangle. $$

\begin{lem}
\label{XY}
For any choice of indices, the following relations hold:
\begin{equation}
\label{X}
X(e_i \wedge e_j \wedge e_k, \gamma_s \wedge \gamma_t) \equiv   \sum_{h=1}^{r_2} \langle w^{(2)}_2(e_i \wedge e_j \wedge e_k, f_h),  \gamma_s \wedge \gamma_t \rangle \phi_h \, \mbox{ \rm mod } \ker(a_3^*).
\end{equation}
Replace the term $\langle f_h, X \rangle$ in the definition of $Y$ by 
$ \langle w^{(2)}_2(e_i \wedge e_j \wedge e_k, f_h),  \gamma_s \wedge a_3^*(\phi_l) \rangle $. Then we get
\begin{equation}
\label{Y}
Y(e_i \wedge e_j \wedge e_k \otimes \gamma_s \otimes \phi_l) \equiv \sum_{t=1}^{r_1} \langle w^{(3)}_2(e_{i}, e_{j}, e_{k}, e_t ), \phi_l \otimes \gamma_s  \rangle) \epsilon_t \, \mbox{ \rm mod } \ker(a_2^*).
\end{equation}
\end{lem}

\begin{proof}
By Theorem \ref{equivariantrelations} it is sufficient to check both relations over a split exact complex. This is done in Lemma \ref{split1}.
\end{proof}

%Working on a split exact complex with arbitrary ranks we check that, modulo the kernel of $a_3^*$, we have

%\bf add this computation in the file; \rm 
 
%\rc is $Y$ correct? \ec 
%\bf Note: \rm 
%(1) This provides an easier way to compute $w^{(2)}_2$. \\
%(2) There is an analogous definition for each maps $w^{(2)}_j$ in terms of the map $w^{(1)}_{j-1}$ if $r_1= 4$, or if we use this map considering only subspaces of $A_1$ of dimension 4. These maps are defined as those for the formats $(1,4,m+3, m)$. \bc I wonder if we need only those maps to compute $w^{(3)}_2$ on the linked ideal or more indices get involved. \ec \\
%(3) Say that $w^{(2)}_{j,1}$ are the maps defined for $(1,4,m+3, m)$ and computed on a choice of $4$ indices out of $1, \ldots, r_1$.

%\medskip

%Denote now by $(\alpha_1)_{ijk}$ the maximal minor of $\alpha_1$ obtained taking the rows indexed by $i,j,k$ and by $(\alpha_1)_{ij}^{\hat{k}}$ the $2 \times 2 $ minor of $\alpha_1$ obtained taking the rows indexed by $i,j$ and excluding the column indexed by $k$.
%By definition of $\alpha_1$, %and by linearity of the maps $ w^{(2)}_2 $, 
%we have the relations 
%\begin{equation}
%\label{rel1}
%  \alpha_1(s_1) \wedge \alpha_1(s_2) \wedge \alpha_1(s_3)= \sum_{ijk} (\alpha_1)_{ijk} (e_i \wedge e_j \wedge e_k).
%\end{equation}
%\begin{equation}
%\label{rel2}
%\alpha_1(s_{k_1}) \wedge \alpha_1(s_{k_2})= \pm \sum_{ij} (\alpha_1)_{ij}^{\hat{k}} (e_i \wedge e_j). 
%\end{equation}
%where $k_1, k_2$ are the two indices in $ \lbrace  1,2,3 \rbrace $ different from $k$.

\begin{lem}
\label{lambdacorrection}
%\rc not completely sure \ec
The term $\lambda(\gamma_u \wedge \gamma_t)$ appearing in (\ref{eqlambda}) can be chosen to be zero.
\end{lem}

\begin{proof}
Let $ \Theta := \sum_{h=1}^{r_2} \langle w^{(2)}_2(\alpha_1(s_1),\alpha_1(s_2),\alpha_1(s_3)  \otimes f_h),  \gamma_u \wedge \gamma_t \rangle \phi_h$.
%Let $ \Theta := \sum_{h=1}^{r_2} \sum_{ijk} (\alpha_1)_{ijk} \langle w^{(2)}_2(e_i, e_j, e_k  \otimes f_h),  \gamma_u \wedge \gamma_t \rangle) \phi_h$.
We have to show that $\beta_2(\Theta)= 0.$ By definition of $\beta_2$ (see Remark \ref{notation}) we know that 
$$ \beta_2(\phi_h)= \sum_{j=1}^3 (-1)^{j+1} \langle \alpha_1(s_{k_1})^.\alpha_1(s_{k_2}), \phi_h \rangle s_j $$
where $k_1, k_2$ are 
%$\beta_2(\phi_h) = \chi_1 s_1 + \chi_2 s_2 + \chi_3 s_3$ where
%$ \chi_k = %\sum_{1 \leq i < j \leq r_1} (\alpha_1)_{ij}^{\hat{k}} \langle e_i^.e_j, \phi_h \rangle = \pm \langle \alpha_1(s_{k_1})^.\alpha_1(s_{k_2}), \phi_h \rangle $$ and $k_1, k_2$ are 
the two indices in $ \lbrace  1,2,3 \rbrace $ different from $j$. We show that the coefficient of $s_1$ in $ \beta_2(\Theta) $ is zero. For those of $s_2, s_3$ the argument is analogous. %By relations (\ref{rel1}), (\ref{rel2}), 
The coefficient of $s_1$ is
\begin{equation}
\label{relationforlambda}
\sum_{h=1}^{r_2} \langle w^{(2)}_2(\alpha_1(s_{1}),\alpha_1(s_{2}),\alpha_1(s_{3})  \otimes f_h),  \gamma_u \wedge \gamma_t \rangle) \cdot \langle \alpha_1(s_{2})^.\alpha_1(s_{3}), \phi_h \rangle.
\end{equation}
Since $\mathfrak{a}$ is a regular sequence, the map $\alpha_1$ is injective and we can take $\alpha_1(s_{1}),\alpha_1(s_{2}),\alpha_1(s_{3})$ to be linearly independent elements of $A_1$. By linearity of $w^{(2)}_2$ and $ w^{(3)}_1$ we can expand 
(\ref{relationforlambda}) in terms involving generators of $F_1$ and coefficients depending on the minors of $\alpha_1$. We can group together the terms having the same coefficient.
We need therefore to show that, for every choice of $e_{i_1},e_{i_2},e_{i_3},e_{j_2},e_{j_3} \in A_1$ and $\gamma_u, \gamma_t \in A_3^*$, the term
%$$ (a): \sum_{h=1}^{r_2} \langle w^{(2)}_2(e_{i_1},e_{i_2},e_{i_3}  \otimes f_h),  \gamma_u \wedge \gamma_t \rangle \cdot \langle e_{i_2},e_{i_3}, \phi_h \rangle  $$
%$$ (b): \sum_{h=1}^{r_2} \langle w^{(2)}_2(e_{i_1},e_{i_2},e_{i_3}  \otimes f_h),  \gamma_u \wedge \gamma_t \rangle \cdot \langle e_{i_2},e_{j_3}, \phi_h \rangle + \langle w^{(2)}_2(e_{i_1},e_{i_2},e_{j_3}  \otimes f_h),  \gamma_u \wedge \gamma_t \rangle \cdot \langle e_{i_2},e_{i_3}, \phi_h \rangle    $$
$$  \sum_{h=1}^{r_2} \langle w^{(2)}_2(e_{i_1},e_{i_2},e_{i_3}  \otimes f_h),  \gamma_u \wedge \gamma_t \rangle \cdot \langle e_{j_2},e_{j_3}, \phi_h \rangle +\langle w^{(2)}_2(e_{i_1},e_{i_2},e_{j_3}  \otimes f_h),  \gamma_u \wedge \gamma_t \rangle \cdot \langle e_{j_2},e_{i_3}, \phi_h \rangle + $$ $$
\langle w^{(2)}_2(e_{i_1},e_{j_2},e_{i_3}  \otimes f_h),  \gamma_u \wedge \gamma_t \rangle \cdot \langle e_{i_2},e_{j_3}, \phi_h \rangle + \langle w^{(2)}_2(e_{i_1},e_{j_2},e_{j_3}  \otimes f_h),  \gamma_u \wedge \gamma_t \rangle \cdot \langle e_{i_2},e_{i_3}, \phi_h \rangle  $$
is zero.
In Lemma \ref{split2}, we perform the required computation over a split exact complex with defect variables. %Since this relation is $\GL(D_1), \GL(D_2), \GL(D_3)$-equivariant, 
By Theorem \ref{equivariantrelations} the same relations hold over the complex $\mathbb{D}$.
%If we show that 
%for each choice of generators $e_i,e_j,e_k \in A_1$ and $\gamma_u, \gamma_t \in A_3^*$ over a split exact complex, by \bc cite Weyman's Theorem \ec we obtain that the same term is zero over the complex $\mathbb{D}$. By linearity of $w^{(2)}_2$ and $ w^{(3)}_1$ it follows that the term in (\ref{relationforlambda}) is zero.
%The computation over a split exact complex with defect variables is done in Lemma \ref{split2}.
%\rc is correct that since this relation holds for each choice of generators, then it extends by linearity to all the elements? \ec
\end{proof}

As a consequence of Lemma \ref{XY} and Lemma \ref{lambdacorrection}, Theorem \ref{AKMthm} %\cite[Theorem 1.13]{AKM88} 
can be restated in the following way:

%w^{(2)}_1(\alpha_1(s_i) \otimes f_h)

\begin{prop}
\label{w3,1link}
The multiplication maps $\bigwedge^2 D_1 \to D_2$ and $ D_1 \otimes D_2 \to D_3$ are described as follows. 
\begin{equation}
s_i^.s_j = (-1)^{i+j+1}t_k.
\end{equation}
\begin{equation}
s_i^. \gamma_t = \sum_{h=1}^{r_2} \langle \alpha_1(s_i)^.f_h,  \gamma_t \rangle \phi_h.
\end{equation}
\begin{equation}
\label{eqlambda}
\gamma_u^. \gamma_t = \sum_{h=1}^{r_2} \langle w^{(2)}_2(\alpha_1(s_1),\alpha_1(s_2),\alpha_1(s_3)  \otimes f_h),  \gamma_u \wedge \gamma_t \rangle \phi_h.
\end{equation}
%\begin{equation}
%\label{eqlambda}
%\gamma_u^. \gamma_t = \sum_{h=1}^{r_2} \sum_{ijk} (\alpha_1)_{ijk} \langle w^{(2)}_2(e_i, e_j, e_k  \otimes f_h),  \gamma_u \wedge \gamma_t \rangle) \phi_h + \lambda(\gamma_u \wedge \gamma_t).
%\end{equation}
%The term $\lambda(\gamma_u \wedge \gamma_t) \in B_2$ is defined as the lift of $$\beta_2(\sum_{h=1}^{r_2} \langle w^{(2)}_2(\alpha_1(s_1),\alpha_1(s_2),\alpha_1(s_3)  \otimes f_h),  \gamma_u \wedge \gamma_t \rangle \phi_h)$$
% $$\beta_2(\sum_{h=1}^{r_2} \sum_{ijk} (\alpha_1)_{ijk} \langle w^{(2)}_2(e_i, e_j, e_k  \otimes f_h),  \gamma_u \wedge \gamma_t \rangle) \phi_h)$$
 %along the differential $-b_2$ in the Koszul complex $\BB$. 
%\end{prop}
%\bc I am not sure if we can interpret $\lambda$ using the higher structure theorems, I actually have not find yet examples where $\lambda$ is nonzero. Maybe $\lambda=0$ can be shown using split exact complex \ec
%\begin{itemize}
%\item 
%\item $ $
%\end{itemize}
%\begin{prop}
%\label{w2,1link}
%The multiplication map $ D_1 \otimes D_2 \to D_3$ is described as follows. 
\begin{equation}
s_j^.t_p = \langle s_j^.t_p, w^*  \rangle (\sum_{k=1}^{r_1}x_k \epsilon_k ).
\end{equation}
\begin{equation}
s_j^. \phi_h = \sum_{k=1}^{r_1} \langle \alpha_1(s_j)^.e_k,  \phi_h \rangle \epsilon_k.
\end{equation}
\begin{equation}
\label{w2,1gt}
\gamma_u^. t_j  = 
\sum_{k=1}^{r_1} (-1)^{j+1} \langle \alpha_1(s_{k_1})^.\alpha_1(s_{k_2})^.e_k,  \gamma_u \rangle) \epsilon_k, \mbox{ with } k_1,k_2 \in  \lbrace 1,2,3 \rbrace \setminus \lbrace j \rbrace.
\end{equation}
%%\sum_{k=1}^{r_1} \langle \alpha_2(t_j)^.e_k,  \gamma_u \rangle) \epsilon_k
\begin{equation}
\label{w2,1gp}
\gamma_u^. \phi_h =  \sum_{k=1}^{r_1} \langle w^{(3)}_2(\alpha_1(s_1), \alpha_1(s_2), \alpha_1(s_3), e_k ), \phi_h \otimes \gamma_u  \rangle) \epsilon_k. 
\end{equation}
%\gamma_u^. \phi_h =  \sum_{k=1}^{r_1} \sum_{i,j,l} (\alpha_1)_{ijl} \langle w^{(3)}_2(e_{i}, e_{j}, e_{l}, e_k ), \phi_h \otimes \gamma_u  \rangle) \epsilon_k + \rc a_1^*(\mu(\gamma_u \wedge \phi_h). \ec
%\rc The term $\mu(\gamma_u \wedge \phi_h)= 0$, proved below provided that $\lambda=0$ \ec  %the lift of $$\beta_2(\sum_{h=1}^{r_2} \sum_{ijk} (\alpha_1)_{ijk} \langle w^{(2)}_2(e_i, e_j, e_k  \otimes f_h),  \gamma_u \wedge \gamma_t \rangle) \phi_h)$$ along the differential $-b_2$ in the Koszul complex $\BB$. 
\end{prop}

\begin{proof}
By Theorem \ref{AKMthm}, Lemma \ref{XY} and Lemma \ref{lambdacorrection} we only need to prove relations (\ref{w2,1gt}) and (\ref{w2,1gp}). %Comparing \cite[Lemma 1.9, 1.10 and Theorem 1.13]{AKM88} with Lemma \ref{XY}, Proposition \ref{w3,1link} and Lemma \ref{lambdacorrection} we obtain that 
For the first one it is sufficient to observe that $\alpha_2(t_j) = (-1)^{j+1} \alpha_1(s_{k_1})^.\alpha_1(s_{k_2})$.

For the second one, comparing all these results also with \cite[Lemma 1.9, 1.10]{AKM88} we obtain that the components of $d_1(\gamma_u) \phi_h - \gamma_1^.d_2(\phi_h)$ with respect to $\phi_1, \ldots, \phi_{r_2}$ coincide with those of
$$ a_2^*(\sum_{k=1}^{r_1} \langle w^{(3)}_2(\alpha_1(s_1), \alpha_1(s_2), \alpha_1(s_3), e_k ), \phi_h \otimes \gamma_u  \rangle) \epsilon_k). $$ %coincides with the components . 

Since $\lambda= 0$, the term $d_1(\gamma_u) \phi_h - \gamma_1^.d_2(\phi_h)$ has no nonzero components with respect to $t_1,t_2,t_3$.  %= d_1(\gamma_u) t_j - \gamma_1^.d_2(t_j).   $$ 
We only need to show that $$ \beta_1(\sum_{k=1}^{r_1} \langle w^{(3)}_2(\alpha_1(s_1), \alpha_1(s_2), \alpha_1(s_3), e_k ), \phi_h \otimes \gamma_u  \rangle) \epsilon_k)=0. $$
But $\beta_1(\epsilon_k)= \sum_{j=1}^3 u_{kj} t_j $ and $\sum_{k=1}^{r_1} u_{kj} e_k = \alpha_1(s_j) .$ Thus the coefficients of $t_j$ in the above term is 
$$ \langle w^{(3)}_2(\alpha_1(s_1), \alpha_1(s_2), \alpha_1(s_3), \alpha_1(s_j) ), \phi_h \otimes \gamma_u  \rangle. $$ This is zero since we are applying $w^{(3)}_2$ to a wedge product of four elements, two of which are equal.
\end{proof}

\begin{remark}
\label{remarkw2,1}
Proposition \ref{w3,1link} implies that some particular relations are satisfied by the structure maps.
Using the relation $ d_3(s_j^. \phi_h) = d_1(s_j) \phi_h - s_j^.(d_2(\phi_h))$ we obtain 
\begin{equation}
\label{eqw21,1}
\sum_{k=1}^{r_1}  y_{k\rho}  
\langle  e_k^.e_i, \phi_h \rangle =  -\delta_{\rho h} x_i + \sum_{u=1}^{r_3} z_{hu} \langle e_i f_{\rho}, \gamma_u \rangle, 
\end{equation}  
where $\delta_{\rho h}$ denotes the classical Kronecker delta.

Using the relation $ d_3(\gamma_u^. \phi_h) = d_1(\gamma_u) \phi_h - s_j^.(d_2(\phi_h))$ we obtain 
$$ \sum_{k=1}^{r_1}  y_{k\rho}  
\langle  w^{(3)}_2(\alpha_1(s_1), \alpha_1(s_2), \alpha_1(s_3), e_k ), \phi_h \otimes \gamma_1 \rangle =   - \delta_{\rho h}
 \langle \alpha_1(s_1)^.\alpha_1(s_2)^.\alpha_1(s_3), \gamma_1 \rangle+ $$  $$ + \sum_{u=1}^{r_3} z_{hu} \langle w^{(2)}_2(\alpha_1(s_1),\alpha_1(s_2),\alpha_1(s_3) \otimes f_{\rho}), \gamma_1 \wedge\gamma_u \rangle+ \langle \alpha_1(s_3)^.f_{\rho}, \gamma_1 \rangle \cdot \langle \alpha_1(s_1)^.\alpha_1(s_2), \phi_h \rangle+$$
\begin{equation}
\label{eqw21,2}
- \langle \alpha_1(s_2)^.f_{\rho}, \gamma_1 \rangle \cdot \langle \alpha_1(s_1)^.\alpha_1(s_3), \phi_h \rangle +   \langle \alpha_1(s_1)^.f_{\rho}, \gamma_1 \rangle \cdot \langle \alpha_1(s_2)^.\alpha_1(s_3), \phi_h \rangle. 
\end{equation}
Expanding linearly this equation with respect to the coefficients given by the maximal minors of $\alpha_1$ we obtain that the same equality holds replacing one or more $\alpha_1(s_j)$ by generators $e_j$ of $A_1$.
\end{remark}

\medskip

We pass now to identify the formulas for the multiplication map $w^{(1)}_1: \bigwedge^3 D_1 \to D_3.$ This map could be computed using the previous two multiplications and the associativity of the multiplicative structure. However, it is interesting to compute it following the definition.
 For compactness, let us use the notation $\varepsilon_{s_1,s_2, s_3}:= \alpha_1(s_1) \wedge \alpha_1(s_2) \wedge \alpha_1(s_3), $ and
$\varepsilon_{i,s_1,s_2, s_3}:= e_i \wedge \varepsilon_{s_1,s_2, s_3}. $

\begin{thm}
\label{w1,1link}
The multiplication map $ \bigwedge^3 D_1 \to D_3$ is described as follows. 
\begin{equation}
\label{w1,1,sss}
s_1^.s_2^.s_3 =  w = \sum_{i=1}^{r_1}x_i \epsilon_i.
\end{equation}
\begin{equation}
\label{w1,1,gss}
\gamma_1^.s_1^.s_2 = \sum_{i=1}^{r_1} \langle \alpha_1(s_1)^.\alpha_1(s_2)^.e_i, \gamma_1 \rangle \epsilon_i.
\end{equation}
\begin{equation}
\label{w1,1,ggs}
\gamma_1^.\gamma_2^.s_1 = \sum_{i=1}^{r_1} \langle w_{2,1}^{(1)}(\varepsilon_{i,s_1,s_2, s_3} \otimes \alpha_1(s_1)),  \gamma_1 \wedge \gamma_2 \rangle \epsilon_i.
\end{equation}
\begin{equation}
\label{w1,1,ggg}
\gamma_1^.\gamma_2^.\gamma_3 = \sum_{i=1}^{r_1} \langle w^{(1)}_{3,1}(\varepsilon_{i,s_1,s_2, s_3} \otimes \varepsilon_{s_1,s_2, s_3} ), \gamma_1 \wedge \gamma_2 \wedge \gamma_3  \rangle \epsilon_i. 
\end{equation}
The same formulas hold for all the possible combinations of basis elements $\gamma_u$ and $s_j$.
\end{thm}

\begin{proof}
%Denote by $y_{ij}$ the entries of $a_2$, by $x_1, \ldots, x_{r_1}$ the entries of $d_1$, and by $u_{ij}$ the entries of $\beta_1$. 
Observe that $d_3(\epsilon_i)= \sum_{k=1}^{r_2} y_{ik} \phi_k + \sum_{j=1}^3 u_{ij} t_j. $ The proof of (\ref{w1,1,sss}) is straightforward. For the other cases we use similar methods but we deal with each of them separately.
All the computations over a split exact complex with defect variables are postponed to Lemma \ref{splitw1,1}. We recall that 
$d_1(s_j)= b_1(s_j) = \sum_{i=1}^{r_1}u_{ij} x_i$ and $d_1(\gamma_t)= \beta_3(\gamma_t) = \langle \alpha_1(s_1)^.\alpha_1(s_2)^.\alpha_1(s_3), \gamma_1 \rangle. $ \\
% Consider (\ref{w1,1,gss})
\bf Case 1: $\gamma_1^.s_1^.s_2$. \rm \\
Apply $d_3$ to the right side term of (\ref{w1,1,gss}). Call $\Theta $ the obtained element. The coefficient of $\phi_h$ in $\Theta$ is $\sum_{i=2}^{r_1} \langle \alpha_1(s_1)^.\alpha_1(s_2)^.e_i, \gamma_1 \rangle y_{ih}. $ Using the formula to compute the multiplication $\bigwedge^3 D_1 \to D_3$ we need to show that this is equal to the coefficient of $\phi_h$ in $ d_1(\gamma_1)s_1^.s_2 - d_1(s_1) \gamma_1^.s_2 + d_1(s_2) \gamma_1^.s_1.  $
By Proposition \ref{w3,1link}, such coefficient is 
$$  - (\sum_{i=1}^{r_1}u_{i1} x_i)\langle \alpha_1(s_2)^.f_h, \gamma_1 \rangle + (\sum_{i=1}^{r_1}u_{i2} x_i) \langle \alpha_1(s_1)^.f_h, \gamma_1 \rangle. $$
Expanding $\alpha_1(s_j)= \sum_{i=1}^{r_1}u_{ij} e_i$, we reduce to checking that the equation 
$$ \sum_{i=1}^{r_1} \langle e_j^.e_k^.e_i, \gamma_t \rangle y_{ih} = x_k \langle e_j^.f_h, \gamma_t \rangle - x_j \langle e_k^.f_h, \gamma_t \rangle  $$
holds for every choice of indices. By Theorem \ref{equivariantrelations} it is sufficient to check this relation over a split exact complex with defect variables. This is done in relation (W11,1) in Lemma \ref{splitw1,1}. 

The coefficient of $t_j$ in $\Theta$ is $\sum_{i=1}^{r_1} \langle \alpha_1(s_1)^.\alpha_1(s_2)^.e_i, \gamma_1 \rangle u_{ij}$. Using the relation $\sum_{i=1}^{r_1} u_{ij} e_i = \alpha_1(s_j)$, we get that this coefficient is zero if $j=1,2$, while it is equal to $\langle \alpha_1(s_1)^.\alpha_1(s_2)^.\alpha_1(s_3), \gamma_1 \rangle = \beta_3(\gamma_1)$ if $j=3$. Relation $s_1^.s_2= t_3$ implies the thesis. \\
\bf Case 2: $\gamma_1^.\gamma_2^.s_1$. \rm \\
Now call $\Theta $ the image of the right side of (\ref{w1,1,ggs}) after applying $d_3$. We first show that the coefficient of $t_j$ in $\Theta$ is zero for each $j=1,2,3$. Indeed such coefficient is $\sum_{i=1}^{r_1} \langle w_{2,1}^{(1)}(\varepsilon_{i,s_1,s_2, s_3} \otimes \alpha_1(s_1)), \gamma_1\wedge \gamma_2 \rangle u_{ij}$ which is equal to %$$ \sum_{i=1}^{r_1} \langle w_{2,1}^{(1)}(\varepsilon_{i,s_1,s_2, s_3} \otimes \alpha_1(s_1)), \gamma_1\wedge \gamma_2 \rangle u_{ij} = \langle w_{2,1}^{(1)}(\varepsilon_{s_j,s_1,s_2, s_3} \otimes \alpha_1(s_1)), \gamma_1\wedge \gamma_2 \rangle = 0, $$ because $\varepsilon_{s_j,s_1,s_2, s_3} = 0$ .
$$  \langle w_{2,1}^{(1)}(\alpha_1(s_j) \wedge \alpha_1(s_1) \wedge \alpha_1(s_2) \wedge \alpha_1(s_3) \otimes \alpha_1(s_1)), \gamma_1\wedge \gamma_2 \rangle = 0. $$ 
%\alpha_1(s_j) \wedge \alpha_1(s_1) \wedge \alpha_1(s_2) \wedge \alpha_1(s_3)
The coefficient of $\phi_h$ in $\Theta$ is $\sum_{i=1}^{r_1} \langle w_{2,1}^{(1)}(\varepsilon_{i,s_1,s_2, s_3} \otimes \alpha_1(s_1)), \gamma_1\wedge \gamma_2 \rangle y_{ih}$. We have to show that this is equal to the coefficient of $\phi_h$ in $ d_1(\gamma_1)\gamma_2^.s_1 - d_1(\gamma_2)\gamma_1^.s_1 +  d_1(s_1)\gamma_1^.\gamma_2.  $ By Proposition \ref{w3,1link}, this coefficient is 
$$   (\sum_{i=1}^{r_1}u_{i1} x_i)\langle w^{(2)}_2(\alpha_1(s_1),\alpha_1(s_2),\alpha_1(s_3) \otimes f_h), \gamma_1 \wedge \gamma_2 \rangle+ $$ $$ - \langle \alpha_1(s_1)^.\alpha_1(s_2)^.\alpha_1(s_3), \gamma_2  \rangle \cdot \langle \alpha_1(s_1)^.f_h, \gamma_1 \rangle + \langle \alpha_1(s_1)^.\alpha_1(s_2)^.\alpha_1(s_3), \gamma_1  \rangle \cdot \langle \alpha_1(s_1)^.f_h, \gamma_2 \rangle. $$
 By computation with the split exact complex, these two coefficients agree as consequence of relation (W11,2) in Lemma \ref{splitw1,1}. \\ %DONE \ec This implies the thesis. \\
\bf Case 3: $\gamma_1^.\gamma_2^.\gamma_3$. \rm \\
As in the other two cases call $\Theta $ the image of the right side of (\ref{w1,1,ggg}) after applying $d_3$. The coefficient of $t_j$ in $\Theta$ can be shown to be zero for every $j=1,2,3$ as in Case 2. As before we need to compare the coefficient of $\phi_h$ in $\Theta$ and in $ d_1(\gamma_1)\gamma_2^.\gamma_3 - d_1(\gamma_2)\gamma_1^.\gamma_3 +  d_1(\gamma_3)\gamma_1^.\gamma_2.  $ This consists of checking that $\sum_{i=1}^{r_1} \langle w_{3,1}^{(1)}(\varepsilon_{i,s_1,s_2, s_3} \otimes \varepsilon_{s_1,s_2, s_3}), \wedge_{t=1}^3 \gamma_t \rangle y_{ih}$ is equal to 
$$ \langle \alpha_1(s_1)^.\alpha_1(s_2)^.\alpha_1(s_3), \gamma_1  \rangle \cdot \langle w^{(2)}_2(\alpha_1(s_1),\alpha_1(s_2),\alpha_1(s_3) \otimes f_h), \gamma_2 \wedge \gamma_3 \rangle + $$
$$ -\langle \alpha_1(s_1)^.\alpha_1(s_2)^.\alpha_1(s_3), \gamma_2  \rangle \cdot \langle w^{(2)}_2(\alpha_1(s_1),\alpha_1(s_2),\alpha_1(s_3) \otimes f_h), \gamma_1 \wedge \gamma_3 \rangle + $$ 
$$ +\langle \alpha_1(s_1)^.\alpha_1(s_2)^.\alpha_1(s_3), \gamma_3  \rangle \cdot \langle w^{(2)}_2(\alpha_1(s_1),\alpha_1(s_2),\alpha_1(s_3) \otimes f_h), \gamma_1 \wedge \gamma_2 \rangle. $$
The thesis now follows by relation (W11,3) in Lemma \ref{splitw1,1}.
%\bc  TO DO \ec
\end{proof}

\subsection{Higher structure maps in the second graded components}

We deal now with the maps from the second graded components of the critical representations.
The next map we consider is $w^{(3)}_2: \bigwedge^4 D_1 \to D_2 \otimes D_3 $. Recall that this map is computed by lifting the image of the map $q^{(3)}_2$, which is defined on four elements $e_1,e_2,e_3,e_4$ as $e_1^.e_2 \otimes e_3^.e_4 - e_1^.e_3 \otimes e_2^.e_4 + e_1^.e_4 \otimes e_2^.e_3$.

\begin{thm}
\label{w3,2link}
The map $w^{(3)}_2$ on the complex $\DD$ is computed as follows:
\begin{equation}
\label{w3,2,gsss}
w^{(3)}_2(\gamma_1, s_1, s_2, s_3) = \sum_{i,h} \langle e_i^.f_h,  \gamma_1 \rangle (\epsilon_i \otimes \phi_h). 
\end{equation}
\begin{equation}
\label{w3,2,ggss}
w^{(3)}_2(\gamma_1, \gamma_2, s_1, s_2) = \sum_{i,h} \langle w^{(2)}_2(e_i, \alpha_1(s_1), \alpha_1(s_2) \otimes f_h),  \gamma_1 \wedge \gamma_2 \rangle (\epsilon_i \otimes \phi_h).
\end{equation}
\begin{equation}
\label{w3,2,gggs}
w^{(3)}_2(\gamma_1, \gamma_2, \gamma_3, s_1) = \sum_{i,h} \langle w^{(2)}_{3,1}(\varepsilon_{i,s_1,s_2,s_3} \otimes \alpha_1(s_1) \otimes f_h), \wedge_{u=1}^3 \gamma_u \rangle (\epsilon_i \otimes \phi_h). 
\end{equation}
%\begin{equation}
%\label{w3,2,gggs}
%w^{(3)}_2(\gamma_1, \gamma_2, \gamma_3, s_1) = \sum_{i,h} \sum_{j,k,r \neq i} (\alpha_1)_{jkr} \langle w^{(2)}_3(\varepsilon_{ijkr} \otimes \alpha_1(s_1) \otimes f_h), \wedge_{j=1}^3 \gamma_j \rangle (\epsilon_i \otimes \phi_h)+ $$ $$  \rc + \sum_{i,j} \lambda_{ij}(\epsilon_i \otimes t_j). \ec
%\end{equation}
%+ $$ $$  \rc + \sum_{i,j} \lambda_{ij}(\epsilon_i \otimes t_j). \ec
\begin{equation}
\label{w3,2,gggg}
w^{(3)}_2(\gamma_1, \gamma_2, \gamma_3, \gamma_4) = \sum_{i,h} \langle w^{(2)}_{4,1}(\varepsilon_{i,s_1,s_2,s_3} \otimes \varepsilon_{s_1,s_2, s_3} \otimes f_h), \wedge_{u=1}^4 \gamma_u \rangle (\epsilon_i \otimes \phi_h). 
\end{equation}
The same formulas hold for all the possible combinations of basis elements $\gamma_u$ and $s_j$.
\end{thm}

%\rc we are assuming that the terms $\lambda$ in Proposition \ref{w3,1link} are zero, this seems to be true by Lemma \ref{lambdacorrection}. It is always true for minimal linkage -- or linkage where the regular sequence is chosen among multiple of generators  \ec \\
%\bc the last formula has not been checked on examples, the other are checked on example with $r_i= 4$. \ec

%Relation (\ref{rel1}) shows that the maps $  w^{(2)}_3 $ and $ w^{(2)}_4$ in the formulas (\ref{w3,2,gggs}) and (\ref{w3,2,gggg}) are well-defined. Indeed, after a linear change of basis, we can assume $e_i, \alpha_1(s_1), \alpha_1(s_2), \alpha_1(s_3)$ to be the basis of a 4-dimensional subspace of $D_1$.
% after a linear change of basis which takes $e_i, \alpha_1(s_1), \alpha_1(s_2), \alpha_1(s_3)$ as basis of a 4-dimensional subspace of $F_1$.

\begin{proof}
We follow the same method as in the proof of Theorem \ref{w1,1link}.
 %Denote by $y_{ij}$ the entries of $a_2$ and by $u_{ij}$ the entries of $\beta_1$. 
 Let $\widehat{d_3}: D_3 \otimes D_2 \to S_2D_2$ be the map induced by $d_3$ and let $q^{(3)}_2: \bigwedge^4 D_1 \to S_2D_2 $ be defined as above.
Let $\Phi: \bigwedge^4 D_1 \to D_2 \otimes D_3 $ be the linear map defined for each choice of basis elements by taking the opportune right side term of equations (\ref{w3,2,gsss}), (\ref{w3,2,ggss}), (\ref{w3,2,gggs}), (\ref{w3,2,gggg}).
We have to show $\widehat{d_3} \cdot \Phi= q^{(3)}_2$. %that applying $\widehat{d_3}$ to the right side of each of the equations (\ref{w3,2,gsss}), (\ref{w3,2,ggss}), (\ref{w3,2,gggs}), (\ref{w3,2,gggg}) we get the same term obtained by computing the corresponding $q^{(3)}_2$.
%Apply the map $F_3 \otimes F_2 \to S_2F_2$ induced by $d_3$ to the latter term in (\ref{w3,2,gsss}). 
Recall that $d_3(\epsilon_i)= \sum_{k=1}^{r_2} y_{ik} \phi_k + \sum_{j=1}^3 u_{ij} t_j. $ We describe each case separately. For each of them, we use Theorem \ref{equivariantrelations} to reduce the case of a split exact complex with defect variables. The computations over a split exact complex are postponed to Lemma \ref{splitw3,2}. \\
%the required computations over a split exact complex with defect variables are 
%For (\ref{w3,2,gsss}), 
\bf Case 1: $w^{(3)}_2(\gamma_1, s_1, s_2, s_3)$. \rm \\
The coefficient of $(\phi_h \cdot \phi_k)$ in $\widehat{d_3}(\Phi(\gamma_1, s_1, s_2, s_3))$ is equal to 
$ \sum_{i=1}^{r_1} \langle e_i^.f_h,  \gamma_1 \rangle y_{ik} + \langle e_i^.f_k,  \gamma_1 \rangle y_{ih}.  $  This coefficient is zero since so it is over a split exact complex with defect variables (see relation (W32,1) in Lemma \ref{splitw3,2}).
Since $\alpha_1(s_j)= \sum_i u_{ij}e_i$,  the coefficient of $(\phi_h \cdot t_j)$ is 
$ \sum_i \langle e_i^.f_h,  \gamma_1 \rangle u_{ij} = \langle \alpha_1(s_j)^.f_h,  \gamma_1 \rangle. $ Hence, by Proposition \ref{w3,1link} %the term in (\ref{w3,2,gsss}) is mapped to 
$$ \widehat{d_3}(\Phi(\gamma_1, s_1, s_2, s_3)) = \sum_{h=1}^{r_2} \sum_{j=1}^3 \langle \alpha_1(s_j)^.f_h,  \gamma_1 \rangle (\phi_h \otimes t_j)= q^{(3)}_2(\gamma_1, s_1, s_2, s_3). $$ \\ %which, by (\ref{w3,1link}) is equal to $  q^{(3)}_2(\gamma_1, s_1, s_2, s_3) $.
\bf Case 2: $w^{(3)}_2(\gamma_1, \gamma_2, s_1, s_2)$. \rm \\
In this case, the coefficient of $(\phi_h \cdot \phi_k)$ in $\widehat{d_3}(\Phi)$ is
$$ \sum_{i=1}^{r_1} \langle w^{(2)}_2(e_i, \alpha_1(s_1), \alpha_1(s_2) \otimes f_h),  \gamma_1 \wedge \gamma_2 \rangle y_{ik} + \langle w^{(2)}_2(e_i, \alpha_1(s_1), \alpha_1(s_2) \otimes f_k),  \gamma_1 \wedge \gamma_2 \rangle y_{ih}.  $$
By computation on the split exact complex (see relation (W32,2) in Lemma \ref{splitw3,2}), expanding with respect the opportune minors of $\alpha_1$, we obtain that this coefficient is equal to
$$  \langle \alpha_1(s_1)^.f_h,  \gamma_2 \rangle \cdot  \langle \alpha_1(s_2)^.f_k,  \gamma_1 \rangle -  \langle \alpha_1(s_1)^.f_h,  \gamma_1 \rangle \cdot  \langle \alpha_1(s_2)^.f_k,  \gamma_2 \rangle + $$
$$ + \langle \alpha_1(s_1)^.f_k,  \gamma_2 \rangle \cdot  \langle \alpha_1(s_2)^.f_h,  \gamma_1 \rangle -  \langle \alpha_1(s_1)^.f_k,  \gamma_1 \rangle \cdot  \langle \alpha_1(s_2)^.f_h,  \gamma_2 \rangle. $$
The coefficient of $(\phi_h \cdot t_j)$ is equal to 
$$  \sum_i \langle w^{(2)}_2(e_i, \alpha_1(s_1), \alpha_1(s_2) \otimes f_h),  \gamma_1 \wedge \gamma_2 \rangle u_{ij} = \langle w^{(2)}_2(\alpha_1(s_j), \alpha_1(s_1), \alpha_1(s_2) \otimes f_h),  \gamma_1 \wedge \gamma_2 \rangle.  $$
This term is zero for $j=1,2$ and it is equal to % $\sum_{i_1, i_2, i_3} (\alpha_1)_{i_1, i_2, i_3} \langle w^{(2)}_2(e_{i_1}, e_{i_2}, e_{i_3}  \otimes f_h),  \gamma_1 \wedge \gamma_2 \rangle)$ if $j=3$. Computing $q^{(3)}_2(\gamma_1, \gamma_2, s_1, s_2)$ using (\ref{w3,1link}) we get the desired equality. \rc only if $\lambda=0$ \ec   \\
$ \langle w^{(2)}_2(\alpha_1(s_1), \alpha_1(s_2), \alpha_1(s_3)  \otimes f_h),  \gamma_1 \wedge \gamma_2 \rangle$ if $j=3$. Computing $q^{(3)}_2(\gamma_1, \gamma_2, s_1, s_2)$ using Proposition \ref{w3,1link} we get the desired equality. \\
\bf Case 3: $w^{(3)}_2(\gamma_1, \gamma_2, \gamma_3, s_1)$. \rm \\
The coefficient of $(\phi_h \cdot \phi_k)$ in $\widehat{d_3}(\Phi)$ is
$$ \sum_{i=1}^{r_1} \langle w^{(2)}_{3,1}(\varepsilon_{i,s_1,s_2,s_3} \otimes \alpha_1(s_1) \otimes f_h), \wedge_{u=1}^3 \gamma_u \rangle y_{ik} + \langle w^{(2)}_{3,1}(\varepsilon_{i,s_1,s_2,s_3} \otimes \alpha_1(s_1) \otimes f_k), \wedge_{u=1}^3 \gamma_u \rangle y_{ih}.  $$
Set $\vartheta_{ut}^h:= \langle w^{(2)}_2(\alpha_1(s_1), \alpha_1(s_2), \alpha_1(s_3) \otimes f_h),  \gamma_u \wedge \gamma_t \rangle$.
By relation (W32,3) in Lemma \ref{splitw3,2} %if $\lbrace u,s,t \rbrace= \lbrace 1,2,3 \rbrace$, 
this coefficient is equal to
%$$ \sum_{s=1}^3 (-1)^{s+1} \langle w^{(2)}_2(\alpha_1(s_1), \alpha_1(s_2), \alpha_1(s_3) \otimes f_h),  \gamma_u \wedge \gamma_t \rangle \cdot  \langle \alpha_1(s_1)^.f_k,  \gamma_s \rangle + $$ $$ +  \langle w^{(2)}_2(\alpha_1(s_1), \alpha_1(s_2), \alpha_1(s_3) \otimes f_k),  \gamma_u \wedge \gamma_t \rangle \cdot  \langle \alpha_1(s_1)^.f_h,  \gamma_s \rangle.  $$
$$ \vartheta_{12}^h \cdot \langle \alpha_1(s_1)^.f_k,  \gamma_3 \rangle + \vartheta_{12}^k \cdot \langle \alpha_1(s_1)^.f_h,  \gamma_3 \rangle - \vartheta_{13}^h \cdot \langle \alpha_1(s_1)^.f_k,  \gamma_2 \rangle - \vartheta_{13}^k \cdot \langle \alpha_1(s_1)^.f_h,  \gamma_2 \rangle + $$  $$ + \vartheta_{23}^h \cdot \langle \alpha_1(s_1)^.f_k,  \gamma_1 \rangle + \vartheta_{23}^k \cdot \langle \alpha_1(s_1)^.f_h,  \gamma_1 \rangle.  $$
The coefficient of $(\phi_h \cdot t_j)$ is equal to 
$$  \sum_i \langle w^{(2)}_{3,1}(\varepsilon_{i,s_1,s_2,s_3} \otimes \alpha_1(s_1) \otimes f_h), \wedge_{u=1}^3 \gamma_u \rangle u_{ij} = \langle w^{(2)}_{3,1}(\varepsilon_{s_1,s_1,s_2,s_3} \otimes \alpha_1(s_1) \otimes f_h), \wedge_{u=1}^3 \gamma_u \rangle = 0.  $$ Again the thesis follows computing $q^{(3)}_2(\gamma_1, \gamma_2, \gamma_3, s_1)$ using Proposition \ref{w3,1link}.  \\
\bf Case 4: $w^{(3)}_2(\gamma_1, \gamma_2, \gamma_3, \gamma_4)$. \rm \\
Similarly to the previous cases the coefficient of 
$(\phi_h \cdot \phi_k)$ in $\widehat{d_3}(\Phi)$ is
$$ \sum_{i=1}^{r_1} \langle w^{(2)}_{4,1}(\varepsilon_{i,s_1,s_2,s_3} \otimes \varepsilon_{s_1,s_2, s_3} \otimes f_h), \wedge_{u=1}^4 \gamma_u \rangle y_{ik} + \langle w^{(2)}_{4,1}(\varepsilon_{i,s_1,s_2,s_3} \otimes \varepsilon_{s_1,s_2, s_3} \otimes f_k), \wedge_{u=1}^4 \gamma_u \rangle y_{ih}.  $$ By relation (W32,4) in Lemma \ref{splitw3,2} % if $\lbrace u_1,u_2,u_3,u_4 \rbrace= \lbrace 1,2,3,4 \rbrace$, 
this coefficient is equal to
%$$ \sum_{u_1,u_2} (-1)^{u_1+u_2+1} \langle w^{(2)}_2(\alpha_1(s_1), \alpha_1(s_2), \alpha_1(s_3) \otimes f_h),  \gamma_{u_1} \wedge \gamma_{u_2} \rangle \cdot  \langle w^{(2)}_2(\alpha_1(s_1), \alpha_1(s_2), \alpha_1(s_3) \otimes f_k),  \gamma_{u_3} \wedge \gamma_{u_4} \rangle + $$ $$ +  \langle w^{(2)}_2(\alpha_1(s_1), \alpha_1(s_2), \alpha_1(s_3) \otimes f_k),  \gamma_{u_1} \wedge \gamma_{u_2} \rangle \cdot  \langle w^{(2)}_2(\alpha_1(s_1), \alpha_1(s_2), \alpha_1(s_3) \otimes f_h),  \gamma_{u_3} \wedge \gamma_{u_4} \rangle.  $$
$$  \vartheta^h_{12} \cdot \vartheta^k_{34} - \vartheta^h_{13} \cdot \vartheta^k_{24} + \vartheta^h_{14} \cdot \vartheta^k_{23} + \vartheta^h_{23} \cdot \vartheta^k_{14} - \vartheta^h_{24} \cdot \vartheta^k_{13} + \vartheta^h_{34} \cdot \vartheta^k_{12}.   $$
The coefficient of $(\phi_h \cdot t_j)$ is equal to zero for the same reason as in Case 3.
%$$  \sum_i \langle w^{(2)}_{4,1}(\varepsilon_{i,s_1,s_2,s_3} \otimes \hat{e_i} \otimes f_h), \wedge_{u=1}^4 \gamma_u \rangle u_{ij} = \langle w^{(2)}_{4,1}(\varepsilon_{s_1,s_1,s_2,s_3} \otimes \hat{e_i} \otimes f_h), \wedge_{u=1}^4 \gamma_u \rangle = 0.  $$
Therefore $\widehat{d_3} \cdot \Phi= q^{(3)}_2(\gamma_1, \gamma_2, \gamma_3, \gamma_4).$
\end{proof}

\medskip

We conclude this section with the description of the map $w^{(2)}_2$. Recall that this map is computed by lifting the image of the map $q^{(2)}_2$, which is defined on elements $e_1,e_2,e_3, f_h $ as $ e_1^.e_2^.e_3 \otimes f_h +  e_1^.e_2 \otimes e_3^.f_h - e_1^.e_3 \otimes e_2^.f_h + e_2^.e_3 \otimes e_1^.f_h - w^{(3)}_2(e_1,e_2,e_3, d_2(f_h))$.

\begin{thm}
\label{w2,2link}
The map $w^{(2)}_2$ on the complex $\DD$ is computed as follows:
\begin{equation}
\label{w2,2,ssst}
w^{(2)}_2(s_1,s_2,s_3 \otimes t_j) =  0.
\end{equation}
\begin{equation}
\label{w2,2,sssp}
w^{(2)}_2(s_1,s_2,s_3 \otimes \phi_h) = \sum_{i,k} \langle e_i^.e_k, \phi_h \rangle (\epsilon_i \wedge \epsilon_k).
\end{equation}
\begin{equation}
\label{w2,2,gsst}
w^{(2)}_2(\gamma_1, s_1, s_2 \otimes t_j) = \sum_{i,k} (1- \delta_{j3}) \langle \alpha_1(s_{3-j})^.e_i^.e_k, \gamma_1 \rangle (\epsilon_i \wedge \epsilon_k). \ec
\end{equation}
%\sum_{i,k} \langle e_i^.\alpha_2(t_j), \gamma_1 \rangle \langle \alpha_1(s_1)^.\alpha_1(s_2)^.e_k, \gamma_1 \rangle (\epsilon_i \wedge \epsilon_k)
\begin{equation}
\label{w2,2,gssp}
w^{(2)}_2(\gamma_1, s_1, s_2 \otimes \phi_h) = \sum_{i,k} \langle w^{(3)}_2(\varepsilon_{i,k,s_1,s_2}), \phi_h \otimes \gamma_1 \rangle (\epsilon_i \wedge \epsilon_k).
\end{equation}
\begin{equation}
\label{w2,2,ggst}
w^{(2)}_2(\gamma_1, \gamma_2, s_1 \otimes t_j) =   \sum_{i,k} \langle w^{(1)}_{2,1}(\varepsilon_{i,k, s_{k_1},s_{k_2}} \otimes \alpha_1(s_1)), \gamma_1 \wedge \gamma_2 \rangle (\epsilon_i \wedge \epsilon_k), \mbox{ with } k_1,k_2 \neq j.  \ec
\end{equation}
\begin{equation}
\label{w2,2,ggsp}
w^{(2)}_2(\gamma_1, \gamma_2, s_1 \otimes \phi_h) =  \sum_{i,k} \langle w^{(3)}_{3,1}( \varepsilon_{i,k,s_1,s_2,s_3} \otimes \alpha_1(s_1)), \phi_h \otimes \gamma_1 \wedge \gamma_2 \rangle (\epsilon_i \wedge \epsilon_k).  \ec
\end{equation}
\begin{equation}
\label{w2,2,gggt}
w^{(2)}_2(\gamma_1, \gamma_2, \gamma_3 \otimes t_1) = \sum_{i,k} \langle w^{(1)}_{3,1}(\varepsilon_{i,k, s_2,s_3} \otimes \varepsilon_{s_1,s_2, s_3}, \wedge_{u=1}^3 \gamma_u \rangle (\epsilon_i \wedge \epsilon_k).  \ec
\end{equation}
\begin{equation}
\label{w2,2,gggp}
w^{(2)}_2(\gamma_1, \gamma_2, \gamma_3 \otimes \phi_h) =  \sum_{i,k} \langle w^{(3)}_{4,1}(\varepsilon_{i,k,s_1,s_2,s_3} \otimes \varepsilon_{s_1,s_2, s_3}), \phi_h \otimes \wedge_{u=1}^3 \gamma_u \rangle (\epsilon_i \wedge \epsilon_k).  \ec
\end{equation}
The same formulas hold for all the possible combinations of basis elements $\gamma_u$, $s_k$ $\phi_h$, and $t_j$.
%\rc %probably there are other coefficients involved, divide by $ \langle \alpha_1(s_1)^.\alpha_2(t_1), \gamma_1 \rangle  \langle \alpha_1(s_2)^.\alpha_2(t_1), \gamma_1 \rangle$ if they are nonzero? 
%Check for ideal of minors of generic $2 \times 4$ \ec \\
%\bc maps $w^{(3)}_3$ and $w^{(3)}_4$ here are defined on subspaces of $A_1$ of dimension at least 5, as for the format $(1,5,7,3)$. If this is correct these entries are always zero for ideal with $<5$ generators \ec
\end{thm}

\begin{proof}
For (\ref{w2,2,ssst}) simply observe that this map is always zero over a Koszul complex of length 3. To prove the other cases we adopt the procedure used in Theorems \ref{w1,1link} and \ref{w3,2link}. 
% Denote by $z_{ij}, y_{ij}, x_i$ the entries of $a_3, a_2, a_1$ and by $u_{ij}$ the entries of $\beta_1$. 
 Let $\widehat{d_3}: \bigwedge^2 D_3 \to D_2 \otimes D_3$ be the map defined by sending $\epsilon \wedge \epsilon' \mapsto d_3(\epsilon) \otimes \epsilon' - d_3(\epsilon') \otimes \epsilon $. Let $q^{(2)}_2: \bigwedge^3 D_1 \otimes D_2 \to D_2 \otimes D_3 $ be defined as above.
 
Let $\Phi: \bigwedge^3 D_1 \otimes  D_2 \to \bigwedge^2 D_3 $ be the linear map defined for each choice of basis elements by taking the opportune right side term of equations (\ref{w2,2,sssp})-(\ref{w2,2,gggp}).
We have to show $\widehat{d_3} \cdot \Phi= q^{(2)}_2$. %that applying $\widehat{d_3}$ to the right side of each of the equations (\ref{w3,2,gsss}), (\ref{w3,2,ggss}), (\ref{w3,2,gggs}), (\ref{w3,2,gggg}) we get the same term obtained by computing the corresponding $q^{(3)}_2$.
%Apply the map $F_3 \otimes F_2 \to S_2F_2$ induced by $d_3$ to the latter term in (\ref{w3,2,gsss}). 
%Recall that $d_3(\epsilon_i)= \sum_{k=1}^{r_2} y_{ik} \phi_k + \sum_{j=1}^3 u_{ij} t_j. $ 
Notice that $$\widehat{d_3}(\epsilon_i \wedge \epsilon_k) = \sum_{\rho} y_{i \rho} (\epsilon_k \otimes \phi_{\rho}) - y_{k \rho} (\epsilon_i \otimes \phi_{\rho}) + \sum_{j=1}^3 u_{ij} (\epsilon_k \otimes t_j) -u_{kj} (\epsilon_i \otimes t_j).$$
We describe each case separately. Again, the required computations over a split exact complex with defect variables are postponed to Lemma \ref{splitw2,2}. \\
\bf Case 1: $w^{(2)}_2(s_1,s_2,s_3 \otimes \phi_h)$. \rm \\
%Let $\sigma_{ik}$ equal $1$ if $k < i$, $-1$ otherwise.
The coefficient of $(\epsilon_i \otimes \phi_{\rho})$ in $\widehat{d_3}(\Phi)$ is $\sum_{k=1}^{r_1}  y_{k\rho}  
\langle  e_k^.e_i, \phi_h \rangle $. To compute $q^{(2)}_2$ using all the previous results in this section, we recall that $$ w^{(3)}_2(s_1,s_2,s_3, d_2(\phi_h)) = w^{(3)}_2(s_1,s_2,s_3, a_3^*(\phi_h)- \beta_2(\phi_h)) = \sum_{u=1}^{r_3} z_{hu} w^{(3)}_2(s_1,s_2,s_3, \gamma_u)  - 0. $$
Hence, the coefficient of $(\epsilon_i \otimes \phi_{\rho})$ in $q^{(2)}_2(s_1,s_2,s_3 \otimes \phi_h)$ is 
$ \delta_{\rho h} x_i - \sum_{u=1}^{r_3} z_{hu} \langle e_i f_{\rho}, \gamma_u \rangle.  $ Equality follows now by (\ref{eqw21,1}) in Remark \ref{remarkw2,1}.
%\bc These coefficient are checked to be equal over a split exact complex with defect variables. TO TYPE \ec
%\rc IMPORTANT: this is the same relation needed for $s_j^.\phi_h$ in $w^{(2)}_1$, we do not need it over a split exact complex \ec

Similarly, the coefficient of $(\epsilon_i \otimes t_j)$ in $\widehat{d_3}(\Phi)$ is $\sum_{k=1}^{r_1}  u_{kj}
 \langle  e_k^.e_i, \phi_h \rangle = \langle  \alpha_1(s_j)^.e_i, \phi_h \rangle $. This coincides with the coefficient of $(\epsilon_i \otimes t_j)$ in $q^{(2)}_2(s_1,s_2,s_3 \otimes \phi_h)$, that is equal to the coefficient of $\epsilon_i$ in 
 $s_j^. \phi_h$. \\
\bf Case 2: $w^{(2)}_2(\gamma_1, s_1, s_2 \otimes t_j)$. \rm \\
In this case $ q^{(2)}_2 = \gamma_1^.s_1^.s_2 \otimes t_j - \gamma_1^.s_1 \otimes s_2^.t_j + \gamma_1^.s_2 \otimes s_1^.t_j - t_3 \otimes \gamma_1^.t_j + w^{(3)}_2(\gamma_1, s_1,s_2, b_2(t_j)). $ 
If $j=3$, $b_2(t_3) = d_1(\alpha_1(s_1))s_2 - d_1(\alpha_1(s_2))s_1=   \sum_{k=1}^{r_1} x_k (u_{k2}s_1 - u_{k1}s_2).  $ Thus, one can check that $q^{(2)}_2=0$.

In the case $j \neq 3$, let us assume $j=1$ (the case $j=2$ is analogous). The coefficient of $(\epsilon_i \otimes t_p)$ in
$\widehat{d_3}(\Phi)$ is $\sum_{k=1}^{r_1}  u_{kp} \langle  \alpha_1(s_{2})^.e_i^.e_k, \gamma_1 \rangle = \langle  \alpha_1(s_p)^.\alpha_1(s_{2})^.e_i, \gamma_1 \rangle. $ This is zero if $p=2$, it is $ \langle \gamma_1^.s_1^.s_2, \epsilon_i \rangle $ if $p=1$, and is $ \langle \gamma_1^.s_2^.s_3, \epsilon_i \rangle $ if $p=3$. Therefore, in any case it coincides with the coefficient of $(\epsilon_i \otimes t_p)$ in $q^{(2)}_2$.
%\sum_{i,k} \langle \alpha_1(s_{2})^.e_i^.e_k, \gamma_1 \rangle (\epsilon_i \wedge \epsilon_k)

Using the equation $s_1^.t_1 = \sum_{i=1}^{r_1} x_i \epsilon_i$, the coefficient $(\epsilon_i \otimes \phi_{\rho})$ in $ q^{(2)}_2 $ is $$ x_i \langle \alpha_1(s_2)^.f_{\rho}, \gamma_1 \rangle -\sum_{k=1}^{r_1}u_{k2} x_k \langle e_i^.f_{\rho}, \gamma_1 \rangle.   $$
Using the relation $(W11,1)$ in Lemma \ref{splitw1,1}, the coefficient $(\epsilon_i \otimes \phi_{\rho})$ in $\widehat{d_3}(\Phi)$ is $$ \sum_{k=1}^{r_1}  y_{k\rho}  
\langle  \alpha_1(s_2)^.e_i^.e_k, \gamma_1 \rangle = \sum_{l=1}^{r_1} u_{l2} \sum_{k=1}^{r_1}  y_{k\rho} \langle  e_l^.e_i^.e_k, \gamma_1 \rangle =  \sum_{l=1}^{r_1} u_{l2}[ x_i \langle e_l^.f_{\rho}, \gamma_1 \rangle - x_l \langle e_i^.f_{\rho}, \gamma_1 \rangle].  $$ Hence, it coincides with the coefficient in $ q^{(2)}_2 $. \\
\bf Case 3: $w^{(2)}_2(\gamma_1, s_1, s_2 \otimes \phi_h)$. \rm \\ In this case 
$$ w^{(3)}_2(\gamma_1, s_1,s_2, d_2(\phi_h)) = \sum_{u=1}^{r_3} z_{hu} w^{(3)}_2(\gamma_1, s_1,s_2, \gamma_u)  - \langle \alpha_1(s_1)^.\alpha_1(s_2), \phi_h \rangle \cdot w^{(3)}_2(\gamma_1, s_1,s_2, s_3). $$
Hence, the coefficient of $(\epsilon_i \otimes \phi_{\rho})$ in $\widehat{d_3}(\Phi)$ is $\sum_{k=1}^{r_1}  y_{k\rho}  
\langle  w^{(3)}_2(\varepsilon_{i,k,s_1,s_2}), \phi_h \otimes \gamma_1 \rangle $ and in $q^{(2)}_2$ is $$ \delta_{\rho h}
 \langle \alpha_1(s_1)^.\alpha_1(s_2)^.e_i, \gamma_1 \rangle - \sum_{u=1}^{r_3} z_{hu} \langle w^{(2)}_2(e_i,\alpha_1(s_1),\alpha_1(s_2) \otimes f_{\rho}), \gamma_1 \wedge\gamma_u \rangle+   $$
$$+ \langle e_i^.f_{\rho}, \gamma_1 \rangle \cdot \langle \alpha_1(s_1)^.\alpha_1(s_2), \phi_h \rangle+  \langle \alpha_1(s_1)^.f_{\rho}, \gamma_1 \rangle \cdot \langle \alpha_1(s_2)^.e_i, \phi_h \rangle - \langle \alpha_1(s_2)^.f_{\rho}, \gamma_1 \rangle \cdot \langle \alpha_1(s_1)^.e_i, \phi_h \rangle. $$
These coefficients coincide because of equation (\ref{eqw21,2}) in Remark \ref{remarkw2,1}.
%\bc check better \ec
%\rc IMPORTANT This is the same relation used to compute $\gamma_u^.\phi_h$ in $w^{(2)}_1$, after replacing $e_i$ with $\alpha_1(s_3)$, maybe we can get this just by a remark without using the split exact complex \ec
It is easy to check that the coefficient of $(\epsilon_i \otimes t_j)$ in $\widehat{d_3}(\Phi)$ and in $q^{(2)}_2$ is zero if $j=1,2$ and it is equal to $\langle w^{(3)}_2(\varepsilon_{i,s_1,s_2,s_3}), \phi_h \otimes \gamma_1 \rangle = \langle \gamma_1^.\phi_h, \epsilon_i \rangle $ for $j=3$. \\
\bf Case 4: $w^{(2)}_2(\gamma_1, \gamma_2, s_1 \otimes t_j)$. \rm \\
The coefficient with respect to $(\epsilon_i \otimes t_p)$ in $q^{(2)}_2$ comes only from the term $\gamma_1^.\gamma_2^.s_1 \otimes t_j$. This is nonzero only for $j=p$. Using Theorem \ref{w1,1link} and the equality $\sum_{k=1}^{r_1}u_{kp}= \alpha_1(s_p)$, we obtain that this is equal to the coefficient of $(\epsilon_i \otimes t_p)$ in $\widehat{d_3}(\Phi)$. Analyzing the coefficient of $(\epsilon_i \otimes \phi_{\rho})$, the desired result follows by relation $(W11,2)$ in Lemma \ref{splitw1,1}. \\ %if $j \neq 1$, the desired result follows using the relation $(W11,2)$ in Lemma \ref{splitw1,1}, after observing that $\alpha_2(t_j)= \alpha_1(s_{k_1})^.\alpha_1(s_{k_2}).$
%Instead, if $j=1$, we get \rc to finish \ec
\bf Case 5: $w^{(2)}_2(\gamma_1, \gamma_2, s_1 \otimes \phi_h)$. \rm 
Similarly to the previous theorems in this case the coefficient of $(\epsilon_i \otimes t_j)$ in both $\widehat{d_3}(\Phi)$ and $q^{(2)}_2$ is zero. We have 
$$ w^{(3)}_2(\gamma_1, \gamma_2, s_1, d_2(\phi_h)) = \sum_{u=1}^{r_3} z_{hu} w^{(3)}_2(\gamma_1, \gamma_2, s_1, \gamma_u)  - \langle \alpha_1(s_1)^.\alpha_1(s_3), \phi_h \rangle \cdot w^{(3)}_2(\gamma_1, \gamma_2, s_1,s_2) + $$ $$ + \langle \alpha_1(s_1)^.\alpha_1(s_2), \phi_h \rangle \cdot w^{(3)}_2(\gamma_1, \gamma_2, s_1,s_3). $$
Thus, for the coefficient of $(\epsilon_i \otimes \phi_{\rho})$, one has to check over a split exact complex with defect variables that the term
$  \sum_{k=1}^{r_1}  y_{k\rho}  
\langle  w^{(3)}_{3,1}(\varepsilon_{i,k,s_1,s_2,s_3} \otimes \alpha_1(s_1)), \phi_h \otimes \gamma_1\wedge \gamma_2 \rangle $ is equal to the term $$  \delta_{\rho h}
  \langle w_{2,1}^{(1)}(\varepsilon_{i,s_1,s_2, s_3} \otimes \alpha_1(s_1)),  \gamma_1 \wedge \gamma_2 \rangle - \sum_{u=1}^{r_3} z_{hu} \langle w^{(2)}_{3,1}(\varepsilon_{i,s_1,s_2,s_3} \otimes \alpha_1(s_1) \otimes f_{\rho}), \gamma_1 \wedge \gamma_u \wedge \gamma_2 \rangle+   $$
  $$  - \langle \alpha_1(s_1)^.\alpha_1(s_3), \phi_h \rangle \cdot \langle w^{(2)}_2(e_i, \alpha_1(s_1), \alpha_1(s_2)\otimes \phi_{\rho}), \gamma_1 \wedge\gamma_2 \rangle  + $$  $$ +  \langle \alpha_1(s_1)^.\alpha_1(s_2), \phi_h \rangle \cdot \langle w^{(2)}_2(e_i, \alpha_1(s_1), \alpha_1(s_3) \otimes \phi_{\rho}), \gamma_1 \wedge\gamma_2 \rangle + $$
  $$ + \langle  \alpha_1(s_1)^.e_i, \phi_h \rangle \cdot \langle w^{(2)}_2(\alpha_1(s_1), \alpha_1(s_2), \alpha_1(s_3) \otimes \phi_{\rho}) , \gamma_1 \wedge \gamma_2 \rangle  +  $$
$$-  \langle \alpha_1(s_1)^.f_{\rho}, \gamma_1 \rangle \cdot \langle w^{(3)}_2(\varepsilon_{i,s_1,s_2,s_3}, \phi_h \otimes \gamma_2 \rangle + \langle \alpha_1(s_1)^.f_{\rho}, \gamma_2 \rangle \cdot \langle w^{(3)}_2(\varepsilon_{i,s_1,s_2,s_3}, \phi_h \otimes \gamma_1 \rangle.  $$ This is done in relation (W22,5) in Lemma \ref{splitw2,2}. \\
\bf Case 6: $w^{(2)}_2(\gamma_1, \gamma_2, \gamma_3 \otimes t_j)$. \rm \\
The equality condition for the coefficient of $(\epsilon_i \otimes t_p)$ can be done analogously to Case 4.
Analyzing the coefficient of $(\epsilon_i \otimes \phi_{\rho})$, the desired result follows by relation $(W11,3)$ in Lemma \ref{splitw1,1}. \\
%\rc TO DO, check precisely \ec \\
\bf Case 7: $w^{(2)}_2(\gamma_1, \gamma_2, \gamma_3 \otimes \phi_h)$. \rm 
Again by similar arguments the coefficient of $(\epsilon_i \otimes t_j)$ is zero in both $\widehat{d_3}(\Phi)$ and $q^{(2)}_2$. In this case one concludes by checking over a split exact complex that 
$\sum_{k=1}^{r_1}  y_{k\rho}  
\langle  w^{(3)}_{4,1}(\varepsilon_{i,k,s_1,s_2,s_3} \otimes \varepsilon_{s_1,s_2, s_3}), \phi_h \otimes \wedge_{u=1}^3 \gamma_u \rangle$ is equal to
$$ \delta_{\rho h} \langle w^{(1)}_{3,1}(\varepsilon_{i,s_1,s_2, s_3} \otimes \varepsilon_{s_1,s_2, s_3} ), \gamma_1 \wedge \gamma_2 \wedge \gamma_3  \rangle -  \sum_{u=1}^{r_3} z_{hu} \langle w^{(2)}_{4,1}(\varepsilon_{i,s_1,s_2,s_3} \otimes \varepsilon_{s_1,s_2, s_3} \otimes f_{\rho}), \wedge_{s=1}^3 \gamma_s \wedge \gamma_u \rangle+ $$
$$ + \sum_{j=1}^3 (-1)^{j+1} \langle \alpha_1(s_{k_1})^.\alpha_1(s_{k_2}), \phi_h   \rangle \cdot \langle w^{(2)}_{3,1}(\varepsilon_{i,s_1,s_2,s_3} \otimes \alpha_1(s_j) \otimes f_{\rho}), \wedge_{u=1}^3 \gamma_u \rangle+  $$
$$ + \sum_{j=1}^3 (-1)^{j+1} \langle w^{(3)}_2(\varepsilon_{i,s_1,s_2,s_3}, \phi_h \otimes \gamma_j \rangle \cdot \langle w^{(2)}_2(\alpha_1(s_1), \alpha_1(s_2), \alpha_1(s_3) \otimes f_{\rho}) , \gamma_{k_1} \wedge \gamma_{k_2} \rangle.  $$
%$$ + \langle w^{(2)}_2(\alpha_1(s_1), \alpha_1(s_2), \alpha_1(s_3) \otimes \phi_{\rho}) , \gamma_1 \wedge \gamma_2 \rangle \cdot \langle w^{(3)}_2(\varepsilon_{i,s_1,s_2,s_3}, \phi_h \otimes \gamma_3 \rangle +  $$
%$$-  \langle w^{(2)}_2(\alpha_1(s_1), \alpha_1(s_2), \alpha_1(s_3) \otimes \phi_{\rho}) , \gamma_1 \wedge \gamma_3 \rangle \cdot \langle w^{(3)}_2(\varepsilon_{i,s_1,s_2,s_3}, \phi_h \otimes \gamma_2 \rangle + $$
%$$ + \langle w^{(2)}_2(\alpha_1(s_1), \alpha_1(s_2), \alpha_1(s_3) \otimes \phi_{\rho}) , \gamma_2 \wedge \gamma_3 \rangle \cdot \langle w^{(3)}_2(\varepsilon_{i,s_1,s_2,s_3}, \phi_h \otimes \gamma_1 \rangle. $$ 
In the above formula again we have $\lbrace k_1,k_2 \rbrace = \lbrace  1,2,3  \rbrace \setminus \lbrace j \rbrace.$
  This follows from relation (W22,7) in Lemma \ref{splitw2,2}.
\end{proof}

\section{Structure maps and licci ideals}

Now we apply the results from the previous section to linkage, and demonstrate how higher structure maps can detect whether the total Betti number of an ideal decreases after some number of links. For Dynkin formats this has nice consequences related to Conjecture \ref{licciconjecture}.

%to show how when the rank of some structure maps are strictly positive modulo the maximal ideal of a local ring, then one can find . \bc discuss more about licci conjecture \ec

We work over a Gorenstein local ring $R$ with maximal ideal $\mathfrak{m}$ and infinite residue field $K$. As before, we assume $2,3\notin \mf{m}$. Let $I=(x_1, \ldots, x_n)$ be a perfect ideal of height 3 and denote by $\beta(I)$ the sum of Betti numbers of $I$. We denote by $w^{(i)}_{j,k}(I)$ some choice of higher structure maps associated to a minimal free resolution of $I$. 

Given an ideal $J$, minimally linked to $I$, it is well-known that $\beta(J) \leq \beta(I)$. This inequality may be strict, depending on the multiplicative structure of the free resolution of $I$. The ranks of the linear maps $w^{(3)}_1(I)$ and $w^{(2)}_1(I)$ modulo the maximal ideal $\mathfrak{m}$ are fundamental invariants of $I$ playing a role in its linkage properties. It is well-known (see for instance \cite[Equation 1.8]{AKM88}) that if $w^{(3)}_1(I) \otimes K \neq 0$, then it is possible to find a maximal regular sequence $\mathfrak{a} \subseteq I$, such that, setting $J:= (\mathfrak{a}):I$, then $\beta(J) < \beta(I)$. Notice that the rank of the structure maps modulo the maximal ideal $\mathfrak{m}$ does not depend on the particular choice of the lifts.

Similarly, by \cite[Section 3]{CVWlinkage}, if $w^{(2)}_1(I) \otimes K \neq 0$, it is possible to find an ideal $J$, minimally linked to $I$ such that also $w^{(3)}_1(J) \otimes K \neq 0$. Therefore in this case, with at most two links, one can find an ideal $H$ in the same linkage class of $I$, such that $\beta(H) < \beta(I)$.

%Let $I \sim J$ be two minimally linked perfect ideals of height 3. Setting $I=(x_1, \ldots, x_n)$, we suppose the regular sequence $\al$ defining the link $I \sim J$ to be equal to $x_1, x_2, x_3$. Denote by $\beta(I)$ the sum of the total Betti numbers of $I$. Hence, it is well-know that $\beta(J) \leq \beta(I)$, and the inequality is strict if and only if at least one among $\rank(w^{(3)}_1(I) \otimes K)$ and $ \rank(w^{(2)}_1(I) \otimes K) $ is nonzero.

\subsection{Conditions on the maps $w^{(3)}_2$ and $w^{(2)}_2$}

Now we show that also if either $w^{(3)}_2(I)$ or $ w^{(2)}_2(I) $ is nonzero modulo $\mathfrak{m}$, then there exists an ideal $H$ in the linkage class of $I$, such that $\beta(H) < \beta(I)$. Following the notation of the preceding section, if the regular sequence $\mathfrak{a} = \lbrace x_1, x_2, x_3 \rbrace $, then $\alpha_1(s_j)= e_j$ for $j=1,2,3$.

\begin{thm}
\label{licciW2}
Let $I$ be a perfect ideal of height 3. Suppose that either $w^{(3)}_2(I) $ or $w^{(2)}_2(I) $ is nonzero modulo $\mathfrak{m}$. Then there exists an ideal $H$, in the linkage class of $I$ such that $\beta(H) < \beta(I)$. In particular $H$ can be obtained from $I$ with at most 3 links.
\end{thm}

\begin{proof}
As in the previous section, denote by $\A$ the minimal free resolution of $I$.
Assume $w^{(3)}_2(I) \otimes K \neq 0$. Hence there exist generators $e_{i_1}, e_{i_2}, e_{i_3}, e_{i_4} \in A_1$, $\phi_h \in A_2^*$, $\gamma_u \in A_3^*$ such that $\langle w^{(3)}_2(e_{i_1}, e_{i_2}, e_{i_3}, e_{i_4}), \gamma_u \otimes \phi_h \rangle$ is a unit in $R$. 

Since $R$ is local and its residue field is infinite, up to change set of generators for $I$ and using a standard argument as in \cite[Appendix A.5]{CVWlinkage}, we can say that $x_{i_1}, x_{i_2}, x_{i_3}$ form a regular sequence. %If this was not the case, first assume $x_{i_1}, x_{i_2}$ to be a maximal regular sequence in $\lbrace x_{i_1}, x_{i_2}, x_{i_3}, x_{i_4} \rbrace$. By possibly changing basis, there exists $k$ distinct from $i_3, i_4$ such that $x_{i_1}, x_{i_2}, x_{k}$ is a regular sequence. If $\langle w^{(3)}_2(e_{i_1}, e_{i_2}, e_k, e_{i_4}), \gamma_u \otimes \phi_h \rangle$ is also a unit we simply replace $e_{i_3}$ by $ e_{k} $. Otherwise, it is possible to find some nonzero unit $u_k $ such that $x_{i_1}, x_{i_2}, x_{i_3} + u_k x_k$ is also a regular sequence.
%Since $w^{(3)}_2$ is a linear map and $R$ is local, we have that $\langle w^{(3)}_2(e_{i_1}, e_{i_2}, e_{i_3} + u_k e_k, e_{i_4}), \gamma_u \otimes \phi_h \rangle$ is equal to
%$$    \langle w^{(3)}_2(e_{i_1}, e_{i_2}, e_{i_3}, e_{i_4}), \gamma_u \otimes \phi_h \rangle + u_k \langle w^{(3)}_2(e_{i_1}, e_{i_2}, e_k, e_{i_4}), \gamma_u \otimes \phi_h \rangle \in R \setminus \mathfrak{m}. $$  
%Similarly, if $ x_{i_1} $ is a maximal regular sequence in $\lbrace x_{i_1}, x_{i_2}, x_{i_3}, x_{i_4} \rbrace$, we find a regular sequence of minimal generators $x_{i_1}, x_{k_2}, x_{k_3}$ and replace $x_{i_2}, x_{i_3}$ by $x_{i_2}+ u_2 x_{k_2}, x_{i_3}+ u_3 x_{k_3}$ choosing $u_2,u_3$ such that $x_{i_1}, x_{i_2}+ u_2 x_{k_2}, x_{i_3}+ u_3 x_{k_3}$ is a regular sequence. 

Therefore, up to a change of basis, we may assume that $x_1,x_2,x_3$ is a regular sequence and $\langle w^{(3)}_2(e_{1}, e_{2}, e_{3}, e_{4}), \gamma_u \otimes \phi_h \rangle$ is a unit in $R$. Let $J:= (x_1,x_2,x_3) : I$. Computing the multiplicative structure on the free resolution of $J$, by equation (\ref{w2,1gp}) in Proposition \ref{w3,1link}, the coefficient with respect to $\epsilon_4$ of $\gamma_u^.\phi_h$ is a unit. Hence $w^{(2)}_1(J)$ is nonzero modulo $\mathfrak{m}$. This implies that there exists an ideal $H$, obtained by linking from $J$ in at most 2 links, such that $\beta(H) < \beta(J) \leq \beta(I)$.

In the case $w^{(2)}_2(I) \otimes K \neq 0$, using an analogous argument as above we can change basis of generators to assume that $x_1,x_2,x_3$ is a regular sequence and $\langle w^{(2)}_2(e_{1}, e_{2}, e_{3} \otimes f_h), \gamma_u \wedge \gamma_t \rangle$ is a unit in $R$. Setting again $J:= (x_1,x_2,x_3) : I$, by equation (\ref{eqlambda}) in Proposition \ref{w3,1link}, the coefficient with respect to $\phi_h$ of $\gamma_u^.\gamma_t$ is a unit. Thus $w^{(3)}_1(J)$ is nonzero modulo $\mathfrak{m}$. Again this implies the existence of the required ideal $H$. 
\end{proof}

%By a similar argument, 
Combining Theorems \ref{w3,2link} and \ref{w2,2link} with Theorem \ref{licciW2} we obtain the following corollary.

\begin{cor}
Let $I$ be a perfect ideal of height 3. Suppose that at least one map among $w^{(2)}_{3,1}(I)$, $w^{(2)}_{4,1}(I)$, $w^{(1)}_1(I)$, $w^{(1)}_{2,1}(I)$, $w^{(1)}_{3,1}(I)$, $w^{(3)}_{3,1}(I)$, $w^{(3)}_{4,1}(I)$ is nonzero modulo the maximal ideal of $R$. Then there exists an ideal $H$ in the linkage class of $I$, such that $\beta(H) < \beta(I)$.
\end{cor}

%From this, it is natural to propose next conjecture.

\begin{remark}
\label{remarkcycles}
 Over an acyclic complex $\FF$ of length 3, all the structure maps in $W(d_3)$ are maps between Schur functors of the form $w^{(3)}_{s,k}: S_{\boldsymbol{\lambda}}F_1 \to F_2 \otimes S_{\boldsymbol{\mu}}F_3$ where $\boldsymbol{\lambda}$ is a partition of an even integer $2s$ and $ \boldsymbol{\mu} $ is a partition of $s-1$. Similarly the structure maps in $W(d_2)$ are of the form $w^{(2)}_{s,k}:  S_{\boldsymbol{\lambda}}F_1 \otimes F_2 \to S_{\boldsymbol{\mu}}F_3$ where $\boldsymbol{\lambda}$ is a partition of an odd integer $2s-1$ and $ \boldsymbol{\mu} $ is a partition of $s$. All the structure maps that we have been able to compute until now can be obtained by lifting a cycle in some acyclic complex related to Schur complexes (for a treatment of Schur functors and Schur complexes see \cite{ABW82}). %However, 
It is still an open question whether this is true in general for all higher structure maps.
 \end{remark}

Looking at the preceding remark, we conjecture that the pattern we see in all the result of the previous section continues for all the higher structure maps.
Using our notation with complexes $\A$ and $\DD$, we ask whether in general the following situation occurs. Let $\tau $ be a generator of $S_{\boldsymbol{\lambda}}D_1$ involving %the largest possible number of 
some elements of the form $\gamma_u$, let $\boldsymbol{\lambda'}$ be a subpartition of $\boldsymbol{\lambda}$ corresponding to the position of all these elements $\gamma_u$ and let $\tau'$ be the corresponding element in $S_{\boldsymbol{\lambda'}}A_3^*$. %involved in $\varepsilon $ 
(Ex. if $\boldsymbol{\lambda}= (2,2,2,1)$, $\tau = \gamma_1 \wedge \gamma_2 \wedge s_1 \wedge s_2 \otimes \gamma_1 \wedge \gamma_2 \wedge s_1 $ then $\boldsymbol{\lambda'}=(2,2)$ and $\tau'= \gamma_1 \wedge \gamma_2 \otimes \gamma_1 \wedge \gamma_2 $). Say that $\boldsymbol{\lambda}$ is a partition of $m$ and $\boldsymbol{\lambda'}$ is a partition of $t \leq m$.
Then if $m= 2s$,
$$  w^{(3)}_{s,k}(\tau) = \sum_{\zeta \in S_{\boldsymbol{\mu}}D_3, \atop \scriptstyle \phi_h \in A_2^*} \langle w^{(2)}_{t,l}( e(\zeta^*, s_1, s_2, s_3)  \otimes f_h ), \tau' \rangle (\zeta \otimes \phi_h),  $$ while if $m= 2s-1$
$$  w^{(2)}_{s,k}(\tau \otimes \phi_h) = \sum_{\zeta \in S_{\boldsymbol{\mu}}D_3, \atop \scriptstyle \phi_h \in A_2^*} \langle w^{(3)}_{t+1,l}(e(\zeta^*, s_1, s_2, s_3)), \tau' \otimes \phi_h \rangle \zeta.  $$
In the above formulas $e(\zeta^*, s_1, s_2, s_3)$ is a generator of the source of the appropriate map, defined over the complex $\A$ and depending on the integers $k,l$ and on $ \zeta^*, \alpha(s_1), \alpha_1(s_2), \alpha_1(s_3).$ 

Similar relations are expected to hold also for the maps in $W(d_1)$. In particular we expect the maps in $W(d_1)$ to be crucial in determining whether a perfect ideal of height 3 is licci (one motivation for this is the fact that $w^{(1)}_1(I)$ is nonzero modulo the maximal ideal if and only if $I$ is a complete intersection).
 %and $S_{\boldsymbol{\lambda}+1}F_1 \neq 0$

To prove these formulas in general one would need to know the definition of each arbitrary structure map and to check their relations performing the required computation for a split exact complex. This is computationally very hard already for higher maps in the formats $E_7$ and $E_8$. We hope that different approaches, possibly using methods related to the representation theory of the generic ring, may help towards a solution of this problem.
As a consequence of the pattern observed above we state here the following conjecture: %which is a stronger version of the Conjecture \ref{licciconjecture}.

\begin{conjecture}
\label{conj1} \rm
%A perfect ideal $I$ of grade 3 in a regular local ring $R$ is licci if and only if some structure map $w^{(i)}_{j,k}(I)$, for $i=1,2,3$, $j \geq 1$ has strictly positive rank modulo the maximal ideal of $R$.
Let $I$ be a perfect ideal of height 3 in a Gorenstein local ring $R$ with infinite residue field. Then the following are equivalent:
\begin{enumerate}
\item $I$ is licci.
\item For every ideal $J$ in the linkage class of $I$ there exist some structure map $w^{(i)}_{j,k}(J)$, with $i=1,2,3$, $j \geq 1$ which is nonzero modulo the maximal ideal of $R$.
%\item There exists $N$ such that all the maps of the form $w^{(i)}_{j,k}(I)$, for $i=1,2,3$, $j \geq N$ are identically zero.
\item There exists some structure map $w^{(1)}_{j,k}(I)$, $j \geq 1$ which is nonzero modulo the maximal ideal of $R$.
\end{enumerate}
\end{conjecture}

%\bf describe the possible general pattern \rm

\subsection{Free resolutions of format $(1,5,6,2)$}

%\rc this section is just my guess, $w^{(3)}_3$ and $w^{(2)}_3$ should work in case $\rank A_3 = 2$ \ec

In this subsection we deal with free resolutions of perfect ideals of format $(1,5,6,2)$. Let $\FF$ be an arbitrary acyclic complex of this format.
Looking back to Section 2 we notice that there are two unique top components in $W(d_3)$ and $W(d_2)$ which are respectively the maps $w^{(3)}_3:= w^{(3)}_{3,1} : \bigwedge^5 F_1 \otimes F_1 \to F_2 \otimes \bigwedge^2 F_3 $ and $w^{(2)}_{3}:= w^{(2)}_{3,2}: \bigwedge^5 F_1 \otimes F_2 \to S_{2,1}F_3 \cong \bigwedge^2 F_3 \otimes F_3$. 

%To define such maps we work over an arbitrary acyclic complex $\FF$ of format $(1,5,6,2)$. 
The map $w^{(3)}_3(e_1)$ is defined by lifting the term 
% $$q^{(3)}_3(e_1):= w^{(3)}_2(e_1, e_2, e_3, e_4) \otimes e_1^.e_5 - w^{(3)}_2(e_1, e_2, e_3, e_5) \otimes e_1^.e_4 + $$ $$+ w^{(3)}_2(e_1, e_2, e_4, e_5) \otimes e_1^.e_3 - w^{(3)}_2(e_1, e_3, e_4, e_5) \otimes e_1^.e_2 \in F_3 \otimes F_2 \otimes F_2,      $$ along the map $\bigwedge^2 F_3 \to F_3 \otimes F_2 \otimes F_2$.
$$q^{(3)}_3(e_1):= w^{(3)}_2(\hat{e_5}) \otimes e_1^.e_5 - w^{(3)}_2(\hat{e_4}) \otimes e_1^.e_4 +  w^{(3)}_2(\hat{e_3}) \otimes e_1^.e_3 - w^{(3)}_2(\hat{e_3}) \otimes e_1^.e_2 \in F_3 \otimes F_2 \otimes F_2      $$ along the map $\bigwedge^2 F_3 \otimes F_2 \to F_3 \otimes F_2 \otimes F_2$.

%as in Section 2. %Since $D_1$ has rank 5 the term $q^{(3)}_3$ is simpler
To define the map $w^{(2)}_{3}$, we first need to define $w^{(1)}_{2,2}: \bigwedge^5 F_1 \to S_2F_3$. This can be obtained simply by lifting the cycle $q^{(1)}_{2,2}: \sum_{1 \leq i<j \leq 5} (-1)^{i+j+1} e_i^.e_j \otimes w^{(1)}_1(\hat{e_i},\hat{e_j}) \in F_2 \otimes F_3. $

The map $w^{(2)}_{3}(f_h)$ is defined by lifting the term 
$$ q^{(2)}_3(f_h):=  e_1^.f_h \otimes w^{(3)}_2(\hat{e_1}) - e_2^.f_h \otimes w^{(3)}_2(\hat{e_2}) + e_3^.f_h \otimes w^{(3)}_2(\hat{e_3}) -e_4^.f_h \otimes w^{(3)}_2(\hat{e_4})+ $$  
$$ + e_5^.f_h \otimes w^{(3)}_2(\hat{e_5}) - f_h \otimes \frac{1}{2}w^{(1)}_{2,2}(\varepsilon) \in S_2F_3 \otimes F_2 $$ along the map $ \bigwedge^2F_3 \otimes F_3 \to S_2F_3 \otimes F_3$ (notice that the terms in $F_3 \otimes F_3 \otimes F_2$ are sent to $S_2F_3 \otimes F_2$ by symmetrization).
 %\begin{equation}
 %\label{prova1}
%\begin{matrix}
% 0&\rightarrow&S_{2,1}F_3&\rightarrow&F_3\otimes F_3\otimes F_2&\rightarrow& F_3\otimes F_2\otimes F_2\\ 
%&&&\nwarrow\ w^{(2)}_3&\uparrow\ q_3^{(2)}&&\\
%&&&&F_2\otimes \bigwedge^5F_1&&&&
%\end{matrix}
%\end{equation}
%where the upper row is the complex $S_{2,1}d_3$ and $q_3^{(2)}$ is obtained as linear combination of $ v^{(2)}_1v^{(3)}_2$ and $F_2\otimes v^{(1)}_{2,1}$.

%In this case $S_{2,1}F_3 \cong \bigwedge^2F_3 \otimes F_3$.
%In the complex (\ref{prova1}), $(a \wedge b) \otimes c \to c \otimes (b \otimes d(a) - a \otimes d(b))$.
%We have $$ q(h)=  e_1^.f_h \otimes v^{(3)}_2(2345) - e_2^.f_h \otimes v^{(3)}_2(1345) + e_3^.f_h \otimes v^{(3)}_2(1245) -e_4^.f_h \otimes v^{(3)}_2(1235)+ $$  
%$$ + e_5^.f_h \otimes v^{(3)}_2(1234) - f_h \otimes \frac{1}{2}v^{(1)}_{2,1}(\varepsilon). $$

\medskip

Assume now $I$ to be a perfect ideal in $R$ having minimal free resolution $\A$ of format $(1,5,6,2)$, and define $J$ and its free resolution $\DD$ as done before in this paper.
In this case, the basis of $D_1$ can be chosen to be $\lbrace \gamma_1, \gamma_2, s_1, s_2, s_3 \rbrace$. 
Whenever $\beta(J) = \beta(I)$ (if the linkage is minimal and the total Betti number does not decrease), then also the minimal free resolution of $J$ has format $(1,5,6,2)$. 
%  Notice that if we consider minimal linkage for this format and 
In this case, setting $\mathfrak{a}=(x_1,x_2,x_3)$, we can choose the basis of $D_2$ equal to $\lbrace \phi_4, \phi_5, \phi_6, t_1, t_2, t_3 \rbrace$ and the basis of $D_3$ equal to $\lbrace \epsilon_4, \epsilon_5 \rbrace$. Denote by $\varepsilon$ the element $\gamma_1 \wedge \gamma_2 \wedge s_1 \wedge s_2 \wedge s_3$.

\begin{thm}
\label{w3,3link}
The map $w^{(3)}_3$ on the complex $\DD$ is computed as follows:
\begin{equation}
\label{w3,3,s}
w^{(3)}_3(\varepsilon \otimes s_j) = \sum_{i,k,h} \langle w^{(2)}_2(e_i, e_k, \alpha_1(s_j) \otimes f_h ),  \gamma_1 \wedge \gamma_2 \rangle (\epsilon_i \wedge \epsilon_k \otimes \phi_h). 
\end{equation}
\begin{equation}
\label{w3,3,g}
w^{(3)}_3(\varepsilon \otimes \gamma_u) = \sum_{i,k,h} \langle w^{(2)}_{3}(\varepsilon_{i,k,s_1,s_2,s_3} \otimes f_h ),  \gamma_1 \wedge \gamma_2 \otimes \gamma_u \rangle (\epsilon_i \wedge \epsilon_k \otimes \phi_h).
\end{equation}
%The same formulas hold for all the possible combinations of indices of $\gamma_u$ and $s_j$.
\end{thm}

\begin{proof}
For simplicity we consider only $w^{(3)}_3(s_1)$ and $w^{(3)}_3(\gamma_1)$. Define $\Theta$ as the right side of (\ref{w3,3,s}) setting $j=1$ and define $\Phi$ as the right side of (\ref{w3,3,g}) setting $u=1$.
%Set $$ \Theta:= \sum_{i,k,h} \langle w^{(2)}_2(e_i, e_k, \alpha_1(s_1) \otimes f_h ),  \gamma_1 \wedge \gamma_2 \rangle (\epsilon_i \wedge \epsilon_k \otimes \phi_h).$$ 
Call $\widetilde{d_3}$ the map $\bigwedge^2 F_3 \otimes F_2 \to F_3 \otimes S_2F_2$ induced by $d_3$. We have to show that $\widetilde{d_3}(\Theta)= q^{(3)}_3(s_1)$ and $\widetilde{d_3}(\Phi)= q^{(3)}_3(\gamma_1)$.
Compute the map $q^{(3)}_3$ using Theorem \ref{w3,2link} and Proposition \ref{w3,1link}. 
We know that 
$$ q^{(3)}_3(s_1) = w^{(3)}_2(\gamma_1, \gamma_2, s_1, s_2) \otimes s_1^.s_3 -  w^{(3)}_2(\gamma_1, \gamma_2, s_1, s_3) \otimes s_1^.s_2 + $$ $$ + w^{(3)}_2(\gamma_1, s_1, s_2, s_3) \otimes s_1^.\gamma_2 - w^{(3)}_2( \gamma_2, s_1, s_2, s_3) \otimes s_1^.\gamma_1. $$ 
The coefficient of $\epsilon_i \otimes \phi_h^.t_p$ in $q^{(3)}_3(s_1)$ is $-\langle w^{(2)}_2(e_1, \alpha_1(s_1), \alpha_1(s_p) \otimes f_h), \gamma_1 \wedge \gamma_2 \rangle.  $
The coefficient of $\epsilon_i \otimes \phi_h^.t_p$ in $\widetilde{d_3}(\Theta)$ is $$\sum_{k=1}^{r_1} u_{kp} \langle w^{(2)}_2(e_i, e_k, \alpha_1(s_1) \otimes f_h ),  \gamma_1 \wedge \gamma_2 \rangle = \langle w^{(2)}_2(e_1, \alpha_1(s_p), \alpha_1(s_1) \otimes f_h), \gamma_1 \wedge \gamma_2 \rangle. $$ Therefore they coincide.
The coefficient of $\epsilon_i \otimes \phi_h^.\phi_{\rho}$ in $q^{(3)}_3(s_1)$ is 
$$  \langle \alpha_1(s_1)^.f_{\rho},  \gamma_1 \rangle \cdot  \langle e_i^.f_h,  \gamma_2 \rangle -  \langle \alpha_1(s_1)^.f_{\rho},  \gamma_2 \rangle \cdot  \langle e_i^.f_h,  \gamma_1 \rangle + $$
$$ + \langle \alpha_1(s_1)^.f_h,  \gamma_1 \rangle \cdot  \langle e_i^.f_{\rho},  \gamma_2 \rangle -  \langle \alpha_1(s_1)^.f_h,  \gamma_2 \rangle \cdot  \langle e_i^.f_{\rho},  \gamma_1 \rangle. $$
The same one in $\widetilde{d_3}(\Theta)$ is 
$$ \sum_{k=1}^{r_1} \langle w^{(2)}_2(e_i, \alpha_1(s_1), e_k \otimes f_h),  \gamma_1 \wedge \gamma_2 \rangle y_{k \rho} + \langle w^{(2)}_2(e_i, \alpha_1(s_1), e_k \otimes f_{\rho}),  \gamma_1 \wedge \gamma_2 \rangle y_{kh}.  $$
These coefficients agree as consequence of relation $(W32,2)$ in Lemma \ref{splitw3,2}.

Similarly, $$ q^{(3)}_3(\gamma_1) = w^{(3)}_2(\gamma_1, \gamma_2, s_1, s_2) \otimes \gamma_1^.s_3 -  w^{(3)}_2(\gamma_1, \gamma_2, s_1, s_3) \otimes \gamma_1^.s_2 + $$ $$ + w^{(3)}_2(\gamma_1, \gamma_2, s_2, s_3) \otimes \gamma_1^.s_1 - w^{(3)}_2( \gamma_1, s_1, s_2, s_3) \otimes \gamma_1^.\gamma_2. $$ 
The coefficient of $\epsilon_i \otimes \phi_h^.t_p$ is zero both in $q^{(3)}_3(\gamma_1)$ and in $\widetilde{d_3}(\Phi)$. 
Set $$ \psi^h_{jp}:= \langle w^{(2)}_2(e_i, \alpha_1(s_j), \alpha_1(s_p), \otimes f_h),  \gamma_1 \wedge \gamma_2 \rangle, \quad \psi^h_{123}:= \langle w^{(2)}_2(\alpha_1(s_1), \alpha_1(s_2), \alpha_1(s_3), \otimes f_h),  \gamma_1 \wedge \gamma_2 \rangle. $$
%Set $\psi^h_{jp}:= \langle w^{(2)}_2(e_i, \alpha_1(s_j), \alpha_1(s_p), \otimes f_h),  \gamma_1 \wedge \gamma_2 \rangle$ and $\psi^h_{123}:= \langle w^{(2)}_2(\alpha_1(s_1), \alpha_1(s_2), \alpha_1(s_3), \otimes f_h),  \gamma_1 \wedge \gamma_2 \rangle$.
 To conclude, we need to check the equality for the coefficients of $\epsilon_i \otimes \phi_h^.\phi_{\rho}$. Thus we have to check that
$$ \sum_{k=1}^{r_1} \langle w^{(2)}_3(\varepsilon_{i,k,s_1,s_2,s_3} \otimes f_h ),  \gamma_1 \wedge \gamma_2 \otimes \gamma_1 \rangle y_{k \rho} + \langle w^{(2)}_3(\varepsilon_{i,k,s_1,s_2,s_3} \otimes f_{\rho} ),  \gamma_1 \wedge \gamma_2 \otimes \gamma_1 \rangle y_{kh}  $$ is equal to
%$$ \sum_{s=1}^3 (-1)^{s+1} \langle w^{(2)}_2(\alpha_1(s_1), \alpha_1(s_2), \alpha_1(s_3) \otimes f_h),  \gamma_u \wedge \gamma_t \rangle \cdot  \langle \alpha_1(s_1)^.f_k,  \gamma_s \rangle + $$ $$ +  \langle w^{(2)}_2(\alpha_1(s_1), \alpha_1(s_2), \alpha_1(s_3) \otimes f_k),  \gamma_u \wedge \gamma_t \rangle \cdot  \langle \alpha_1(s_1)^.f_h,  \gamma_s \rangle.  $$
$$ \psi_{12}^h \cdot \langle \alpha_1(s_3)^.f_{\rho},  \gamma_1 \rangle + \psi_{12}^{\rho} \cdot \langle \alpha_1(s_3)^.f_h,  \gamma_1 \rangle - \psi_{13}^h \cdot \langle \alpha_1(s_2)^.f_{\rho},  \gamma_1 \rangle - \psi_{13}^{\rho} \cdot \langle \alpha_1(s_2)^.f_h,  \gamma_1 \rangle + $$  $$ + \psi_{23}^h \cdot \langle \alpha_1(s_1)^.f_{\rho},  \gamma_1 \rangle + \psi_{23}^{\rho} \cdot \langle \alpha_1(s_1)^.f_h,  \gamma_1 \rangle - \psi_{123}^h \cdot \langle e_i^.f_{\rho},  \gamma_1 \rangle + \psi_{123}^{\rho} \cdot \langle e_i^.f_h,  \gamma_1 \rangle.  $$
The computation is performed over a split exact complex with defect variables in Lemma \ref{splitw3,3}. %\bc This seems to be very similar to the computation required for W32,3 (here you do not change $\gamma$, there you do not change $\alpha_1(s_j)$), maybe they could be done together \ec
\end{proof}

%\begin{thm}
%\label{w2,3linkE6}
%\bc this seems to be not needed for the licci conjecture 1562, remove from here and do it in $E_7$ \ec
%The map $w^{(2)}_3$ on the complex $\DD$ is computed as follows:
%\begin{equation}
%\label{w2,3,t}
%w^{(2)}_3(t_p) = \rc  \sum_{i,j} \sum_{k = i,j} \langle w^{(1)}_{2,1}(\varepsilon_{i,j,s_{k_1},s_{k_1}} \otimes e_k),  \gamma_1 \wedge \gamma_2 \rangle (\epsilon_i \wedge \epsilon_j \otimes \epsilon_k), \mbox{ with } k_1,k_2 \neq p. \ec
%\end{equation}
%\begin{equation}
%\label{w2,3,p}
%w^{(2)}_3(\phi_h) = \sum_{i,j} \sum_{k = i,j} \langle w^{(3)}_3(\varepsilon_{i,j,s_1,s_2,s_3} \otimes e_k),  \gamma_1 \wedge \gamma_2 \otimes \phi_h \rangle (\epsilon_i \wedge \epsilon_j \otimes \epsilon_k).
%\end{equation}
%The same formulas hold for all the possible combinations of indices of $\gamma_u$ and $s_j$.
%\bc is there a more general version? \ec
%\end{thm}

%\begin{proof}

%\end{proof}

%\begin{remark}
%\label{Ftop}
%\end{remark}

We are now able to prove that, if $I$ is perfect with minimal free resolution of format $(1,5,6,2)$, then $I$ is licci if and only if some structure map in $W(d_2)$ or in $W(d_3)$ is nonzero modulo the maximal ideal of $R$. 

For this we recall that any perfect ideal $J$ of height 3 such that $\beta(J) < \beta(I)$ is licci. Indeed, any such $J$ is either Gorenstein or almost complete intersection. Gorenstein ideals of height 3 are proved to be licci in \cite{Watanabe}.
Any almost complete intersection is minimally linked to a Gorenstein ideal, hence those of height 3 are also licci (see \cite{kunz}, \cite{Peskine-Szpiro-linkage}). 

\begin{thm}
\label{licci1562}
Let $R$ be a Gorenstein local ring with maximal ideal $\m$ and infinite residue field $K$.
Let $I$ be a perfect ideal of height 3 having minimal free resolution of format $(1,5,6,2)$. % Suppose that some of the structure maps $w^{(2)}_j(I)$, $w^{(3)}_j(I)$ has nonzero rank modulo $K$. Then $I$ is licci.
Then $I$ is licci if and only if some of the structure maps $w^{(2)}_j(I)$, $w^{(3)}_j(I)$ is nonzero modulo $\m$. 
\end{thm}

\begin{proof}
%It already known \rc cite \ec that every perfect ideal $J$ of height 3 such that $\beta(J) < \beta(I)$ is licci. 
First assume that some structure map in $W(d_2)$ or in $W(d_3)$ is nonzero modulo $\m$.
It is sufficient to find an ideal $J$ in the linkage class of $I$ such that $\beta(J) < \beta(I)$. Recall also that if $J$ is an ideal linked to $I$ and $\beta(J) = \beta(I)$, then also the minimal free resolution of $J$ has format $(1,5,6,2)$.
 Denote by $\A$ the minimal free resolution of $I$.
 %By Remark \ref{Ftop}, the complex $\A^{top}_{\bullet}$ is split exact. 
 By Theorem \ref{licciW2}, it is sufficient to assume that either 
 %This implies that at least one among 
 $w^{(3)}_3(I) \otimes K$ or $ w^{(2)}_3(I) \otimes K $ is nonzero.
 
 First suppose $ w^{(2)}_3(I) \otimes K \neq 0 $. In this case there exist generators $f_h \in A_2$ and 
 $\gamma_u \in A_3^*$ such that
 $\langle w^{(2)}_3(e_1 \wedge \ldots \wedge e_5 \otimes f_h ),  \gamma_1 \wedge \gamma_2 \otimes \gamma_u \rangle$ is a unit in $R$. Choose any regular sequence $\mathfrak{a}$ among the minimal generators of $I$, say $\mathfrak{a}= \lbrace x_1, x_2, x_3 \rbrace$. Let $J:= (\mathfrak{a}) : I$.  Use Theorem \ref{w3,3link} to compute the map $w^{(3)}_3$ on the free resolution of $J$. Notice then that the coefficient of $w^{(3)}_3(\gamma_u)$ with respect to $\epsilon_4 \wedge \epsilon_5 \otimes \phi_h$ is a unit in $R$. It follows that $ w^{(3)}_3(J) \otimes K \neq 0 $.
 
 Now, by replacing $I$ by some other ideal in its linkage class, we can assume that $w^{(3)}_3(I)$ is nonzero modulo $\m$. %$ \rank(w^{(3)}_3(I) \otimes K) \geq 1 $. 
 Hence, there exist generators $e_k \in A_1$ and $\phi_h \in A_2^*$ such that 
$\langle w^{(3)}_3(e_1 \wedge \ldots \wedge e_5 \otimes e_k),  \gamma_1 \wedge \gamma_2 \otimes \phi_h \rangle$ is a unit in $R$. By changing basis of $A_1$, assume $k=1$ and $\mathfrak{a}= \lbrace x_1, x_2, x_3 \rbrace$ is a regular sequence. 
Let $J:= (\mathfrak{a}) : I$ and use Theorem \ref{w2,2link} to compute the map $w^{(2)}_2$ on the free resolution of $J$. It follows that the coefficient of $w^{(2)}_2(\gamma_1, \gamma_2, s_1 \otimes \phi_h)$ with respect to $\epsilon_4 \wedge \epsilon_5 $ is a unit in $R$. Thus, $ w^{(2)}_2(J) \otimes K \neq 0 $. Theorem \ref{licciW2} implies the thesis.

Conversely, assume that all the structure maps in $W(d_2)$ and in $W(d_3)$ are zero modulo $\m$. Relations (\ref{eqw21,2}) in Remark \ref{remarkw2,1} and (W22,5) in Lemma \ref{splitw2,2} shows that also the maps $w^{(1)}_1$ and $w^{(1)}_{2,1}$ are zero modulo $\m$. Let $J$ be an ideal minimally linked to $I$. Combining Theorem \ref{w3,3link} with all the Theorems in Section 3, we get that also all the structure maps of the resolution of $J$ are zero modulo $\m$. Thus $J$ is a perfect ideal with minimal free resolution of format $(1,5,6,2).$ Iterating the process, we find that there exists no $H$ in the linkage class of $I$ such that $\beta(H) < \beta(I)$ and therefore $I$ cannot be licci.
\end{proof}

%\bc recall relations of Dynkin ideal and top complex, question about split exact \ec

We believe %that Theorem \ref{licci1562} implies 
that every perfect ideal of format $(1,5,6,2)$ is licci. For ideals of Dynkin type (except type $A_n$ and $(1,n,n,1)$ with $n$ odd), the top structure maps of the three critical representations, when computed with generic liftings by adding the defect variables, are the differential of a new complex, defined over a polynomial extension of $R$. If $\FF$ is the free resolution of one of such ideals, this second complex, canonically associated to $\FF$, is called $\FF^{top}_{\bullet}$, see \cite{LW19}, \cite{SW21}, \cite{GW20}. 

\begin{remark}
\label{remarkschubert}
In \cite{SW21}, it is conjectured that all perfect ideals of Dynkin type are obtained as specialization of Schubert varieties, and that the defining ideals of Schubert varieties are the generic perfect ideals of these formats.
From this it would follow that, after some change of basis in the defect variables, the complex $\FF^{top}_{\bullet}$ is split exact, and therefore the highest non-vanishing structure maps $w^{(2)}_{top}$ and $w^{(3)}_{top}$ are nonzero modulo the maximal ideal of $R$. %Thus Theorem \ref{licci1562} would imply that all perfect ideals of format $(1,5,6,2)$ are licci.
\end{remark}

For the format $(1,5,6,2)$, the generic perfect ideal coming from the Schubert variety has been investigated in \cite{CKLW} and \cite{kustin}, where is proved that is licci and rigid in the sense of deformation theory.
We believe that the same results as above may be true also for the generic perfect ideals of formats $E_7$ and $E_8$. However, the generators of these generic ideals contain many terms and the difficulty of their computations increases a lot compared to smaller formats. We state formally the following conjecture.

%\bc more explanations/ references? mention paper on examples and $U_{CM}= U_{split}$? \ec

\begin{conjecture}
\label{conj2} \rm
Let $I$ be a perfect ideal of Dynkin type in a local Gorenstein ring $R$. Then some structure map $w^{(i)}_{j,k}(I)$, for $i=2,3$, $j \geq 1$ has strictly positive rank modulo the maximal ideal of $R$.
\end{conjecture}
%Conjectures \ref{conj1} and \ref{conj2} together imply the Licci conjecture (item 2. of Conjecture \ref{licciconjecture}).
If Conjecture \ref{conj2} is true, then all perfect ideals of format $(1,5,6,2)$ are licci by Theorem \ref{licci1562}. By analogous arguments we expect also all the perfect ideals of other Dynkin formats to be licci.

 This conjecture is true if for every perfect ideal of Dynkin type, the complex $\FF^{top}_{\bullet}$ is split exact.
This last fact has an important relation with an old question posed by Peskine and Szpiro, which is now included among the unsolved homological conjectures, cf. \cite[Section 8]{homconj}, \cite{Peskine-Szpiro-dimension}, \cite{Peskine-Szpiro-4pages}. 

The question is the following:  let $R$ be a local ring and $M, N$ be finitely
generated $R$-modules such that $M$ has finite projective dimension and $l(M \otimes N) < \infty$, is
$\dim(M) + \dim(N) \leq \dim(R)$?

To explain the relation with this question we first need to recall the definition of two important open subsets of the spectrum of the generic ring ${\hat R}_{gen}$. Denote by $U_{CM}$ the set of all prime ideals of ${\hat R}_{gen}$ for which %the localization of the dual complex $(\FF^{gen}_\bullet )^*$ at $P$ is acyclic, i.e. the points where 
the localization of the generic homology module $H_0 (\FF^{gen}_\bullet)$ is perfect.
The open set $U_{split}$ consists of the points for which the complex $\FF^{top}_\bullet$ is split exact. These two sets are conjecturally equal for all Dynkin formats. The equality has been proved for $D_n$ formats in \cite{CVW18} and \cite{GW20}. Using the same method, with the help of computer algebra softwares to compute all the required formulas, we expect this equality to hold also for $E_6$. 

Let now $J$ be a perfect ideal of height 3 and of Dynkin type in a local ring $S$ %$\mathfrak{f}$ 
and assume $U_{CM}= U_{split}$. Going modulo a regular sequence, passing to completion, and adding free variables to the generic ring $R:= {\hat R}_{gen}$ we can assume that $\dim(S)=3$ and the homomorphism $\phi: {\hat R}_{gen} \to S$ is surjective. Let $P$ be a prime ideal of ${\hat R}_{gen}$ in the preimage of the maximal ideal of $S$. Call $I:= I_{gen}$ the ideal resolved by $\FF^{gen}_\bullet$.
%the generic complex of format $\mathfrak{f}$ over ${\hat R}_{gen}$. 
If the complex $\FF^{top}_{\bullet}$ associated to the free resolution of $J$ is not split exact, then since $U_{CM}= U_{split}$ we get that ht$(I_P)=2$. Using that $R$ and its localizations are Cohen-Macaulay (see \cite{W18}), we get that the pair $M:= \frac{R_P}{I_P}$, $N:=S$ over the ring $R_P$ would provide a counterexample to Peskine-Szpiro question.

%In particular, a perfect ideal of Dynkin type for which the complex $\FF^{top}_{\bullet}$ is not split exact (for any choice of defect variables) would provide a counterexample to Peskine-Szpiro question.

%\bc should we keep this? say more?  \ec

\section{Computation over a split exact complex}

In this section we exhibit some of the formulas for the higher structure maps over a split exact complex of length 3 of arbitrary format. All maps are computed with generic liftings. For this we introduce new sets of indeterminates, called \it defect variables \rm and use them to parametrize generically the kernels of the maps along we lift. This procedure as been already described in \cite{GW20} for split exact complexes of format $(1,n,n,1)$ and $(1,4, m+3, m)$. After writing down the formulas, we prove several equivariant relations between the maps $w^{(i)}_j$. As a consequence of Theorem \ref{equivariantrelations} these relations hold in general (for any format for which the maps are well-defined). %More in general many of these relations may be needed as defining equations to describe the generic ring as algebra over the complex numbers. 

The required computations get quickly very long and technical. For this reason, after identifying the correct patterns and listing the opportune definitions and formulas, we add explicit proof only of some of the first relations. This gives a precise idea of the general method that can be used to check the validity of the more complicated relations following the same pattern.

 We computed the more complicated structure maps over a split exact complex and checked all the relations stated in this section using the help of the computer algebra system Macaulay2. The results of these computations are available online on GitHub, see \cite{xianglong}.
 %\rc TO DO, cite package \ec
 
 \medskip 

Let us work over a Noetherian ring $R$. Let $r \geq 4$ be an integer.
Consider the split exact complex 
\begin{equation}
\label{basecomplex}
\FF: 0 \longrightarrow F_3 \buildrel{d_3}\over\longrightarrow  F_2 \buildrel{d_2}\over\longrightarrow F_1 \buildrel{d_1}\over\longrightarrow R 
\end{equation}
on the free $R$-modules $F_1$,$F_2$,$F_3$ having bases $ \lbrace e_1, \ldots, e_r \rbrace$, $ \lbrace f_1, \ldots, f_{r+m-1} \rbrace$, $ \lbrace g_1, \ldots, g_m \rbrace$. Denote the dual basis by $ \lbrace \epsilon_1, \ldots, \epsilon_{r} \rbrace$, $ \lbrace \phi_1, \ldots, \phi_{r+m-1} \rbrace$, $ \lbrace \gamma_1, \ldots, \gamma_m \rbrace$.

The differentials are defined by imposing $d_1(e_r)=1 $; $d_1(e_i)= 0$ and $ d_2(f_i)= e_i $ for $i < r$; $d_2(f_i)= 0$ for $i \geq r$; $d_3(g_i)= f_{i+r-1}$.

Let us construct a polynomial ring over $R$ adding the so-called defect variables.
Let $ b_{ij}^u$ be indeterminates over the ring $R$ defined for any $1 \leq i, j \leq r$, $1 \leq u \leq m$ and satisfying the relation $ b_{ij}^r=-b_{ji}^r$. 
Similarly, let $c_{i_1i_2i_3i_4}^{ut}$ be indeterminates over $R$ defined for any $1 \leq i_1,i_2,i_3,i_4 \leq r$, $1 \leq u,t \leq m$ and satisfying skew-symmetric relations in $i_1,i_2,i_3,i_4$ and in $u,t$. %\bc we probably need more sets of defect variables \ec
 These indeterminates are used to compute the maps $w^{(3)}_1$, $w^{(3)}_2$ in a generic way, expressing all the possible liftings. 
 To compute higher maps in the critical representation $W(d_3)$, new more sets of defect variables need to be introduced. We do not provide explicit formulas for those maps here in the paper, referring the reader to the Macaulay2 computation in \cite{xianglong}. However, notice that for the format $E_6$ the only needed sets of defect variables are $ b_{ij}^u$ and $c_{i_1i_2i_3i_4}^{ut}$. 
 
From now on we denote by $v^{(i)}_j$ the map obtained over the complex $\FF$ by computing the corresponding $w^{(i)}_j$ with a generic lifting. As example, the multiplication $e_1^.e_2$ in $\FF$ can be chosen to be equal to $0 + \beta$ where $\beta$ is any element of the kernel of $d_2$ (that is equal to the image of $d_3$). To express this generically we set $e_1^.e_2 = 0+ \sum_{u=1}^m b_{12}^u d_3(g_u)$. Similarly we do for the other entries as in \cite{GW20}.

We describe some of the maps $v^{(i)}_j$ over the complex $\FF$ in order to check the relations appearing in Section 3. %over the split exact complex. 
We list some of the entries. Clearly, by permutation of the indices with the usual sign rules one can obtain all the possible entries.
 As in the previous sections $\langle \cdot, \cdot \rangle$ is the evaluation map
and $\delta_{ij}$ denotes the Kronecker delta.
For the basic multiplication we get
$$ \langle e_i^.e_j, \phi_h  \rangle=  \left\{ \begin{array}{ccc} b^{h-r+1}_{ij} &\mbox{if } h \geq r, \\
       -\delta_{hi} &\mbox{ if } h < j = r, \\
       0  &\mbox{ otherwise. } \\
    \end{array}\right.  \quad      
 \langle e_i^.f_h, \gamma_u  \rangle= \left\{ \begin{array}{ccc} -b^u_{ih} &\mbox{if } h < r, \\
       \delta_{h-r+1,u} &\mbox{ if } h \geq r = i, \\
       0  &\mbox{ otherwise. } \\
    \end{array}\right. $$
 % -b^u_{ih}, \mbox{ for every } i  \mbox{ and for } h < r, $$
%$$ \langle e_i^.f_h, \gamma_u \rangle=  0, \mbox{ for } i < r, \, h \geq r, \, \, \, \langle e_r^.f_h, \gamma_u \rangle=  \delta_{h-r+1,u}, \mbox{ for } h \geq r. $$
$$ \langle e_i^.e_j^.e_k, \gamma_u \rangle =  \left\{ \begin{array}{ccc} b^u_{ij} &\mbox{if } k=r, \\
       0  &\mbox{ if } i,j,k < r. \\
    \end{array}\right. $$
%0, \mbox{ for } 1 \leq i < j < k < r, \quad 
 %\langle e_i^.e_j^.e_r, \gamma_u \rangle = b^u_{ij}. $$
 %For a choice of four indices $  i_1, i_2, i_3, i_4 $, set 
 %$$ L_{i_1, i_2, i_3, i_4} = \dfrac{1}{2}\sum_{1 \leq j < k \leq 4} (-1)^{j+k+1}  G_{i_ji_k} \otimes  G'_{\hat{i_j}\hat{i_k}}. $$ %\mbox{ for } 1 \leq i < j < r, \quad 
%\langle e_i^.e_r, \phi_h \rangle = -\delta_{hi} + b^h_{ir}. $$
 For the second graded component, set $$ P_{i_1,i_2,i_3,i_4}^{ut} := c^{ut}_{i_1,i_2,i_3,i_4} + \frac{1}{2}\sum_{j,k} (-1)^{i+j+1} b^{u}_{i_ji_k}b^{t}_{\hat{i_j}\hat{i_k}}. $$ 
 Assuming $i_1, i_2, i_3 < r$  we get  
%$$ \langle v^{(3)}_2(i_1, i_2, i_3, i_4), \phi_h \otimes \gamma_u \rangle = \left\{ \begin{array}{ccc} P_{i_1,i_2,i_3,i_4}^{u,h-r+1} &\mbox{if } h \geq r; \\
 %      0  &\mbox{ if } h < r. \\
  %  \end{array}\right.
 %$$ 
$$ \langle v^{(3)}_2(i_1, i_2, i_3, i_4), \phi_h \otimes \gamma_u \rangle = \left\{ \begin{array}{ccc} P_{i_1,i_2,i_3,i_4}^{u,h-r+1} &\mbox{if } h \geq r; \\
      (-1)^{j+k} b_{i_ji_k}^u &\mbox{ if } i_4 = r, h=i_l \mbox{ with } \lbrace j,k,l \rbrace = \lbrace 1,2,3 \rbrace; \\
       0  &\mbox{ otherwise. } \\
    \end{array}\right.
 $$ 
For $v^{(2)}_2$ set $ B^{ut}_{ij,kl}:= b^{u}_{ij}b^{t}_{kl}-b^{u}_{kl}b^{t}_{ij}. $ Then:
 $$ \langle v^{(2)}_2(i_1, i_2, i_3, f_h), \gamma_u \wedge \gamma_t \rangle =  \left\{ \begin{array}{ccc} \dfrac{1}{2}[ B^{ut}_{i_1i_2,i_3h}  - B^{ut}_{i_1i_3,i_2h} + B^{ut}_{i_2i_3,i_1h}] + c^{ut}_{i_1, i_2, i_3, h} &\mbox{if } h < r; \\
      \delta_{h-r+1,t}b_{i_1i_2}^{u}-\delta_{h-r+1,u}b_{i_1i_2}^{t} &\mbox{ if } i_3 = r, h \geq r, \\
       0  &\mbox{ otherwise. } \\
    \end{array}\right. $$

The next series of lemmas describes relations over the complex $\FF$ involving some of the maps $v^{(i)}_j$.

\begin{lem}
\label{split1}
For any choice of indices, equations (\ref{X}) and (\ref{Y}) hold over the complex $\FF$.
\end{lem}

\begin{proof}
For (\ref{X}) we have to show that $$a_3^*(X(e_i \wedge e_j \wedge e_k, \gamma_s \wedge \gamma_t)):= \langle e_i e_j e_k, \gamma_s   \rangle \gamma_t -
\langle e_i e_j e_k, \gamma_t   \rangle \gamma_s = a_3^*(\sum_{h=1}^{r+m-1} \langle v^{(2)}_2(i,j,k, f_h),  \gamma_s \wedge \gamma_t \rangle \phi_h). $$
If $i,j,k < r$, then $e_i e_j e_k=0$ and $v^{(2)}_2(i_1, i_2, i_3, f_h) = 0$ when $h \geq r$.
Hence, the sum in the right side term is taken only over $h < r$.  %$\sum_{h=1}^{r+m-1} \langle v^{(2)}_2(i,j,k, f_h),  \gamma_s \wedge \gamma_t \rangle \phi_h = \sum_{h=1}^{r - 1} \langle v^{(2)}_2(i,j,k, f_h),  \gamma_s \wedge \gamma_t \rangle \phi_h$. 
Since $ a_3^*(\phi_h)= 0$ for each $h < r$, we get the desired equality.

Instead, if $k=r$, we get $\langle e_i e_j e_r, \gamma_s   \rangle \gamma_t -
\langle e_i e_j e_r, \gamma_t   \rangle \gamma_s = b^s_{ij} \gamma_t - b^t_{ij} \gamma_s.    $ Working on the other term, we get $$   \sum_{h \geq r} \langle G_{ij} \wedge g_{h-r+1},  \gamma_s \wedge \gamma_t \rangle a_3^*(\phi_h) = b^s_{ij} \gamma_t - b^t_{ij} \gamma_s,  $$ since $a_3^*(\phi_h)= g_{h-r+1}$.

% \rc do the part for $Y$ \ec

For (\ref{Y}) first we observe that if $h \geq r$, then for every $\epsilon \in A_1^*$, $\langle f_h, a_2^*(\epsilon)  \rangle=0.$ Thus assume $h < r$. Hence,
$$ \langle f_h, a_2^*(\sum_{t=1}^{r} \langle v^{(3)}_2(e_{i}, e_{j}, e_{k}, e_t ), \phi_l \otimes \gamma_s  \rangle) \epsilon_t)  \rangle = \langle v^{(3)}_2(e_{i}, e_{j}, e_{k}, e_h ), \phi_l \otimes \gamma_s  \rangle =  $$ 
$$  = \left\{ \begin{array}{ccc} P_{ijkh}^{su} &\mbox{if } l \geq r \, (\mbox{here } u= l-r+1), \\
       b^s_{\hat{l}\hat{r}} &\mbox{ if } l \in \{ i,j,h \}, k=r, \\
       0  &\mbox{ otherwise. } \\
    \end{array}\right.  $$
    We need to compare this with the term 
    $$  \langle e_i^.e_j^.e_k, \gamma_s  \rangle \langle f_h, \phi_l \rangle - \langle w^{(2)}_2(e_i \wedge e_j \wedge e_k, f_h),  \gamma_s \wedge a_3^*(\phi_l) \rangle   - \langle e_i^.e_j, \phi_l \rangle \langle e_k^.f_h, \gamma_s \rangle + $$
$$ + \langle e_i^.e_k, \phi_l \rangle \langle e_j^.f_h, \gamma_s \rangle - \langle e_j^.e_k, \phi_l \rangle \langle e_i^.f_h, \gamma_s \rangle. $$
If $l \geq r$, $a_3^*(\phi_l) = \gamma_u$ with $u= l-r+1$. Thus the above term is equal to
$$  \dfrac{1}{2}[ B^{su}_{ij,kh } -B^{su}_{ik,jh } + B^{su}_{jk,ih }] + c^{su}_{ijkh}-  b^{s}_{ij}b^{u}_{kh} + b^{s}_{ik}b^{u}_{jh} - b^{s}_{jk}b^{u}_{ih} = P_{ijkh}^{su}. $$
If $l < r$, then $a_3^*(\phi_l) = 0$. If either $l \neq i,j,k,h$ or $i,j,k < r$ it can be easily check that all the summands above are zero. If $k=r$ (or equivalently if one of $i,j$ is equal to $r$), the only nonzero term is $\langle e_i^.e_j^.e_k, \gamma_s  \rangle \langle f_h, \phi_l \rangle = b_{ij}^s$ if $l=h$, and  $\langle e_i^.e_k, \phi_l \rangle \langle e_j^.f_h, \gamma_s \rangle = b^s_{jh}$ if $l=i$ (or the similar term if $l=j$). 
%$$ Y(e_i \wedge e_j \wedge e_k \otimes \gamma_s \otimes \phi_l) = \sum_{t=1}^{r_1} \langle w^{(3)}_2(e_{i}, e_{j}, e_{k}, e_t ), \phi_l \otimes \gamma_s  \rangle) \epsilon_t $$
\end{proof}

\begin{lem}
\label{split2}
For any choice of indices, the equation
%$$ (a): \sum_{h=1}^{r+m-1} \langle v^{(2)}_2(i_1,i_2,i_3  \otimes f_h),  \gamma_u \wedge \gamma_t \rangle \cdot \langle e_{i_2}^.e_{i_3}, \phi_h \rangle=0,  $$
%$$ (b): \sum_{h=1}^{r+m-1} \langle v^{(2)}_2(i_1,i_2,i_3  \otimes f_h),  \gamma_u \wedge \gamma_t \rangle \cdot \langle e_{i_2}^.e_{j_3}, \phi_h \rangle + \langle v^{(2)}_2(i_1,i_2,j_3  \otimes f_h),  \gamma_u \wedge \gamma_t \rangle \cdot \langle e_{i_2}^.e_{i_3}, \phi_h \rangle =0,   $$
$$  \sum_{h=1}^{r+m-1} \langle v^{(2)}_2(i_1,i_2,i_3  \otimes f_h),  \gamma_u \wedge \gamma_t \rangle \cdot \langle e_{j_2}^.e_{j_3}, \phi_h \rangle +\langle v^{(2)}_2(i_1,i_2,j_3  \otimes f_h),  \gamma_u \wedge \gamma_t \rangle \cdot \langle e_{j_2}^.e_{i_3}, \phi_h \rangle + $$ $$
\langle v^{(2)}_2(i_1,j_2,i_3  \otimes f_h),  \gamma_u \wedge \gamma_t \rangle \cdot \langle e_{i_2}^.e_{j_3}, \phi_h \rangle + \langle v^{(2)}_2(i_1,j_2,j_3  \otimes f_h),  \gamma_u \wedge \gamma_t \rangle \cdot \langle e_{i_2}^.e_{i_3}, \phi_h \rangle=0  $$
holds over the complex $\FF$.
%the equation $$ \sum_{h=1}^{r+m-1} \langle v^{(2)}_2(i,j,k  \otimes f_h),  \gamma_u \wedge \gamma_t \rangle) \cdot \langle e_i^.e_j, \phi_h \rangle = 0 $$  holds
\end{lem}

\begin{proof}
We first observe that, 
if $i,j < r$, then $e_i^.e_j \in d_3(F_3)= \langle f_r, \ldots, f_{r+m-1} \rangle$. Thus $\langle e_{i}^.e_{j}, \phi_h \rangle=0$ for every $h < r$. By this, whenever all the indices $i_1,i_2,i_3,j_2,j_3 < r$ each of the sums in the above terms is taken over $h \geq r$, and hence those terms are equal to zero. 

Also observe that if $j,k < r$, then 
$ \sum_{h=1}^{r+m-1} \langle v^{(2)}_2(i_2,i_3,r  \otimes f_h),  \gamma_u \wedge \gamma_t \rangle \cdot \langle e_{j},e_{k}, \phi_h \rangle = \sum_{h \geq r} \langle v^{(2)}_2(i_2,i_3,r  \otimes f_h),  \gamma_u \wedge \gamma_t \rangle \cdot b_{jk}^{h-r+1} = b_{i_2i_3}^{t}b_{jk}^{u} - b_{i_2i_3}^{u}b_{jk}^{t} = B^{tu}_{i_2i_3,jr}. $

Consider the case when $i_1=r$. We can restrict to assume that all the other indices are strictly smaller. %Then $(a)$ becomes $B^{tu}_{i_2i_3,i_2i_3}=0$.
%The term $(b)$ becomes $B^{tu}_{i_2i_3,i_2j_3}+B^{tu}_{i_2j_3,i_2i_3} =0$, and $(c)$ 
Our term becomes $B^{tu}_{i_2i_3,j_2j_3}+B^{tu}_{i_2j_3,j_2i_3} + B^{tu}_{j_2i_3,i_2j_3}+B^{tu}_{j_2j_3,i_2i_3} =0$.

Next suppose $i_3=r$ and all the other indices to be strictly smaller. Observing that for $h < r$, $ \langle e_{j}^.e_{r}, \phi_h \rangle = -\delta_{jh}$, our term %The term $(a)$ becomes $$  \langle v^{(2)}_2(i_1,i_2,r  \otimes f_{i_2}),  \gamma_u \wedge \gamma_t \rangle + \sum_{h \geq r} \langle v^{(2)}_2(i_1,i_2,r  \otimes f_h),  \gamma_u \wedge \gamma_t \rangle \cdot b_{i_2r}^{h-r+1} = $$ $$ = \dfrac{1}{2}[B^{ut}_{i_1i_2,ri_2 } +B^{ut}_{i_2r,i_1i_2 }]+ B^{tu}_{i_1i_2,ri_2 }=0.  $$
%By a similar computation $(b)$ becomes 
%$ B^{tu}_{i_1i_2,j_3i_2} + \dfrac{1}{2}[B^{ut}_{i_1i_2,j_3i_2 } +B^{ut}_{i_2j_3,i_1i_2 }]=0,  $ and $(c)$ 
becomes 
$$ B^{tu}_{i_1i_2,j_3j_2} + B^{tu}_{i_1j_2,j_3i_2} + \dfrac{1}{2}[B^{ut}_{i_1i_2,j_3j_2 } -B^{ut}_{i_1j_3,i_2j_2 } + B^{ut}_{i_2j_3,i_1j_2 }] + $$ $$ + \dfrac{1}{2}[B^{ut}_{i_1j_2,j_3i_2 } - B^{ut}_{i_1j_3,j_2i_2 } +B^{ut}_{j_2j_3,i_1i_2 }] + c_{i_1i_2j_3j_2} + c_{i_1j_2j_3i_2} = 0.  $$
%By permutation of indices this gives all the possibilities for $(a)$ and $(c)$. 
We finally need to consider the case $i_2 = j_2=  r$ and all other indices smaller.
This gives $\dfrac{1}{2}[B^{ut}_{i_1r,i_3j_3 } -B^{ut}_{i_1i_3,rj_3 } + B^{ut}_{ri_3,i_1j_3 }]+ B^{tu}_{i_1i_3,rj_3 } + \dfrac{1}{2}[B^{ut}_{i_1r,j_3i_3 } -B^{ut}_{i_1j_3,ri_3 } + B^{ut}_{rj_3,i_1i_3 }]+ B^{tu}_{i_1j_3,ri_3 } = 0.$
%the left term of the equation becomes 
%$ \sum_{h \geq r} \langle v^{(2)}_2(i,j,k  \otimes f_h),  \gamma_u \wedge \gamma_t \rangle \cdot b_{ij}^{h-r+1}. $
%This is clearly zero if $k < r$. Otherwise, if $k=r$ it is equal to $   b_{ij}^{t}b_{ij}^{u} - b_{ij}^{u}b_{ij}^{t}=0. $
%Instead, if $j=r$, then $ \langle e_i^.e_j, \phi_h \rangle = -\delta_{ih} $. Therefore we get
% $$  \langle v^{(2)}_2(i,r,k  \otimes f_i),  \gamma_u \wedge \gamma_t \rangle + \sum_{h \geq r} \langle v^{(2)}_2(i,r,k  \otimes f_h),  \gamma_u \wedge \gamma_t \rangle \cdot b_{ir}^{h-r+1} = \dfrac{1}{2}[B^{ut}_{ir,ki } -B^{ut}_{ik,ri }]+ b_{ik}^{t}b_{ir}^{u} - b_{ik}^{u}b_{ir}^{t}=0.  $$
\end{proof}

 %\bc TO DO: add $v^{(1)}_{2,1}, v^{(1)}_{3,1}$ for split exact complex \ec

The next lemma describes quadratic relations between $W(d_1)$ and $W(d_2)$. We expect a general relation of the form $\sum_{j=0}^k (-1)^j v^{(1)}_j v^{(2)}_{k-j} = 0 $ to be satisfied for any $k$. Here we consider $k=1,2,3$.
 We give here an explicit formula for $v^{(1)}_{2,1}$ computed using (\ref{q1,2eq}). Those for $ v^{(1)}_{3,1}$ and $ v^{(2)}_{3,1}$ can be obtained by computer using (\ref{q1,3eq}), (\ref{q2,3eq}).
 %For this we need to have explicit formulas for $v^{(1)}_{2,1}$, $ v^{(1)}_{3,1}$ and $ v^{(2)}_{3,1}$ over the complex $\FF$.
 Let us again adopt the notation $\varepsilon_{i_1,\ldots, i_s} := e_{i_1} \wedge \ldots \wedge e_{i_s}$.
Then 
 $$ \langle v^{(1)}_{2,1}(\varepsilon_{i_1, i_2, i_3, i_4} \otimes e_{i_5}), \gamma_u \wedge \gamma_t \rangle
=  \left\{ \begin{array}{ccc} \dfrac{1}{2}[ B^{ut}_{i_1i_2,i_3i_5}  - B^{ut}_{i_1i_3,i_2i_5} + B^{ut}_{i_2i_3,i_1i_5}] - c^{ut}_{i_1, i_2, i_3, i_5} &\mbox{if } i_4 = r; \\
       0  &\mbox{ if } i_1,i_2,i_3,i_4 < r. 
    \end{array}\right. $$
%    $$ \rc \langle v^{(1)}_{3,1}(\varepsilon_{i_1, i_2, i_3, i_4} \otimes \varepsilon_{i_5, i_6, i_7}), \wedge_{j=1}^3 \gamma_{u_j} \rangle
%=  \left\{ \begin{array}{ccc}  \\
  %      \\
 %   \end{array}\right. $$
% $$ \rc \langle v^{(2)}_{3,1}(\varepsilon_{i_1, i_2, i_3, i_4} \otimes e_{i_5} \otimes f_h), \wedge_{j=1}^3 \gamma_{u_j} \rangle
%=  \left\{ \begin{array}{ccc}  \\
       
 %   \end{array}\right. $$
%Similarly \rc add formula for $v^{(1)}_{3,1}$, $ v^{(2)}_{3,1}$ \ec

%0, \mbox{ if } i_1,i_2,i_3,i_4 < r, \mbox{ and } $$ 
 %$$ \langle v^{(1)}_{2,1}(\varepsilon_{i_1, i_2, i_3, r} \otimes e_{i_5}), \gamma_u \wedge \gamma_t \rangle = \dfrac{1}{2}[ B^{ut}_{i_1i_2,i_3i_5}  - B^{ut}_{i_1i_3,i_2i_5} + B^{ut}_{i_2i_3,i_1i_5}] - c^{ut}_{i_1, i_2, i_3, i_5}.  $$

\begin{lem}
\label{splitw1,1}
Denote by $y_{ij}$ the entries of $d_2$ and by $x_1, \ldots, x_r$ the entries of $d_1$.
The following relations hold over the complex $\FF$, for any choice of indices such that $k_4, k_5 \in \langle k_1, k_2, k_3 \rangle$.
$$  (W11,1): \sum_{i=1}^{r} y_{ih} \langle e_j^.e_k^.e_i, \gamma_t \rangle  = x_k \langle e_j^.f_h, \gamma_t \rangle - x_j \langle e_k^.f_h, \gamma_t \rangle. $$
$$  (W11,2): \sum_{i=1}^{r} y_{ih} \langle v^{(1)}_{2,1}(\varepsilon_{i, k_1, k_2, k_3} \otimes e_{k_4}), \gamma_u \wedge \gamma_t \rangle = \langle e_{k_1}^.e_{k_2}^.e_{k_3}, \gamma_u \rangle \cdot \langle e_{k_4}^.f_h, \gamma_t \rangle + $$
$$ -\langle e_{k_1}^.e_{k_2}^.e_{k_3}, \gamma_t \rangle \cdot \langle e_{k_4}^.f_h, \gamma_u \rangle - x_{k_1} \langle v^{(2)}_2( \varepsilon_{k_2, k_3, k_4} \otimes f_h  ), \gamma_u \wedge \gamma_t \rangle + $$ 
$$ + x_{k_2} \langle v^{(2)}_2( \varepsilon_{k_1, k_3, k_4} \otimes f_h  ), \gamma_u \wedge \gamma_t \rangle - x_{k_3} \langle v^{(2)}_2( \varepsilon_{k_1, k_2, k_4} \otimes f_h  ), \gamma_u \wedge \gamma_t \rangle.  $$
 % possibly there is a constant $1/2$ 
$$  (W11,3):  \sum_{i=1}^{r} y_{ih} \langle v^{(1)}_{3,1}(\varepsilon_{i, k_1, k_2, k_3} \otimes e_{k_4} \wedge e_{k_5} \wedge e_{k_6}), \gamma_{u_1} \wedge \gamma_{u_2} \wedge \gamma_{u_3} \rangle = $$ 
%$$ \sum_{u_1, u_2, u_3 \in \lbrace 1,2,3 \rbrace} \sum_{j,l,r \in \lbrace 4,5,6 \rbrace} \pm \langle v^{(1)}_{2,1}(\varepsilon_{k_1, k_2, k_3, k_j} \otimes e_{k_l}),  \gamma_{u_1} \wedge \gamma_{u_2} \rangle \cdot \langle v^{(2)}_1(e_{k_r} \otimes f_h),  \gamma_{u_3} \rangle?? +   $$
$$ \sum_{i=1}^3(-1)^{i} \langle e_{k_1}^.e_{k_2}^.e_{k_3}, \gamma_{u_i}  \rangle \cdot \langle v^{(2)}_2(e_{k_4}, e_{k_5}, e_{k_6} \otimes f_h), \gamma_{u_{j_1}} \wedge \gamma_{u_{j_2}} \rangle +   $$
$$ +\sum_{ j,l,r \in \lbrace 1,2,3 \rbrace } (-1)^{\sigma(j,l,r)} x_{k_j}  \langle v^{(2)}_{3,1}(\varepsilon_{k_l, k_4, k_5, k_6} \otimes e_{k_r} \otimes f_h), \gamma_{u_1} \wedge \gamma_{u_2} \wedge \gamma_{u_3} \rangle.     $$
\end{lem} 
%$$  (W11,3):  \sum_{i=1}^{r_1} y_{ih} \langle v^{(1)}_{3,1}(\varepsilon_{i, k_1, k_2, k_3} \otimes e_{k_4} \wedge e_{k_5} \wedge e_{k_6}), \wedge_{u=1}^3 \gamma_{u} \rangle = $$ 
%$$ \sum_{u_1, u_2, u_3 \in \lbrace 1,2,3 \rbrace} (-1)^{u_1} \langle e_{k_1}^.e_{k_2}^.e_{k_3}, \gamma_{u_1}  \rangle \cdot \langle v^{(2)}_2(e_{k_4}, e_{k_5}, e_{k_6} \otimes f_h), \gamma_{u_2} \wedge \gamma_{u_3} \rangle +   $$
%$$ -\sum_{j=1}^3 (-1)^k x_{k_j} [ \langle v^{(2)}_{3,1}(\varepsilon_{k_{j_1}, k_4, k_5, k_6} \otimes e_{k_{j_2}} \otimes f_h), \wedge_{u=1}^3 \gamma_{u} \rangle + \langle v^{(2)}_{3,1}(\varepsilon_{k_{j_2}, k_4, k_5, k_6} \otimes e_{k_{j_1}} \otimes f_h), \wedge_{u=1}^3 \gamma_{u} \rangle].     $$
 \begin{proof}
Notice that $y_{ih} = 1$ if and only if $i=h < r$ and $x_k= 1$ if and only if $k=r$, otherwise they are zero.

Hence, for $(W11,1)$, if either $j,k < r$ or if $h \geq r$ both sides are zero. Clearly also if $j=k$ both sides are zero. Thus suppose $j,h < r$, $k = r$. In this case both terms are equal to $b^t_{hj}$. 

Also for $(W11,2)$, the left side term is nonzero only if $i=h < r$ and one of $k_1,k_2,k_3$ is equal to $r$. Say that $k_3=r$.
Then the left side term is 
$$ \langle v^{(2)}_{2,1}(\varepsilon_{h, k_1, k_2, r} \otimes e_{k_4}), \gamma_u \wedge \gamma_t \rangle = \dfrac{1}{2}[ B^{ut}_{hk_1,k_2k_4}  - B^{ut}_{hk_2,k_1k_4} + B^{ut}_{k_1k_2,hk_4}] - c^{ut}_{h, k_1, k_2, k_4}. $$
In the analogous case, for the right side term we get
$$ b_{k_1k_2}^{t} b_{k_4h}^{u} - b_{k_1k_2}^{u} b_{k_4h}^{t} + \dfrac{1}{2}[ B^{ut}_{k_1k_2,k_4h}  - B^{ut}_{k_1k_4, k_2h} + B^{ut}_{k_2k_4,k_1h}] + c^{ut}_{k_1, k_2, k_4, h}. $$
These terms checks out to be equal by skew-symmetric properties of the indices.
It is not hard to check that the right side term is zero in all the other cases.
The relation $(W11,3)$ can be checked by computer or using similar methods. 
\end{proof}

 % \langle \alpha_1(s_1)^.\alpha_1(s_2)^.\alpha_1(s_3), \gamma_1  \rangle \cdot \langle w^{(2)}_2(\alpha_1(s_1),\alpha_1(s_2),\alpha_1(s_3) \otimes f_h), \gamma_2 \wedge \gamma_3 \rangle + $$ $$ -\langle \alpha_1(s_1)^.\alpha_1(s_2)^.\alpha_1(s_3), \gamma_2  \rangle \cdot \langle w^{(2)}_2(\alpha_1(s_1),\alpha_1(s_2),\alpha_1(s_3) \otimes f_h), \gamma_1 \wedge \gamma_3 \rangle + $$  $$ +\langle \alpha_1(s_1)^.\alpha_1(s_2)^.\alpha_1(s_3), \gamma_3  \rangle \cdot \langle w^{(2)}_2(\alpha_1(s_1),\alpha_1(s_2),\alpha_1(s_3) \otimes f_h), \gamma_1 \wedge \gamma_2 \rangle. $$

 The next lemma deals with quadratic relations in $W(d_2)$ of the form $\sum_{j=0}^k (-1)^j v^{(2)}_j v^{(2)}_{k-j} = 0 $.
 %For this we need to compute $v^{(2)}_{4,1}$. 
 The map $v^{(2)}_{4,1}$ can be computed using (\ref{q2,4eq}).   
% $$ \rc \langle v^{(2)}_{4,1}(\varepsilon_{i_1, i_2, i_3, i_4} \otimes \varepsilon_{i_5, i_6, i_7} \otimes f_h), \wedge_{j=1}^4 \gamma_{u_j} \rangle
%=  \left\{ \begin{array}{ccc}  \\
 %       \\
 %   \end{array}\right. $$

\begin{lem}
\label{splitw3,2}
Denote by $y_{ij}$ the entries of $d_2$. Set $\vartheta_{ij}^h:= \langle w^{(2)}_2(e_{k_1}, e_{k_2}, e_{k_3}) \otimes f_h),  \gamma_{u_i} \wedge \gamma_{u_j} \rangle$.
The following relations hold over the complex $\FF$, for any choice of indices such that $k_4, k_5, k_6 \in \langle k_1, k_2, k_3 \rangle$.
$$  (W32,1): \sum_{i=1}^{r} \langle e_i^.f_h,  \gamma_t \rangle y_{ik} + \langle e_i^.f_k,  \gamma_t \rangle y_{ih}= 0. $$
$$  (W32,2): \sum_{i=1}^{r} \langle v^{(2)}_2(e_i, e_j, e_l \otimes f_h),  \gamma_t \wedge \gamma_u \rangle y_{ik} + \langle v^{(2)}_2(e_i, e_j, e_l \otimes f_k),  \gamma_t \wedge \gamma_u \rangle y_{ih} =  $$
$$ \langle e_j^.f_h,  \gamma_t \rangle \cdot  \langle e_l^.f_k,  \gamma_u \rangle  - \langle e_j^.f_h,  \gamma_u \rangle \cdot  \langle e_l^.f_k,  \gamma_t \rangle + \langle e_j^.f_k,  \gamma_t \rangle \cdot  \langle e_l^.f_h,  \gamma_u \rangle  - \langle e_j^.f_k,  \gamma_u \rangle \cdot  \langle e_l^.f_h,  \gamma_t \rangle. $$ 
$$  (W32,3):  \sum_{i=1}^{r} \langle v^{(2)}_{3,1}(\varepsilon_{i,k_1,k_2,k_3} \otimes e_{k_1} \otimes f_h), \wedge_{s=1}^3 \gamma_{u_s} \rangle y_{ik} + \langle v^{(2)}_{3,1}(\varepsilon_{i,k_1,k_2,k_3} \otimes e_{k_1} \otimes f_k), \wedge_{s=1}^3 \gamma_{u_s} \rangle y_{ih} = $$
$$ -\vartheta_{12}^h \cdot \langle e_{k_1}^.f_k,  \gamma_{u_3} \rangle - \vartheta_{12}^k \cdot \langle e_{k_1}^.f_h,  \gamma_{u_3} \rangle + \vartheta_{13}^h \cdot \langle e_{k_1}^.f_k,  \gamma_{u_2} \rangle + \vartheta_{13}^k \cdot \langle e_{k_1}^.f_h,  \gamma_{u_2} \rangle + $$  $$ - \vartheta_{23}^h \cdot \langle e_{k_1}^.f_k,  \gamma_{u_1} \rangle- \vartheta_{23}^k \cdot \langle e_{k_1}^.f_h,  \gamma_{u_1} \rangle.  $$
$$  (W32,4):  \sum_{i=1}^{r} \langle v^{(2)}_{4,1}(\varepsilon_{i,k_1,k_2,k_3} \otimes \varepsilon_{k_1,k_2,k_3} \otimes f_h), \wedge_{s=1}^4 \gamma_{u_s} \rangle y_{ik} + \langle v^{(2)}_{4,1}(\varepsilon_{i,k_1,k_2,k_3} \otimes \varepsilon_{k_1,k_2,k_3} \otimes f_k), \wedge_{s=1}^4 \gamma_{u_s} \rangle y_{ih}=  $$
$$  \vartheta^h_{12} \cdot \vartheta^k_{34} - \vartheta^h_{13} \cdot \vartheta^k_{24} + \vartheta^h_{14} \cdot \vartheta^k_{23} + \vartheta^h_{23} \cdot \vartheta^k_{14} - \vartheta^h_{24} \cdot \vartheta^k_{13} + \vartheta^h_{34} \cdot \vartheta^k_{12}.   $$
\end{lem}

%$$ \sum_{u_1,u_2} (-1)^{u_1+u_2+1} \langle w^{(2)}_2(\alpha_1(s_1), \alpha_1(s_2), \alpha_1(s_3) \otimes f_h),  \gamma_{u_1} \wedge \gamma_{u_2} \rangle \cdot  \langle w^{(2)}_2(\alpha_1(s_1), \alpha_1(s_2), \alpha_1(s_3) \otimes f_k),  \gamma_{u_3} \wedge \gamma_{u_4} \rangle + $$ $$ +  \langle w^{(2)}_2(\alpha_1(s_1), \alpha_1(s_2), \alpha_1(s_3) \otimes f_k),  \gamma_{u_1} \wedge \gamma_{u_2} \rangle \cdot  \langle w^{(2)}_2(\alpha_1(s_1), \alpha_1(s_2), \alpha_1(s_3) \otimes f_h),  \gamma_{u_3} \wedge \gamma_{u_4} \rangle.  $$

%(W32,2): \sum_{i=1}^{r_1} \langle w^{(2)}_2(e_i, \alpha_1(s_1), \alpha_1(s_2) \otimes f_h),  \gamma_1 \wedge \gamma_2 \rangle y_{ik} + \langle w^{(2)}_2(e_i, \alpha_1(s_1), \alpha_1(s_2) \otimes f_k),  \gamma_1 \wedge \gamma_2 \rangle y_{ih} =  $$
%$$  \langle \alpha_1(s_1)^.f_h,  \gamma_2 \rangle \cdot  \langle \alpha_1(s_2)^.f_k,  \gamma_1 \rangle -  \langle \alpha_1(s_1)^.f_h,  \gamma_1 \rangle \cdot  \langle \alpha_1(s_2)^.f_k,  \gamma_2 \rangle + $$
%$$ + \langle \alpha_1(s_1)^.f_k,  \gamma_2 \rangle \cdot  \langle \alpha_1(s_2)^.f_h,  \gamma_1 \rangle -  \langle \alpha_1(s_1)^.f_k,  \gamma_1 \rangle \cdot  \langle \alpha_1(s_2)^.f_h,  \gamma_2 \rangle. $$

\begin{proof}
Notice that $y_{ik} = 1$ if and only if $i=k < r$, otherwise is zero. Also recall that $e_i^.f_k = 0$
 if $i < r$ and $h \geq r$.
For $(W32,1)$, if $h,k < r$ we get 
$ \langle e_k^.f_h,  \gamma_t \rangle + \langle e_h^.f_k,  \gamma_t \rangle = -b_{kh}^t - b_{hk}^t = 0. $ %If $h < r$ and $k \geq r$, we get $\langle e_h^.f_k,  \gamma_t \rangle = 0$
In all the other cases any summand is clearly zero.

For $(W32,2)$, if $j=l$ everything is obviously zero. Hence, also if $h,k \geq r$, both terms are zero. If $h < r$ and $k \geq r$, we again get that both terms are zero if $j,l < r$ or if $u,t \neq k-r+1$. Hence assume $l=r$ and $u= k-r+1$. Thus the first term is 
$ \langle v^{(2)}_2(e_h, e_j, e_r \otimes f_k),  \gamma_t \wedge \gamma_u \rangle = b^t_{hj}.  $ The second term becomes 
$    \langle e_j^.f_h,  \gamma_t \rangle \cdot  \langle e_r^.f_k,  \gamma_u \rangle = -b^t_{jh}. $ Finally suppose $h,k < r$. The first term becomes 
$$ \dfrac{1}{2}[ B^{tu}_{kj,lh}- B^{tu}_{kl,jh} + B^{tu}_{jl,kh}   ] + c_{kjlh} + \dfrac{1}{2}[ B^{tu}_{hj,lk}- B^{tu}_{hl,jk} + B^{tu}_{jl,hk}   ] + c_{hjlk} = B^{tu}_{kj,lh}- B^{tu}_{kl,jh}. $$
The second term coincides with the first one since, if $i,k < r$, $\langle e_i^.f_k, \gamma_t \rangle = -b_{ih}^t.$ 
Relations $(W32,3), (W32,4)$ are checked by computer. 
\end{proof}

Now we consider quadratic relations between all the critical representations having form $\sum_{j=0}^k (-1)^j v^{(2)}_j(\rho) v^{(3)}_{k-j}(h) = \delta_{\rho h}w_{k-1}^{(1)} $. For $k=1,2$ these are those described in Remark \ref{remarkw2,1}. The formulas for $v^{(3)}_{3,1}, v^{(3)}_{4,1}$ are obtained using (\ref{q3,3eq}), (\ref{q3,4eq}).  

\begin{lem}
\label{splitw2,2}
Denote by $y_{ij}$ the entries of $d_2$ and by $z_{ij}$ the entries of $d_3$.
 The following relations hold over the complex $\FF$, for any choice of indices such that $i_5, i_6, i_7 \in \langle i_1, i_2, i_3 \rangle$.
% $$  (W22,1): \sum_{k=1}^{r_1}  y_{k\rho}  \langle  e_k^.e_i, \phi_h \rangle = \delta_{\rho h} x_i - \sum_{u=1}^{r_3} z_{hu} \langle e_i f_{\rho}, \gamma_u \rangle.   $$
% $$  (W22,3): \sum_{k=1}^{r_1}  y_{k\rho}  \langle  w^{(3)}_2(\varepsilon_{i_1,i_2,i_3,k}), \phi_h \otimes \gamma_t \rangle =  \delta_{\rho h} \langle e_{i_1}^.e_{i_2}^.e_{i_3}, \gamma_1 \rangle + \langle e_{i_1}^.f_{\rho}, \gamma_1 \rangle \cdot \langle e_{i_2}^.e_{i_3}, \phi_h \rangle + $$  $$    - \langle e_{i_2}^.f_{\rho}, \gamma_1 \rangle \cdot \langle e_{i_1}^.e_{i_3}, \phi_h \rangle + \rc \langle e_{i_3}^.f_{\rho}, \gamma_1 \rangle \cdot \langle e_{i_1}^.e_{i_2}, \phi_h \rangle \ec - \sum_{u=1}^{r_3} z_{hu} \langle w^{(2)}_2(e_{i_1},e_{i_2},e_{i_3} \otimes f_{\rho}), \gamma_1 \wedge\gamma_u \rangle. $$
 $$  (W22,5):    \sum_{k=1}^{r_1}  y_{k\rho}  
\langle  v^{(3)}_{3,1}(\varepsilon_{i_1,i_2,i_3,i_4,k} \otimes e_{i_5}), \phi_h \otimes \gamma_t \wedge \gamma_u \rangle =   \delta_{\rho h}
  \langle v_{2,1}^{(1)}(\varepsilon_{i_1,i_2,i_3,i_4} \otimes e_{i_5}),  \gamma_t \wedge \gamma_u \rangle + $$ $$- \sum_{s=1}^{m} z_{hs} \langle v^{(2)}_{3,1}(\varepsilon_{i_1,i_2,i_3,i_4} \otimes e_{i_5} \otimes f_{\rho}), \gamma_t \wedge \gamma_u \wedge \gamma_s \rangle - \sum_{j=1}^4 (-1)^j \langle e_{i_5}^.e_{i_j}, \phi_h \rangle \cdot \langle v^{(2)}_2(\hat{i_j}, \hat{i_5} \otimes \phi_{\rho}), \gamma_t \wedge\gamma_u \rangle  $$
%  $$  + \langle e_{i_5}^.e_{i_1}, \phi_h \rangle \cdot \langle w^{(2)}_2(e_{i_2}, e_{i_3}, e_{i_4} \otimes \phi_{\rho}), \gamma_t \wedge\gamma_u \rangle    - \langle e_{i_5}^.e_{i_2}, \phi_h \rangle \cdot \langle w^{(2)}_2(e_{i_1}, e_{i_3}, e_{i_4} \otimes \phi_{\rho}), \gamma_t \wedge\gamma_u \rangle  + $$  
 % $$  + \langle e_{i_5}^.e_{i_3}, \phi_h \rangle \cdot \langle w^{(2)}_2(e_{i_1}, e_{i_2}, e_{i_4} \otimes \phi_{\rho}), \gamma_t \wedge\gamma_u \rangle    - \langle e_{i_5}^.e_{i_4}, \phi_h \rangle \cdot \langle w^{(2)}_2(e_{i_1}, e_{i_2}, e_{i_3} \otimes \phi_{\rho}), \gamma_t \wedge\gamma_u \rangle  + $$   
  $$-  \langle e_{i_5}^.f_{\rho}, \gamma_t \rangle \cdot \langle v^{(3)}_2(\varepsilon_{i_1,i_2,i_3,i_4}, \phi_h \otimes \gamma_u \rangle + \langle e_{i_5}^.f_{\rho}, \gamma_u \rangle \cdot \langle v^{(3)}_2(\varepsilon_{i_1,i_2,i_3,i_4}, \phi_h \otimes \gamma_t \rangle.  $$ 
  % \frac{1}{2}
 $$  (W22,7): 
\sum_{k=1}^{r_1}  y_{k\rho}  
\langle  v^{(3)}_{4,1}(\varepsilon_{k, i_1, i_2, i_3, i_4} \otimes \varepsilon_{i_5, i_6, i_7}), \phi_h \otimes \wedge_{s=1}^3 \gamma_{u_s} \rangle =  $$ 
$$ \delta_{\rho h} \langle v^{(1)}_{3,1}(\varepsilon_{i_1, i_2, i_3, i_4} \otimes \varepsilon_{i_5, i_6, i_7} ),\wedge_{s=1}^3 \gamma_{u_s}  \rangle +  \sum_{u=1}^{m} z_{hu} \langle v^{(2)}_{4,1}(\varepsilon_{i_1, i_2, i_3, i_4} \otimes \varepsilon_{i_5, i_6, i_7} \otimes f_{\rho}),\wedge_{s=1}^3 \gamma_{u_s} \wedge \gamma_u \rangle+ $$
$$ + \sum_{j=5}^7 (-1)^{j} \langle e_{j_2}^. e_{j_3}, \phi_h   \rangle \cdot \langle w^{(2)}_{3,1}(\varepsilon_{i_1, i_2, i_3, i_4} \otimes e_{i_j} \otimes f_{\rho}), \wedge_{s=1}^3 \gamma_{u_s} \rangle+  $$
$$ + \sum_{s=1}^3 (-1)^{s} \langle w^{(3)}_2(\varepsilon_{i_1, i_2, i_3, i_4}, \phi_h \otimes \gamma_{u_s} \rangle \cdot \langle w^{(2)}_2(\varepsilon_{i_5, i_6, i_7} \otimes f_{\rho}) , \hat{\gamma_{u_s}} \rangle.   $$
%$$ + \sum_{j_1, j_2, j_3 = i_5, i_6, i_7} (-1)^{j_1} \langle e_{j_2}^. e_{j_3}, \phi_h   \rangle \cdot \langle w^{(2)}_{3,1}(\varepsilon_{i_1, i_2, i_3, i_4} \otimes e_{j_1} \otimes f_{\rho}), \wedge_{s=1}^3 \gamma_{u_s} \rangle+  $$
%$$ + \sum_{s=1}^3 (-1)^{s+1} \langle w^{(3)}_2(\varepsilon_{i_1, i_2, i_3, i_4}, \phi_h \otimes \gamma_{u_s} \rangle \cdot \langle w^{(2)}_2(\varepsilon_{i_5, i_6, i_7} \otimes f_{\rho}) , \hat{\gamma_{u_s}} \rangle.   $$
\end{lem}
\begin{proof}
These relations can be checked by computer. 
\end{proof}

For the last lemma, assume $\FF$ to be of format $(1,5,6,2)$.
%\bc add the computation of $ v^{(2)}_3 $ for the format 1562, or do it in another paper and cite it here? \ec
For this format, we recall how the map $ v^{(2)}_3: \bigwedge^5 F_1 \otimes F_2  \to \bigwedge^2 F_3 \otimes F_3 $ is expressed in term of the quantities $P_{i_1,i_2,i_3,i_4}^{ut}$ and $B_{i_1i_2,i_3i_4}^{ut}$ defined previously. We describe the specific cases of $\langle v^{(2)}_3(\varepsilon_{1, \ldots, 5} \otimes f_h ),  \gamma_1 \wedge \gamma_2 \otimes \gamma_t \rangle$ for $h=4$ and $h \geq 5$. The cases $h=1,2,3$ can be obtained from the case $h=4$ by permutation. 
$$  \langle v^{(2)}_3(\varepsilon_{1, \ldots, 5} \otimes f_4 ),  \gamma_1 \wedge \gamma_2 \otimes \gamma_t \rangle = \frac{1}{2}[ B^{12}_{34,45} b_{12}^{t} - B^{12}_{24,45} b_{13}^{t} + B^{12}_{14,45} b_{23}^{t} - B^{12}_{24,34} b_{15}^{t}+ $$  $$ +B^{12}_{14,34} b_{25}^{t} - B^{12}_{14,24} b_{35}^{t} ] - c_{2345}^{12} b^t_{14} + c_{1345}^{12} b^t_{24} - c_{1245}^{12} b^t_{34} + c_{1234}^{12} b^t_{45}.  $$
If $h =5,6$, then
$$   \langle v^{(2)}_3(\varepsilon_{1, \ldots, 5} \otimes f_h ),  \gamma_1 \wedge \gamma_2 \otimes \gamma_t \rangle =   P_{1,2,3,4}^{7-h,t}.    $$
%$$  \langle v^{(2)}_3(\varepsilon_{1, \ldots, 5} \otimes f_6 ),  \gamma_1 \wedge \gamma_2 \otimes \gamma_1 \rangle    $$
%(\delta_{h-5+1,t}) c_{1234}^{12} +

For the next lemma, set $\sigma(i,k) = i+1$ if $i > k$ and $\sigma(i,k) = i$ otherwise.

\begin{lem}
\label{splitw3,3}
Assume $\FF$ to be of format $(1,5,6,2)$.
Denote by $y_{ij}$ the entries of $d_2$.  The following relation holds over the complex $\FF$, for any choice of indices.
$$  (W33):    \langle v^{(2)}_3(\varepsilon_{1, \ldots, 5} \otimes f_h ),  \gamma_1 \wedge \gamma_2 \otimes \gamma_t \rangle y_{k \rho} + \langle v^{(2)}_3(\varepsilon_{1, \ldots, 5} \otimes f_{\rho} ),  \gamma_1 \wedge \gamma_2 \otimes \gamma_t \rangle y_{kh} =  $$ 
%$$ \sum_{s=1}^3 (-1)^{s+1} \langle w^{(2)}_2(\alpha_1(s_1), \alpha_1(s_2), \alpha_1(s_3) \otimes f_h),  \gamma_u \wedge \gamma_t \rangle \cdot  \langle \alpha_1(s_1)^.f_k,  \gamma_s \rangle + $$ $$ +  \langle w^{(2)}_2(\alpha_1(s_1), \alpha_1(s_2), \alpha_1(s_3) \otimes f_k),  \gamma_u \wedge \gamma_t \rangle \cdot  \langle \alpha_1(s_1)^.f_h,  \gamma_s \rangle.  $$
$$  \sum_{i \neq k} (-1)^{\sigma(i,k)}  \langle v^{(2)}_2(\varepsilon_{\hat{i}, \hat{k}}, \otimes f_h),  \gamma_1 \wedge \gamma_2 \rangle \cdot \langle e_i^.f_{\rho},  \gamma_t \rangle + \langle v^{(2)}_2(\varepsilon_{\hat{i}, \hat{k}}, \otimes f_{\rho}),  \gamma_1 \wedge \gamma_2 \rangle \cdot \langle e_i^.f_h,  \gamma_t \rangle.
  $$
\end{lem}

\begin{proof}
For simplicity take $t=1$. First consider the case $k=r=5$. Hence $y_{k\rho}= y_{kh}=0$ and the first term is zero.
If one among $h$ and $\rho$ is larger than 4, also the second term is easily seen to be zero.
Thus assume $h, \rho \leq 4$. By symmetry, it suffices to show that the second term is zero if $h=\rho=4$ and if $h=3, \rho=4$. If $h=\rho=4$ the second term becomes 
$$ -2[\langle v^{(2)}_2(\varepsilon_{234} \otimes f_4),  \gamma_1 \wedge \gamma_2 \rangle \cdot b_{14}^1 -  \langle v^{(2)}_2(\varepsilon_{134} \otimes f_4),  \gamma_1 \wedge \gamma_2 \rangle \cdot b_{24}^1 +  \langle v^{(2)}_2(\varepsilon_{124} \otimes f_4),  \gamma_1 \wedge \gamma_2 \rangle \cdot b_{34}^1= $$   
$$ -2[B^{12}_{24,34} b_{14}^1 - B^{12}_{14,34} b_{24}^1 + B^{12}_{14,24} b_{34}^1]=0,  $$ being the determinant of a $3 \times 3$ matrix with two equal rows.

If $h=3, \rho=4$, the second term becomes
$$ -\langle v^{(2)}_2(\varepsilon_{234} \otimes f_3),  \gamma_1 \wedge \gamma_2 \rangle \cdot b_{14}^1 +  \langle v^{(2)}_2(\varepsilon_{134} \otimes f_3),  \gamma_1 \wedge \gamma_2 \rangle \cdot b_{24}^1 -  \langle v^{(2)}_2(\varepsilon_{124} \otimes f_3),  \gamma_1 \wedge \gamma_2 \rangle \cdot b_{34}^1 +  $$
$$ -\langle v^{(2)}_2(\varepsilon_{234} \otimes f_4),  \gamma_1 \wedge \gamma_2 \rangle \cdot b_{13}^1 +  \langle v^{(2)}_2(\varepsilon_{134} \otimes f_4),  \gamma_1 \wedge \gamma_2 \rangle \cdot b_{23}^1 +  \langle v^{(2)}_2(\varepsilon_{123} \otimes f_4),  \gamma_1 \wedge \gamma_2 \rangle \cdot b_{43}^1 =  $$
%$$ =  -B^{12}_{23,43} b_{14}^1 + B^{12}_{13,43} b_{24}^1  - \frac{1}{2}b_{34}^1[B^{12}_{12,43} - B^{12}_{14,23} + B^{12}_{24,13}] - c^{12}_{1243}+ $$
%$$ - B^{12}_{34,24} b_{13}^1 + B^{12}_{34,14} b_{23}^1 + \frac{1}{2}b_{43}^1[B^{12}_{12,34} - B^{12}_{13,24} + B^{12}_{23,14}] - c^{12}_{1234} = $$
$$ = -B^{12}_{23,43} b_{14}^1 + B^{12}_{13,43} b_{24}^1 - B^{12}_{34,24} b_{13}^1 + B^{12}_{34,14} b_{23}^1  +[B^{12}_{14,23} - B^{12}_{24,13}] b_{34}^1  =0,  $$ using the skew-symmetric property of the indices.

We work now in the case $k \neq 5$, and without loss of generality take $k=1$. Supposing $\rho, h \neq 1$ we get the first term equal to zero.
If both $\rho, h \geq 5$, also the second term is clearly zero. If $2 \leq \rho, h \leq 4$, observe that the maps $v^{(2)}_1, v^{(2)}_1$ in this case are independent from the choice of generators in $F_1$, hence we can conclude that the second term is zero exactly as done for the case $k=5$. 

Without loss of generality assume now $h=4$. If $\rho= 6$, then $ \langle e_i^.f_{\rho}, \gamma_1 \rangle = 0 $ for every $i$ and the second term reduces to be $ \langle v^{(2)}_2(\varepsilon_{245} \otimes f_6),  \gamma_1 \wedge \gamma_2 \rangle \cdot b_{34}^1 -  \langle v^{(2)}_2(\varepsilon_{345} \otimes f_6),  \gamma_1 \wedge \gamma_2 \rangle \cdot b_{24}^1 = b_{24}^1b_{34}^1-b_{34}^1b_{24}^1 =0. $ If $\rho= 5$, using that $ \langle e_5^.f_{5}, \gamma_1 \rangle = 1 $ the second term reduces to
$$ \langle v^{(2)}_2(\varepsilon_{234} \otimes f_4),  \gamma_1 \wedge \gamma_2 \rangle + \langle v^{(2)}_2(\varepsilon_{245} \otimes f_5),  \gamma_1 \wedge \gamma_2 \rangle \cdot b_{34}^1 -  \langle v^{(2)}_2(\varepsilon_{345} \otimes f_5),  \gamma_1 \wedge \gamma_2 \rangle \cdot b_{24}^1 = $$
$$  B^{12}_{34,24} - b_{24}^2b_{34}^1+b_{34}^2b_{24}^1 =0.  $$
Finally, assume $h=1$ and consider the cases $\rho= 1,4,5,6$. For $\rho= 6$, computing similarly as above the second term gives $ b_{34}^1b_{12}^1 - b_{24}^1b_{13}^1 + b_{23}^1b_{14}^1 = P_{1234}^{11}. $ For $\rho= 5$, it gives 
$$  b_{34}^2b_{12}^1 - b_{24}^2b_{13}^1 + b_{23}^2b_{14}^1 - \frac{1}{2}[B^{12}_{23,41}-B^{12}_{24,31}+B^{12}_{34,21} ] + c^{12}_{1234} = -P_{1234}^{12}.  $$ 
For $\rho=h=1$, setting
$$ \Psi_1 := b_{12}^1 [B^{12}_{34,51} - B^{12}_{35,41} + B^{12}_{45,31}] - b_{13}^1 [B^{12}_{24,51} - B^{12}_{25,41} + B^{12}_{45,21}]+$$ $$ + b_{14}^1 [B^{12}_{23,51} - B^{12}_{25,31}  + B^{12}_{35,21}] - b_{15}^1 [B^{12}_{23,41} - B^{12}_{24,31} + B^{12}_{34,21}],     $$
the second term is $ \Psi_1  +b_{12}^1c^{12}_{3451}-b_{13}^1c^{12}_{2451}+b_{14}^1c^{12}_{2351} -b_{15}^1c^{12}_{2341}$.
For $\rho=4$, similarly the second term is equal to
%$$  b_{12}^1 B^{12}_{45,34} - b_{13}^1 B^{12}_{45,24} + b_{15}^1 B^{12}_{34,51}
%+\frac{1}{2}b_{14}^1 [B^{12}_{23,54} - B^{12}_{25,34}  + B^{12}_{35,24}]+ $$ $$ 
%- \frac{1}{2}b_{24}^1 [B^{12}_{34,51} - B^{12}_{35,41} + B^{12}_{45,31}] 
%+ \frac{1}{2}b_{34}^1 [B^{12}_{24,51} - B^{12}_{25,41}  + B^{12}_{45,21}]+$$ $$ 
%- \frac{1}{2}b_{45}^1 [B^{12}_{23,41} - B^{12}_{24,31} + B^{12}_{34,21}] 
%+b_{14}^1c^{12}_{2354}-b_{24}^1c^{12}_{3451}+b_{34}^1c^{12}_{2451} -b_{45}^1c^{12}_{2341}= $$ $$ =
$ \frac{1}{2} \Psi_4 +  b_{14}^1c^{12}_{2354}-b_{24}^1c^{12}_{3451}+b_{34}^1c^{12}_{2451} -b_{45}^1c^{12}_{2341}. $
In each of the above cases, comparing with the computation of $v^{(2)}_3$, relation (W33) is satisfied.
\end{proof}

%\sum_{k=1}^{r_1} \sum_{i=1}^5 (-1)^{i+1} 

% \psi_{12}^h \cdot \langle \alpha_1(s_3)^.f_{\rho},  \gamma_1 \rangle + \psi_{12}^{\rho} \cdot \langle \alpha_1(s_3)^.f_h,  \gamma_1 \rangle - \psi_{13}^h \cdot \langle \alpha_1(s_2)^.f_{\rho},  \gamma_1 \rangle - \psi_{13}^{\rho} \cdot \langle \alpha_1(s_2)^.f_h,  \gamma_1 \rangle + $$  $$ + \psi_{23}^h \cdot \langle \alpha_1(s_1)^.f_{\rho},  \gamma_1 \rangle + \psi_{23}^{\rho} \cdot \langle \alpha_1(s_1)^.f_h,  \gamma_1 \rangle - \psi_{123}^h \cdot \langle e_i^.f_{\rho},  \gamma_1 \rangle + \psi_{123}^{\rho} \cdot \langle e_i^.f_h,  \gamma_1 \rangle.

%$\sum_{k=1}^{r_1}  y_{k\rho}   \langle  w^{(3)}_2(\varepsilon_{i,k,s_1,s_2}), \phi_h \otimes \gamma_1 \rangle $ and in $q^{(2)}_2$ is $$ \delta_{\rho h} \langle \alpha_1(s_1)^.\alpha_1(s_2)^.e_i, \gamma_1 \rangle - \sum_{u=1}^{r_3} z_{hu} \langle w^{(2)}_2(e_i,\alpha_1(s_1),\alpha_1(s_2) \otimes f_{\rho}), \gamma_1 \wedge\gamma_u \rangle.   $$
% $$  \langle \alpha_1(s_1)^.f_{\rho}, \gamma_1 \rangle \cdot \langle \alpha_1(s_2)^.e_i, \phi_h \rangle - \langle \alpha_1(s_2)^.f_{\rho}, \gamma_1 \rangle \cdot \langle \alpha_1(s_1)^.e_i, \phi_h \rangle. $$

\section*{Acknowledgements}
The first and third author are supported by the grants MAESTRO NCN - \\
UMO-2019/34/A/ST1/00263 - Research in Commutative Algebra and
Representation Theory, NAWA POWROTY - PPN/PPO/2018/1/00013/U/00001 - Applications of Lie algebras to Commutative Algebra, and OPUS grant  National Science Centre, Poland grant UMO-2018/29/BST1/01290.
The authors would like to thank Ela Celikbas, Lars Christensen, David Eisenbud, Sara Angela Filippini, Jai Laxmi, Jacinta Torres, and Oana Veliche for interesting conversations about the content of this paper. They also would like to thank Craig Huneke, Claudia Polini, and Bernd Ulrich for pointing out the relation between the conjectures appearing in this paper and the unsolved question posed by Peskine and Szpiro.

% \sum_{i_1,i_2,i_3,i_4} \langle w^{(3)}_5(\varepsilon_{i,j,s_1,s_2,s_3} \otimes e_k),  \wedge_{u=1}^3 \gamma_u \otimes \gamma_u \rangle (\epsilon_{i_1} \wedge \epsilon_{i_2}) \cdot (\epsilon_{i_3} \wedge \epsilon_{i_4}). 

\end{document}